\documentclass{memo-l}
\def\Bbb{\mathbb}
\def\cal{\mathcal}
\usepackage{amsthm,amssymb,amsfonts,mathrsfs}
\usepackage{amsmath}
\catcode`@=11 \@addtoreset{equation}{chapter}

\catcode`@=12
\newtheorem{Theorem}{Theorem}[chapter]
\newtheorem{Proposition}{Proposition}[chapter]
\newtheorem{Lemma}{Lemma}[chapter]
\newtheorem{Corollary}{Corollary}[chapter]
\theoremstyle{definition}

\newtheorem{Definition}{Definition}[chapter]
\newtheorem{Remark}{Remark}

\newtheorem{Example}{Example}[chapter]

\newtheorem{Assumptions}{Hypothesis}[chapter]
\def\R{{\mathbb{R}}}

\def\vp{\varphi}
\def\ve{\varepsilon}
\def\ds{\displaystyle}

\def\fd{\mathfrak{d}}
\def\fg{\mathfrak{g}}
\def\fh{\mathfrak{h}}
\def\cH{\mathcal H}

\title {Carleman estimates, observability inequalities and null controllability for interior degenerate non smooth parabolic equations }
\author{{\sc Genni Fragnelli\thanks{Research supported by the GNAMPA project {\sl Equazioni di evoluzione degeneri e singolari: controllo e applicazioni}}}\\
Dipartimento di Matematica\\ Universit\`{a} di Bari "Aldo Moro"\\
Via
E. Orabona 4\\ 70125 Bari - Italy\\ email: genni.fragnelli@uniba.it\\
{\sc Dimitri Mugnai}\\
Dipartimento di Matematica e Informatica\\Universit\`a di
Perugia\\Via Vanvitelli 1, 06123 Perugia - Italy\\ email:
dimitri.mugnai@unipg.it}
\date{}
\begin{document}

\maketitle

\begin{abstract}
We consider a parabolic problem with degeneracy in the interior of
the spatial domain, and we focus on observability results through
Carleman estimates for the associated adjoint problem. The novelties
of the present paper are two. First, the coefficient of the leading operator only
belongs to a Sobolev space. Second, the degeneracy point is allowed to lie
even in the interior of the control region, so that no previous
result can be adapted to this situation; however, different
cases can be handled, and new controllability results are
established as a consequence.
\end{abstract}

\chapter{Introduction}
In the last recent years an increasing interest has been devoted to
degenerate parabolic equations. Indeed, many problems coming from
physics (boundary layer models in \cite{br}, models of Kolmogorov
type in \cite{bz}, \ldots), biology (Wright-Fisher models in
\cite{s} and Fleming-Viot models in \cite{fv}), and economics
(Black-Merton-Scholes equations in \cite{EG}) are described by
degenerate parabolic equations, whose linear prototype is
\begin{equation}\label{0}
\begin{cases}u_t - {\mathcal A}u  =h(t,x), \quad
(t,x) \in (0,T) \times(0,1),\\
u(0,x)=u_0(x)
\end{cases}
\end{equation}
with the associated desired boundary conditions, where ${\mathcal A}u= {\mathcal A}_1u:= (au_x)_x$ or ${\mathcal A}u = {\mathcal A}_2u:= au_{xx}$.

In this paper we concentrate on a special topic related to this
field of research, i.e. Carleman estimates for the adjoint problem
to \eqref{0}. Indeed, they have so many applications that a large
number of papers has been devoted to prove some forms of them and
possibly some applications. For example, it is well known that they
are a fundamental tool to prove observability inequalities, which
lead to global null controllability results for \eqref{0} also in
the non degenerate case: for all $T>0$ and for all initial data
$u_0\in L^2((0,T) \times(0,1))$ there is a suitable control $h\in
L^2((0,T) \times(0,1))$, supported in a subset $\omega$ of $[0,1]$,
such that the solution $u$ of \eqref{0} satisfies $ u(T, x)=0 $ for
all $x \in [0,1]$ (see, for instance, \cite{m0} - \cite{bcg}, \cite{br} - \cite{cfv1},
\cite{fcgb}, \cite{fz}, \cite{fdt}, \cite{f}, \cite{fm}, \cite{LRL}, \cite{lrr},
\cite{lr},
 \cite{Ra} and the references therein).

Moreover, Carleman estimates are also extremely useful for
several other applications, especially for unique continuation
properties (for example, see \cite{esca}, \cite{kt} and \cite{LRL}), for inverse
problems, in parabolic, hyperbolic and fractional settings, e.g. see
\cite{by}, \cite{chue}, \cite{ima}, \cite{salo}, \cite{salo1}, \cite{Wu},
\cite{Ya} and their references.

The common point of all the previous papers dealing with degenerate
equations, is that the function $a$ degenerates at the boundary of
the domain. For example, as $a$, one can consider the double power
function
\[a(x)= x^k(1-x)^\alpha, \quad x \,\in \,[0,1],\] where $k$ and $\alpha$ are positive constants.
For related systems of degenerate equations we refer to \cite{m0},
\cite{m} and \cite{cdt}.

However, the papers cited above deal with a function $a$ that
degenerates at the boundary of the spatial domain. To our best
knowledge, \cite{s1} is the first paper treating the existence of a
solution for the Cauchy problem associated to a parabolic equation
which degenerates in the interior of the spatial domain, while
degenerate parabolic problems modelling biological phenomena and
related optimal control problems are later studied in \cite{ly}
and \cite{bel}. Recently, in \cite{fggr} the authors analyze in
detail the degenerate operator $ {\cal A}$ in the space $L^2 (0,1)$,
with or without weight, proving that it is nonpositive and
selfadjoint, hence it generates a cosine family and, as a
consequence, an analytic semigroup. In \cite{fggr} the
well-posedness of \eqref{0} with Dirichlet boundary conditions is
also treated, but nothing is said about other properties, like
Carleman estimates or controllability results. Indeed, these
arguments are the subject of the recent paper \cite{fm}, where only
the divergence case is considered and the function $a$ is assumed to
be of class $C^1$ far from the degenerate point, which belongs to
the interior of the spatial domain.

In this paper we consider both problems in divergence and in non divergence form with a non smooth coefficient (for the precise assumptions see below), and we
first prove Carleman estimates for the adjoint problem of the
parabolic equation with interior degeneracy
\begin{equation}\label{linear}
\begin{cases}
u_t - \mathcal A u  =h,\quad
(t,x) \in Q_T:=(0,T) \times(0,1),\\
u(t,0)=u(t,1)=0, \\
u(0,x)=u_0(x),
\end{cases}
\end{equation}
that is for solutions of the problem
\begin{equation}\label{01}
\begin{cases}
v_t + \mathcal Av =h, & (t,x) \in Q_T,\\
v(t,1)=v(t,0)=0, &  t \in (0,T).
\end{cases}
\end{equation}
Here $u_0$ belongs to a suitable Hilbert space $X$  ($L^2(0,1)$ in the divergence case and $L^2_{\frac{1}{a}}(0,1)$ in the non divergence case, see the following chapters), and the control $h \in L^2(0,T;X)$
 acts on a nonempty subinterval $\omega$
of $(0,1)$ which is allowed to contain the degenerate point $x_0$.

We underline the fact that in the present paper we consider both
equations {\em in divergence} and {\em in non divergence form}, since the last one {\it cannot} be recast from the equation in
divergence form, in general: for example, the simple equation
\[
u_t=a(x)u_{xx}
\]
can be written in divergence form as
\[
u_t=(au_x)_x-a'u_x,
\]
only if $a'$ does exist; in addition, even if $a'$ exists, considering the well-posedness for the last equation, additional conditions are
necessary: for instance, for the prototype $a(x)= x^K$,
well-posedness is guaranteed if $K \ge 2$ (\cite{mv}). However, in
\cite{cfr} the authors prove that if $a(x)=x^K$ the global null
controllability fails exactly when $K\ge 2$. For this reason, already in
\cite{cfv}, \cite{cfv1} and \cite{f} the authors consider parabolic
problems in non divergence form proving directly that, under suitable
conditions for which well-posedness holds, the problem is still
globally null controllable, that is the solution vanishes
identically at the final time by applying a suitable localized
control. In particular, while in \cite{cfv} or \cite{cfv1} Dirichlet
boundary conditions are considered, in \cite{f} Neumann boundary
conditions are assumed.

\medskip

The question of controllability of partial differential systems with
{\em non smooth} coefficients, i.e. the coefficient $a$ is not of
class $C^1$ (or even with higher regularity, as sometimes it is
required), and its dual counterpart, observability inequalities, is
not fully solved yet. In fact, the presence of a non smooth
coefficient introduces several complications, and, in fact, the
literature in this context is quite poor. We are only aware of the
following few papers in which Carleman estimate are proved always in
the {\em non degenerate case}, but in the case in which the
coefficient of the operator is somehow {\em non smooth}. In
\cite{doubova} and \cite{lrr} the non degenerate coefficient is
actually assumed smooth apart from across an interface where it may
jump, with or without some monotonicity condition (\cite{doubova}
and \cite{lrr}, respectively), while in \cite{bena} the non
degenerate coefficient is assumed to be piecewise smooth. Carleman
estimates for a non degenerate $BV$ coefficient were proved in
\cite{LRBV}, but however, the coefficient was supposed to be of
class $C^1$ in an open subset of $(0,1)$, and then controllability
for \eqref{linear} and semilinear extension are given. Finally, in
\cite{iyinfty} $a$ is supposed to be of class $W^{1,\infty}(0,1)$,
but again it does not degenerate at any point. For completeness, we
also quote \cite{FCZBV}, where boundary controllability result for
non degenerate $BV$ coefficients are proved using Russel's method
and not Carleman estimates.

\medskip

As far as we know, no Carleman estimates for \eqref{01} are known
when $a$ is {\em globally non smooth} and degenerates at an interior
point $x_0$, nor when $a$ is non degenerate and non smooth. For this
reason, the object of this paper is twofold: first, we prove
Carleman estimates in the non degenerate case when $a$ is not
smooth. In particular we treat the case of an absolutely continuous
coefficient, and thus {\em globally} of class $BV$, though with some
restrictions, and the case of a $W^{1,\infty}$ coefficient. This
case was already considered in \cite{iyinfty}, but they proved a
version of Carleman estimates with all positive integrals in the
righ-hand-side, while in our version we include a negative one,
which is needed for the subsequent applications (see Theorem
\ref{mono} and Theorem \ref{mono1}). Second, we prove Carleman
estimates in degenerate non smooth cases. Such estimates are then
used to prove observability inequalities (and hence null
controllability results).

To our best knowledge, this paper is the first one where, in the
case of an absolutely continuous coefficient - which is even allowed
to degenerate - non smoothness is assumed in the whole domain,
though with some restrictions.

\medskip

Concerning the non smooth non degenerate case, in the spatial domain
$(0,1)$ we assume that
\begin{itemize}
\item[$(a_1)$] $a\in W^{1,1}(0,1)$, $a\geq
a_0>0$ in $(0,1)$ and there exist two functions $\fg \in L^1(0,1)$,
$\fh \in W^{1,\infty}(0,1)$ and two strictly positive constants
$\fg_0$, $\fh_0$ such that $\fg(x) \ge \fg_0$ for a.e. $x$ in $[0,1]$ and
\[
-\frac{a'(x)}{2\sqrt{a(x)}}\left(\int_x^1\fg(t) dt + \fh_0 \right)+ \sqrt{a(x)}\fg(x) =\fh(x)\quad \text{for a.e.} \; x \in [0,1],
 \]
 in the divergence case,
 \[
\frac{a'(x)}{2\sqrt{a(x)}}\left(\int_x^1\fg(t) dt + \fh_0 \right)+ \sqrt{a(x)}\fg(x) =\fh(x)\quad \text{for a.e.} \; x \in [0,1],
\]
in the non divergence case; or\\
\item[$(a_2)$] $a\in W^{1,\infty}(0,1)$ and $a\geq
a_0>0$ in $(0,1)$.
\end{itemize}
However, in Chapter \ref{seccarnondeg} we shall present the precise
setting and the related Carleman estimate in a general interval
$(A,B)$, since we shall not use it in the whole $(0,1)$ but in
suitable subintervals.

Concerning the degenerate case, we shall admit two types of
degeneracy for $a$, namely weak and strong degeneracy. More
precisely, we shall handle the two following cases:
\begin{Assumptions}\label{Ass0}
{\bf Weakly degenerate case (WD):} there exists $x_0 \in (0,1)$ such
that $a(x_0)=0$, $a>0$ on $[0, 1]\setminus \{x_0\}$, $a\in
W^{1,1}(0,1)$ and there exists $K \in (0,1)$ such that $(x-x_0)a'
\le K a$ a.e. in $[0,1]$.
\end{Assumptions}

\begin{Assumptions}\label{Ass01}
{\bf Strongly degenerate case (SD):} there exists $x_0 \in (0,1)$
such that $a(x_0)=0$, $a>0$ on $[0, 1]\setminus \{x_0\}$, $a\in
W^{1, \infty}(0,1)$ and there exists $K \in [1,2)$ such that
$(x-x_0)a' \le K a$ a.e. in $[0,1]$.
\end{Assumptions}
Typical examples for weak and strong degeneracies are $a(x)=|x-
x_0|^{\alpha}, \; 0<\alpha<1$ and $a(x)= |x- x_0|^{\alpha}, \;
1\le\alpha<2$, respectively.

For the proof of the related Carleman estimates and observability
inequalities a fundamental r\^{o}le is played by the following
general weighted Hardy-Poincar\'e inequality for functions which may
{\em not be} globally absolutely continuous in the domain, but whose
irregularity point is ``controlled'' by the fact that the weight
degenerates exactly there. Such an inequality, of independent
interest, was proved in \cite[Proposition 2.3]{fm}, and reads as
follows.
\begin{Proposition}[Hardy--Poincar\'{e} inequality]\label{HP}
Assume that $p \in C([0,1])$, $p>0$ on $[0,1]\setminus \{x_0\}$,
$p(x_0)=0$ and there exists $q \in (1,2)$ such that the function
$$
\begin{aligned}
x \mapsto \dfrac{p(x)}{|x-x_0|^{q}} &\mbox { is
nonincreasing on the left of } x=x_0 \\
& \mbox{ and nondecreasing on the right of } x=x_0.
\end{aligned}
$$
Then, there exists a constant $C_{HP}>0$ such that for any function
$w$, locally absolutely continuous on $[0,x_0)\cup (x_0,1]$ and
satisfying
$$
w(0)=w(1)=0 \,\, \mbox{and } \int_0^1 p(x)|w^{\prime}(x)|^2 \,dx <
+\infty \,,
$$ the following inequality holds:
\begin{equation}\label{hardy1}
\int_0^1 \dfrac{p(x)}{(x-x_0)^2}w^2(x)\, dx \leq C_{HP}\, \int_0^1
p(x) |w^{\prime}(x)|^2 \,dx.
\end{equation}
\end{Proposition}
Actually, such a proposition is valid without requiring $q<2$.

Applying the Carleman estimate (and other tools) to any solution $v$
of the adjoint problem \eqref{01}, we derive the observability
inequalities
\[
\int_0^1v^2(0,x) dx \le C\int_0^T\int_{\omega} v^2(t,x)dxdt,
\]
in the divergence case and
\[
\int_0^1v^2(0,x)\frac{1}{a} dx \le C\int_0^T\int_{\omega} v^2(t,x)\frac{1}{a}dxdt,
\]
in the non divergence one. The proof of these last inequalities are
obtained by studying some auxiliary problems, introduced with
suitable cut-off functions and reflections (see Lemmas \ref{lemma3},
\ref{lemma3'} and \ref{lemma31}), and is the content of the long
Chapter \ref{secobserv}, where, using a standard technique in this
framework, one can also prove null controllability results for
\eqref{linear}.

Finally, such results are extended to the semilinear problem
\begin{equation}
    \label{nlo}
    \begin{cases}
    u_t - \mathcal Au  + f(t,x,u)  =h(t,x) \chi_{\omega}(x), & \ (t,x) \in (0,T) \times (0,1), \\
    u(t,1)=u(t,0)=0, & \  t \in (0,T),\\
     u(0,x)=u_0(x) , & \  x \in (0,1),
    \end{cases}
\end{equation}
in the weakly degenerate case using the fixed point method developed
in \cite{fz} for nondegenerate problems. We note that, as in the
nondegenerate case, our method relies on a compactness result for
which the fact that $\displaystyle\frac{1}{a} \in L^1(0,1)$ is an
essential assumption, and it forces us to consider only the weakly
degenerate case. However, in the complete linear case, i.e
$f(t,x,u)= c(t,x)u(t,x)$, the null controllability result holds
also for the strongly degenerate case, since in this case it is a
consequence of the results proved for \eqref{linear}, see Corollary
\ref{cor_c}.

\medskip

The paper is organized as follows. First of all, we underline the
fact that all chapters, except for the final Chapters
\ref{secsemilinear} and \ref{sec7}, are divided into two subsections
that deal with the divergence case and the non divergence one
separately. In Chapter \ref{sec2} we give the precise setting for
the weakly and the strongly degenerate cases and some general tools
we shall use several times. In Chapter \ref{seccarnondeg} we prove
Carleman estimates for the adjoint problem of \eqref{linear} with a
non smooth non degenerate coefficient. In Chapter \ref{Carleman
estimate} we provided one of the main results of this paper, i.e.
Carleman estimates in the degenerate (non smooth) case. In Chapter
\ref{secobserv} we apply the previous Carleman estimates to prove
observability inequalities which, together with Caccioppoli type
inequalities, let us derive new null controllability results for
degenerate problems. In particular, in the divergence case, we
handle both the cases in which the degeneracy point is {\em inside}
or {\em outside} the control region $\omega$; on the contrary, in
the non divergence case we consider only the case of a degeneracy
point being {\em outside} $\omega$ (see Comment 2 in Chapter
\ref{sec7} for the reason of this fact). In Chapter \ref{secsemilinear} we extend the previous results to complete
linear and semilinear problems. Finally, in Chapter \ref{sec7} we
conclude the paper with some general remarks, which we consider
fundamental.

\chapter{Mathematical tools and preliminary results}\label{sec2}
We begin this chapter with a lemma that is crucial for the rest of the paper:
\begin{Lemma}[\cite{fm}, Lemma 2.1] \label{rem}
Assume that Hypothesis $\ref{Ass0}$ or $\ref{Ass01}$ is satisfied.
\begin{enumerate}
\item Then for all $\gamma \ge K$ the map
$$
\begin{aligned}
& x \mapsto \dfrac{|x-x_0|^\gamma}{a} \mbox { is nonincreasing on
the left of } x=x_0 \\
& \mbox{and nondecreasing on the right of }
x=x_0,\\
&\mbox{ so that }\lim_{x\to x_0}\dfrac{|x-x_0|^\gamma}{a}=0 \mbox{
for all }\gamma>K.
\end{aligned}
$$
\item If $K<1$, then
    $\displaystyle\frac{1}{a} \in L^{1}(0,1)$.\\
\item If $K \in[1,2)$, then $\displaystyle \frac{1}{\sqrt{a}} \in
    L^{1}(0,1)$ and $\displaystyle \frac{1}{a}\not \in L^1(0,1)$.
\end{enumerate}
\end{Lemma}

\begin{Remark}\label{nol1}
We underline the fact that if $\displaystyle \frac{1}{a} \in
L^{1}(0,1)$, then $\displaystyle \frac{1}{\sqrt{a}} \in L^{1}(0,1)$.
On the contrary, if $a\in W^{1,\infty}([0,1])$ and $\displaystyle
\frac{1}{\sqrt{a}} \in L^{1}(0,1)$, then $\displaystyle \frac{1}{a}
\not \in L^1(0,1)$ (see \cite[Remark 2]{fm}).
\end{Remark}

\section{Well-posedness in the divergence case}

In order to study the well-posedness of problem \eqref{linear}, we
introduce the operator
\[
{\cal A}_1u:=(au_x)_x
\]
and we consider two different classes of weighted Hilbert spaces,
which are suitable to study two different situations, namely the
{\em weakly degenerate} (WD) and the {\em strongly degenerate} (SD)
cases. We remark that we shall use the standard notation $H$ for
Sobolev spaces with non degenerate weights and the calligraphic
notation ${\cal H}$ for spaces with degenerate weights.

{\bf CASE (WD):} if Hypothesis \ref{Ass0} holds, we consider
\[
\begin{aligned}
{\cal H}^1_a(0,1):=\big\{&u \text{ is absolutely continuous in }
[0,1],
\\ & \sqrt{a} u' \in  L^2(0,1) \text{ and } u(0)=u(1)=0
\big\},
\end{aligned}
\]
and
\[
\label{Ha2} \qquad {\cal H}^2_a(0,1) :=  \big\{ u \in {\cal
H}^1_a(0,1) |\,au' \in H^1(0,1)\big\};
\]

{\bf CASE (SD):} if Hypothesis \ref{Ass01} holds, we consider
\[
\begin{aligned}
{\cal H}^1_a(0,1):=\big\{ u \in L^2(0,1) \ \mid  \,&u \text{ locally
absolutely continuous in } [0,x_0) \cup (x_0,1], \\ & \sqrt{a} u'
\in L^2(0,1) \text{ and } u(0)= u(1)=0 \big\}
\end{aligned}
\]
and
\[
\label{Ha2S} \qquad {\cal H}^2_a(0,1) :=  \big\{ u \in {\cal
H}^1_a(0,1) |\,au' \in H^1(0,1)\big\}.
\]
In both cases we consider the norms
\[
\|u\|^2_{{\cal H}^1_a(0,1)}:= \|u\|^2_{L^2(0,1)} +
\|\sqrt{a}u'\|^2_{L^2(0,1)},\] and
\[\|u\|^2_{{\cal H}^2_a(0,1)} :=
\|u\|^2_{{\cal H}^1_a(0,1)} + \|(au')'\|^2_{L^2(0,1)}
\]
and we set
\[
D({\cal A}_1)={\cal H}^2_a(0,1).
\]

Thanks to lemma \ref{rem} one can prove the following
characterizations for the (SD) case which are already given in
\cite[Propositions 2.1 and 2.2]{fm}.
\begin{Proposition}[\cite{fm}, Proposition 2.1]\label{characterization}
Let
\[
\begin{aligned} X:=\Big\{ u \in L^2(0,1)\,| \ &u \text{ locally
absolutely continuous in } [0,1]\setminus \{x_0\},  \\ &\sqrt{a}u'
\in L^2(0,1), au \in H^1_0(0,1) \;\text{and} \\&
(au)(x_0)=u(0)=u(1)=0\Big\}.
\end{aligned}
\]
Then, under Hypothesis $\ref{Ass01}$ we have
\[
{\cal H}^1_a(0,1)=X.
\]
\end{Proposition}

Using the previous result, one can prove the following additional
characterization.
\begin{Proposition}[\cite{fm}, Proposition 2.2]\label{domain}
Let
\[
\begin{aligned}
D:=\Big\{ u \in L^2(0,1)\,| \ & \ u \text{ locally absolutely
continuous in } [0,1]\setminus \{x_0\}, \\ & au \in H^1_0(0,1), a u'
\in H^1(0,1),  au \text{ is continuous at } x_0 \;\text{and}\\&
(au)(x_0)=(au')(x_0)=u(0)=u(1)=0\Big\}.
\end{aligned}
\]
Then, under Hypothesis $\ref{Ass01}$ we have
\[
{\cal H}^2_a(0,1)=D({\cal A}_1)=D.
\]
\end{Proposition}
 Now, let us go back to problem
\eqref{linear}, recalling the following
\begin{Definition}
If $u_0 \in L^2(0,1)$ and $h\in L^2(Q_T):= L^2(0,T; L^2(0,1))$, a function $u$ is said to
be a (weak) solution of \eqref{linear} if
\[
u \in C([0, T]; L^2(0,1)) \cap L^2(0, T; {\cal H}^1_a(0,1))
\]
and
\[
\begin{aligned}
&\int_0^1u(T,x)\varphi(T,x)\, dx - \int_0^1 u_0(x) \varphi(0,x)\, dx
- \int_{Q_T}u\varphi_t \,dxdt =
\\&- \int_{Q_T} au_x
\varphi_x\,dxdt + \int_{Q_T} h\varphi \chi_\omega\,dx dt
\end{aligned}
\]
for all $\varphi \in H^1(0, T; L^2(0,1)) \cap L^2(0, T; {\cal
H}^1_a(0,1))$.
\end{Definition}
As proved in \cite{fggr} (see Theorems $2.2$, $2.7$ and $4.1$),
problem \eqref{linear} is well-posed in the sense of the following
theorem:
\begin{Theorem}\label{th-parabolic} Assume Hypothesis $\ref{Ass0}$ or $\ref{Ass01}$. For
all $h \in  L^2(Q_T)$ and $u_0 \in L^2(0,1)$, there exists a unique
weak solution  $u \in C([0,T]; L^2(0,1)) \cap L^2 (0,T; {\cal
H}^1_a(0,1))$ of \eqref{linear} and there exists a universal
positive constant $C$ such that
\begin{equation}\label{stima}
\sup_{t \in [0,T]} \|u(t)\|^2_{L^2(0,1)}+\int_0^T\|u(t)\|^2_{{\cal
H}^1_a(0,1)} dt \le C(\|u_0\|^2_{L^2(0,1)}+\|h\|^2_{L^2(Q_T)}).
\end{equation}
Moreover, if $u_0 \in {\cal H}^1_a(0,1)$, then
\begin{equation}\label{regularity}
u \in H^1(0,T; L^2(0,1))\cap C([0,T]; {\cal H}^1_a(0,1)) \cap
L^2(0,T; {\cal H}^2_a(0,1)),
\end{equation}
 and there exists a universal positive constant $C$ such
that
\begin{equation}\label{stima1}
\begin{aligned}
\sup_{t \in [0,T]}\left(\|u(t)\|^2_{{\cal H}^1_a(0,1)} \right)&+
\int_0^{T} \left(\left\|u_t\right\|^2_{L^2(0,1)} +
\left\|(au_x)_x\right\|^2_{L^2(0,1)}\right)dt\\&\le C
\left(\|u_0\|^2_{{\cal H}^1_a(0,1)} + \|h\|^2_{L^2(Q_T)}\right).
\end{aligned}
\end{equation}
In addition, $\mathcal A_1$ generates an analytic contraction
semigroup of angle $\pi/2$.
\end{Theorem}

\section{Well-posedness in the non divergence case}
We start proving some preliminary results concerning the
well-posedness of problem \eqref{linear} in the non divergence case. For this, we consider
the operator
\[
\mathcal A_2u:=au_{xx},
\]
which is related to the following weighted Hilbert spaces:
\[L^2_{\frac{1}{a}}(0,1) :=\left\{ u \in L^2(0,1) \
\mid \int_0^1 \frac{u^2}{a} dx <\infty \right\},
\]
\[
\cH^1_{\frac{1}{a}}(0,1) :=L^2_{\frac{1}{a}}(0,1)\cap H^1_0(0,1),
\]
and
\[
\cH^2_{\frac{1}{a}}(0,1) :=\Big\{ u \in \cH^1_{\frac{1}{a}}(0,1) \,
\big| \, u'\in H^1(0,1)\Big\},
\]
endowed with the associated norms
\[ \|u\|_{L^2_{\frac{1}{a}}(0,1)}^2:= \int_0^1
\frac{u^2}{a} dx, \quad \forall\, u\in L^2_{\frac{1}{a}}(0,1),\]
\[
\|u\|^2_{\cH^1_{\frac{1}{a}}}:=\|u\|_{L^2_{\frac{1}{a}}(0,1)}^2 +
\|u'\|^2_{L^2(0,1)}, \quad \forall\, u\in
\cH^1_{\frac{1}{a}}(0,1),\]
 and
\[
\|u\|_{\cH^2_{\frac{1}{a}}(0,1)}^2
:=\|u\|_{\cH^1_{\frac{1}{a}}(0,1)}^2 +
\|au''\|^2_{L^2_{\frac{1}{a}}(0,1)},\quad \forall\,u\in
\cH^2_{\frac{1}{a}}(0,1).
\]
Indeed, it is a trivial fact that, if $u'\in H^1(0,1)$, then $au''
\in L^2_{\frac{1}{a}}(0,1)$, so that the norm for
$\cH^2_{\frac{1}{a}}(0,1)$ is well defined, and we can also write in
a more appealing way
\[
\cH^2_{\frac{1}{a}}(0,1) :=\Big\{ u \in \cH^1_{\frac{1}{a}}(0,1) \,
\big| \, u'\in H^1(0,1) \mbox{ and }au'' \in
L^2_{\frac{1}{a}}(0,1)\Big\}.
\]
Finally, we take
\[
D(\mathcal A_2)=\cH^2_{\frac{1}{a}}(0,1).
\]

Using Lemma \ref{rem}, also the following characterization in the
(WD) case is straightforward:
\begin{Proposition}[\cite{fggr}, Corollary 3.1]\label{h=h}
Assume Hypothesis $\ref{Ass0}$. Then,
$\cH^1_{\frac{1}{a}}(0,1)$ and $H^1_0(0,1)$ coincide algebraically.
Moreover the two norms are equivalent. As a consequence,
$\cH^2_{\frac{1}{a}}(0,1)=H^2(0,1)\cap H^1_0(0,1)$.
\end{Proposition}

Hence, in the (WD) case, $C^\infty_c(0, 1)$ is dense in
$\cH^1_{\frac{1}{a}}(0, 1)$. 

We also have the following characterization for the (SD) case:
\begin{Proposition}[\cite{fggr}, Propositions 3.6]\label{domain3}
Suppose that Hypothesis $\ref{Ass01}$ holds and set
\[
\begin{aligned} X:=\{ u \in \cH^1_{\frac{1}{a}}(0,1) \ \mid u(x_0)=0
\}.
\end{aligned}
\]
Then
\[
\cH^1_{\frac{1}{a}}(0,1)=X,
\]
and, for all $u\in X$, $\|u\|_{\cH^1_{\frac{1}{a}(0,1)}}$ is
equivalent to $\left(\int_0^1(u')^2dx\right)^{\frac{1}{2}}$.
\end{Proposition}
We remark that \cite[Propositions 3.6]{fggr} was proved assuming
that
\begin{center}
{\it "there exists $x_0 \in (0,1)$ such that $a(x_0)=0$, $a>0$ on
$[0, 1]\setminus \{x_0\}$, $a\in W^{1, \infty}(0,1)$,
$\displaystyle\frac{1}{a} \not \in L^1(0,1)$ and there exists $C>0$
such that $ \frac{1}{a(x)} \le \frac{C}{|x-x_0|^2},$ for all $x \in
[0,1]\setminus\{x_0\}$"}.
\end{center}
However, the last assumption is clearly satisfied under Hypothesis
\ref{Ass01}, thanks to Lemma \ref{rem}.1, Lemma \ref{rem}.2 and
Remark \ref{nol1}.

We shall also need the following characterization:
\begin{Proposition}\label{domain1}
Suppose that Hypothesis $\ref{Ass01}$ holds and set
\[
D:=\{ u \in \cH^2_{\frac{1}{a}}(0,1)\ \mid  \ au' \in H^1(0,1)
\text{ and }u(x_0)=(au')(x_0)=0\}.
\]
Then $D(\mathcal A_2)=\cH^2_{\frac{1}{a}}(0,1)=D$.
\end{Proposition}
\begin{proof}
Since it is clear that $D \subseteq
\cH^2_{\frac{1}{a}}(0,1)$, we take $u \in
\cH^2_{\frac{1}{a}}(0,1)$ and we prove that $u \in D$.

By Proposition \ref{domain3}, $u(x_0)=0$, so that it is sufficient
to prove that $au' \in H^1(0,1)$ and $(au')(x_0)=0$. Since $u'\in
H^1(0,1)$ and $a \in W^{1, \infty}(0,1)$, we immediately have that
$au'\in L^2(0,1)$. Moreover, $(au')'=a'u'+ au'' \in L^2(0,1)$ since
$au''\in L^2_{1/a}(0,1) \subset L^2(0,1)$, and thus $au' \in
H^1(0,1)\subset C([0,1])$. Thus there exists $\lim_{x \to
x_0}(au')(x) = (au')(x_0)=L \in \R$. Assume by contradiction that $L
\neq 0$, then there exists $c>0$ such that
\[
|(au')(x)| \ge c
\]
for all $x$ in a neighborhood of $x_0$. Thus
\[
|\left(a(u')^{2}\right)(x)| \ge \frac{c^2}{a(x)},
\]
for all $x$ in a neighborhood of $x_0$, $ x \neq x_0$. But
$\displaystyle\frac{1}{a} \not \in L^1(0,1)$, thus we would have $\sqrt{a}u' \not
\in L^2(0,1)$, while $\sqrt{a}u' \in L^2(0,1)$, since $a$ is bounded
and $u'\in L^2(0,1)$. Hence $L=0$, that is $(au')(x_0)=0$.
\end{proof}

For the rest of the paper, a crucial tool is also the following Green
formula:
\begin{Lemma}\label{green}
For all $(u,v)\in \cH^2_{\frac{1}{a}}(0,1)\times
\cH^1_{\frac{1}{a}}(0,1)$ one has
\begin{equation}\label{greenformula}
\int_0^1u'' v\, dx= - \int_0^1 u'v'\, dx.
\end{equation}
\end{Lemma}
\begin{proof}
It is trivial, since $u'\in H^1(0,1)$ and $v\in H^1_0(0,1)$.
\end{proof}

Finally, we will use the following
\begin{Lemma}[\cite{fggr}, Lemma 3.7]\label{hardy}
Assume Hypothesis $\ref{Ass01}$. Then, there exists a positive
constant $C$ such that
\[
\int_0^1 v^2\frac{1}{a} dx \le C\int_0^1 (v')^2dx
\]
for all $v \in \cH^1_{\frac{1}{a}}(0,1)$.
\end{Lemma}

We also recall the following definition:
\begin{Definition}
Assume that $u_0 \in L^2_{\frac{1}{a}}(0,1)$ and $h\in
L^2_{\frac{1}{a}}(Q_T):= L^2(0,T; L^2_{\frac{1}{a}}(0,1))$. A function $u$ is said to be a (weak)
solution of \eqref{linear} if
\[u \in C([0, T]; L^2_{\frac{1}{a}}(0,1)) \cap L^2(0, T;
\cH^1_{\frac{1}{a}}(0,1))\] and satisfies
\[
\begin{aligned}
&\int_0^1 \frac{ u(T,x)\varphi(T,x)}{a(x)} dx - \int_0^1
\frac{u_0(x) \varphi(0,x)}{a(x)} dx -\int_{Q_T}
\frac{\varphi_t (t,x)u(t,x)}{a(x)}dxdt =
\\&- \int_{Q_T} u_x(t,x)
\varphi_x(t,x) dxdt + \int_{Q_T} h(t,x) \frac{\varphi(t,x)
}{a(x)}dx dt
\end{aligned}
\]
for all $\varphi \in H^1(0, T; L^2_{\frac{1}{a}}(0,1)) \cap L^2(0,
T; \cH^1_{\frac{1}{a}}(0,1))$.
\end{Definition}

Problem \eqref{linear} is well-posed in the sense of the following
theorem:
\begin{Theorem}\label{theorem_nondivergence}
Assume Hypothesis $\ref{Ass0}$ or $\ref{Ass01}$. Then, the operator
$\mathcal A_2:D(\mathcal A_2)\to L^2_{\frac{1}{a}}(0, 1)$ is self--adjoint, nonpositive on
$L^2_{\frac{1}{a}}(0,1)$ and it generates an analytic contraction
semigroup of angle $\pi/2$. Moreover, for all $h \in
L^2_{{\frac{1}{a}}}(Q_T)$ and $u_0 \in L^2_{{\frac{1}{a}}}(0,1)$,
there exists a unique solution $u \in C\big([0,T];
L^2_{\frac{1}{a}}(0,1)\big) \cap L^2 \big(0,T;
\cH^1_{\frac{1}{a}}(0,1)\big)$ of \eqref{linear} such that
\begin{equation}\label{stima2}
\sup_{t \in [0,T]}
\|u(t)\|^2_{L^2_{\frac{1}{a}}(0,1)}+\int_0^T\|u(t)\|^2_{\cH^1_{\frac{1}{a}}
(0,1)} dt \le
C_T\left(\|u_0\|^2_{L^2_{\frac{1}{a}}(0,1)}+\|h\|^2_{L^2_{\frac{1}{a}}(Q_T)}\right),
\end{equation}
for some positive constant $C_T$. In addition, if $h \in W^{1,1}(0,T; L^2_{\frac{1}{a}}(0,1))$ and $u_0 \in
\cH^1_{\frac{1}{a}}(0,1)$, then
\begin{equation}\label{regularity1}
u\in C^1\big([0,T]; L^2_{\frac{1}{a}}(0,1)\big) \cap C\big([0,T];
D(\mathcal A_2)\big),
\end{equation}
and there exists a positive constant $C$ such that
\begin{equation}\label{stima3}
\begin{aligned}
\sup_{t \in [0,T]}\left(\|u(t)\|^2_{\cH^1_{\frac{1}{a}}(0,1)}
\right)&+ \int_0^{T}
\left(\left\|u_t\right\|^2_{L^2_{\frac{1}{a}}(0,1)} +
\left\|au_{xx}\right\|^2_{L^2_{\frac{1}{a}}(0,1)}\right)dt\\&\le C
\left(\|u_0\|^2_{\cH^1_{\frac{1}{a}}(0,1)} +
\|h\|^2_{L^2(Q_T)}\right).
\end{aligned}\end{equation}
\end{Theorem}
\begin{proof}
In the (WD) case the existence part is proved in \cite[Theorems 3.3 and 4.3]{fggr}.
For the (SD) case, under different assumptions on the domain of the
operator $\mathcal A_2u:=au_{xx}$, it was proved in \cite[Theorems
3.4 and 4.3]{fggr}, but here we must prove the theorem again, since
the domain of $\mathcal A_2$ is different. 

First, $D(\mathcal A_2)$ being dense in $L^2_{\frac{1}{a}}(0,1)$, in order to
show that $\mathcal A_2$ generates an analytic semigroup, it is sufficient to
prove that $\mathcal A_2$ is nonpositive and self-adjoint, hence $m-$dissipative by \cite[Corollary 2.4.8]{ch}.

Thus: $\mathcal A_2$ is nonpositive, since by \eqref{greenformula}, it follows
that, for any $u\in D(\mathcal A_2)$
\[
\langle \mathcal A_2u, u \rangle_{L^2_{\frac{1}{a}}(0,1)} = \int_0^1 u'' u\, dx
= -\int_0^1 (u')^2 dx \le 0.
\]
Let us show that $ \mathcal A_2$ is self-adjoint. First of all,
observe that $\cH^1_{\frac{1}{a}}(0,1)$ is equipped with the natural
inner product
\[
(u, v)_1:= \int_0^1 \left(\frac{u v}{a} + u'v'\right) dx
\]
for any $u, v \in \cH^1_{\frac{1}{a}}(0,1)$ and thanks to Lemma \ref{hardy}, the norm  $\sqrt{(u, u)_1}$ is equivalent to $\|u'\|_{L^2(0,1)},$ for all $u \in \cH^1_{\frac{1}{a}}(0,1).$

Now, consider the function $F:
L^2_{\frac{1}{a}}(0,1) \to L^2_{\frac{1}{a}}(0,1)$ defined as $ F(f) :=
u \in \cH^1_{\frac{1}{a}}(0,1)$ where $u$ is the unique solution of
\[
\int_0^1 u'v' dx = \int_0^1 \frac{f}{a} v dx
\]
for all $u \in \cH^1_{\frac{1}{a}}(0,1)$. Note that $F$ is well-defined by the
Lax--Milgram Theorem, which also implies that $F$ is continuous. Now, easy calculations show that $F$ is symmetric and injective. Hence, $F$ is self-adjoint. As a consequence, $\mathcal A_2= F^{-1}: D(\mathcal A_2) \rightarrow L^2_{\frac{1}{a}}(0,1)$ is self-adjoint by \cite[Proposition A.8.2]{taylor}.

At this point, since $\mathcal A_2$ is a nonpositive, self--adjoint operator on a Hilbert
space, it is well known that $(\mathcal A_2, D(\mathcal A_2))$ generates a cosine family
and an analytic contractive semigroup of angle $\displaystyle\frac{\pi}{2}$ on
$L^2_{\frac{1}{a}}(0,1)$ (see \cite[Example 3.14.16 and
3.7.5]{abhn}) or \cite[Theorem 6.12]{g}).

Finally, let us prove \eqref{stima2}--\eqref{stima3}. First, being
$\mathcal A_2$ the generator of a strongly continuous semigroup on
$L^2_{\frac{1}{a}}(0,1)$, if $u_0\in L^2_{\frac{1}{a}}(0,1)$, then
the solution $u$ of \eqref{linear} belongs to $C\big([0,T];
L^2_{\frac{1}{a}}(0,1)\big) \cap L^2 \big(0,T;
\cH^1_{\frac{1}{a}}(0,1)\big)$, while, if $u_0\in D(\mathcal A_2)$ and $h \in W^{1,1}(0,T; L^2_{\frac{1}{a}}(0,1))$, then $u\in
C^1\big([0,T]; L^2_{\frac{1}{a}}(0,1)\big) \cap C\big([0,T];
\cH^2_{\frac{1}{a}}(0,1)\big)$ by \cite[Lemma 4.1.5 and Proposition 4.1.6]{ch}.

Now, we shall prove \eqref{stima3}.

First, take $u_0\in D(\mathcal A_2)$ and multiply the equation by $u/a$; by the
Cauchy--Schwarz inequality we obtain for every $t\in (0,T]$,
\begin{equation}\label{derivo}
\frac{1}{2}\frac{d}{dt}\|u(t)\|^2_{L^2_{\frac{1}{a}}(0,1)}+
\|u_x(t)\|^2_{L^2(0,1)}\leq
\frac{1}{2}\|u(t)\|^2_{L^2_{\frac{1}{a}}(0,1)}+\frac{1}{2}
\|h(t)\|^2_{L^2_{\frac{1}{a}}(0,1)},
\end{equation}
from which we easily get
\begin{equation}\label{sottoderivo}
\|u(t)\|^2_{L^2_{\frac{1}{a}}(0,1)}\leq
e^T\left(\|u(0)\|^2_{L^2_{\frac{1}{a}}(0,1)}+\|h\|_{L^2_{\frac{1}{a}}(Q_T)}^2
\right)
\end{equation}
for every $t\leq T$. Integrating \eqref{derivo}, from
\eqref{sottoderivo} we also find
\begin{equation}\label{sottosotto}
\int_0^T\|u_x(t)\|^2_{L^2(0,1)}dt\leq
C_T\left(\|u(0)\|^2_{L^2_{\frac{1}{a}}(0,1)}+\|h\|_{L^2_{\frac{1}{a}}(Q_T)}^2
\right)
\end{equation}
for every $t\leq T$ and some universal constant $C_T>0$. Thus, by
\eqref{sottoderivo} and \eqref{sottosotto}, \eqref{stima2} follows
if $u_0\in D(\mathcal A_2)$. Since $D(\mathcal A_2)$ is dense in $L^2_{\frac{1}{a}}(0,1)$,
the same inequality holds if $u_0\in L^2_{\frac{1}{a}}(0,1)$.

Now, we multiply the equation by $-u_{xx}$, we integrate on $(0,1)$
and, using the Cauchy--Schwarz inequality, we easily get
\[
\frac{d}{dt}\|u_x(t)\|^{2}_{L^2(0,1)}+\|au_{xx}(t)\|^2_{L^2_{\frac{1}{a}}(0,1)}\leq
\|h\|_{L^2_{\frac{1}{a}}(0,1)}^2
\]
for every $t\leq T$, so that we find $C_T'>0$ such that
\begin{equation}\label{mah}
\|u_x(t)\|^{2}_{L^2(0,1)}+\int_0^T\|au_{xx}(t)\|^2_{L^2_{\frac{1}{a}}(0,1)}dt
\leq
C_T'\left(\|u_x(0)\|_{L^2(0,1)}+\|h\|_{L^2_{\frac{1}{a}}(Q_T)}^2\right)
\end{equation}
for every $t\leq T$.

Finally, from $u_t=au_{xx}+h$, squaring and integrating, using the
fact that $a^2\leq ca$ for some $c>0$, we find
\[
\int_0^T\|u_t(t)\|_{L^2(0,1)}^2\leq
C\left(\int_0^T\|au_{xx}\|^2_{L^2_{\frac{1}{a}}(0,1)}+\|h\|_{L^2_{\frac{1}{a}}(Q_T)}^2
\right),
\]
and together with \eqref{mah} we find
\begin{equation}\label{allafine}
\int_0^T\|u_t(t)\|_{L^2(0,1)}^2\leq
C\left(\|u_x(0)\|_{L^2(0,1)}+\|h\|_{L^2_{\frac{1}{a}}(Q_T)}^2\right).
\end{equation}

In conclusion, \eqref{sottoderivo}, \eqref{mah} and \eqref{allafine}
give \eqref{stima3}. Clearly,
\eqref{regularity1} and \eqref{stima3} hold also if
$u_0\in \cH^1_{\frac{1}{a}}(0,1)$,  since $\cH^2_{\frac{1}{a}}(0,1)$
is dense in $ \cH^1_{\frac{1}{a}}(0,1)$.
\end{proof}

\chapter{Carleman estimate for non degenerate parabolic
problems with non smooth coefficient}\label{seccarnondeg}
\section{Preliminaries}
In this chapter we prove Carleman estimates in the non degenerate
case, but in the case in which the coefficient of the operator is
{\em globally non} smooth, in the stream of \cite{fm}, thus
improving \cite{bena}, \cite{doubova}, \cite{iyinfty}, \cite{LRBV}
and \cite{lrr}.

Fix two real numbers $A<B$ and consider the problem
\begin{equation}\label{1}
\begin{cases}
v_t + \mathcal A v =h, & (t,x) \in (0,T) \times (A,B),\\
v(t,A)=v(t,B)=0, &  t \in (0,T).\\
\end{cases}
\end{equation}
Here we suppose that in a case $a$ is of class $W^{1,1}(A,B)\subset
BV(A,B)$, but no additional smoothness condition is required in some
subsets, as in the previous related papers, and in the other case we
assume that $a$ is of class $W^{1,\infty}(A,B)$. More precisely, we
assume to deal with a non degenerate problem with a coefficient $a$
satisfying one of the two conditions below:
\begin{Assumptions}\label{ipoadebole}$\ $\\
\begin{itemize}
\item[$(a_1)$] $a\in W^{1,1}(A,B)$, $a\geq
a_0>0$ in $(A,B)$ and there exist two functions $\fg \in L^1(A,B)$,
$\fh \in W^{1,\infty}(A,B)$ and two strictly positive constants
$\fg_0$, $\fh_0$ such that $\fg(x) \ge \fg_0$ for a.e. $x$ in $[A,B]$ and
\[-\frac{a'(x)}{2\sqrt{a(x)}}\left(\int_x^B\fg(t) dt + \fh_0 \right)+ \sqrt{a(x)}\fg(x) =\fh(x)\quad \text{for a.e.} \; x \in [A,B],\]
in the divergence case,
\[\frac{a'(x)}{2\sqrt{a(x)}}\left(\int_x^B\fg(t) dt + \fh_0 \right)+ \sqrt{a(x)}\fg(x) =\fh(x)\quad \text{for a.e.} \; x \in [A,B],\]
in the non divergence one,
or\\
\item[$(a_2)$] $a\in W^{1,\infty}(A,B)$ and $a\geq
a_0>0$ in $(A,B)$.
\end{itemize}
\end{Assumptions}

\begin{Remark}
Of course, the first equality in $(a_1)$ can be written as
$$
-\left[\sqrt{a(x)}\left(\int_x^B\fg(t) dt + \fh_0\right)\right]'=\fh(x),
$$
and the second one as
\[
-a(x)\left(\frac{\int_x^B\fg(t) dt + \fh_0}{\sqrt{a(x)}}\right)'=\fh(x).
\]
\end{Remark}

\begin{Example}
Let us fix $(A,B)=(0,1)$. In the divergence case, if $a(x)=2-\sqrt{1-x}$, we can choose $\fh_0=1$, $\fh=0$ and
\[
\fg(x)=\frac{\sqrt{2}}{4\sqrt{1-x}}a^{-3/2}\geq \frac{1}{8}:=\fg_0;
\]
in the non divergence case, if $a(x)=\sqrt{2-x}$, we can choose $\fh_0=1$, $\fh=0$ and
\[
\fg(x)=\frac{1}{4a^{3/2}}  \geq \frac{1}{8\sqrt{2}}:=\fg_0.
\]
\end{Example}

Now, let us introduce the function $\Phi(t,x): =\Theta(t)\psi(x)$,
where
\begin{equation}\label{theta}
\displaystyle \Theta(t) := \frac{1}{[t(T-t)]^4}
\end{equation} and
\begin{equation}\label{c_1nd}
\psi(x):=\begin{cases} \displaystyle - r\left[\int_A^x
\frac{1}{\sqrt{a(t)}} \int_t^B
\fg(s) dsdt + \int_A^x \frac{\fh_0}{\sqrt{a(t)}}dt\right] -\mathfrak{c}, &\text{ if } (a_1) \text{ holds,}\\
\displaystyle  e^{r\zeta(x)}-\mathfrak{c}, &\text{ if } (a_2) \text{
holds.}\end{cases}
\end{equation}
Here
 \[
\zeta(x)=\mathfrak{d}\int_x^B\frac{1}{a(t)}dt,
\]
where $\fd=\|a'\|_{L^\infty(A,B)}$,  $r>0$ and $\mathfrak{c}>0$ is
chosen in the second case in such a way that $\displaystyle
\max_{[A,B]} \psi<0$.

\begin{Remark}
Hypothesis \ref{ipoadebole} lets us treat {\em non smooth} coefficients in the whole spatial domain. To our best knowledge, this is the first case in which such a situation is studied, and for this we need a technical assumption, precisely represented by our hypothesis. However, we believe that, since non smooth coefficients are present, some conditions must be imposed, otherwise it would be impossible to differentiate and obtain the desired Carleman estimates.
\end{Remark}

\section{The divergence case.}\label{subsubdiv}

Our related Carleman estimate for the problem in divergence form is the following:
\begin{Theorem}\label{mono}
Assume Hypothesis $\ref{ipoadebole}$. Then, there
exist three positive constants $C$, $s_0$ and $r$ such that every
solution $v$ of \eqref{1} in
\[
\mathcal{V}_1:=L^2\big(0, T; H^2_a(A,B)\big) \cap H^1\big(0,
T;H^1_a(A,B)\big)
\]
satisfies, for all $s \ge s_0$,
\begin{equation}\label{3}
\begin{aligned}
&\int_0^T\int_A^B \left(s\Theta (v_x)^2 + s^3 \Theta^3
 v^2\right)e^{2s\Phi}dxdt\\
&\le C\left(\int_0^T\int_A^B h^{2}e^{2s\Phi}dxdt - sr \int_0^{T}
\left[a^{3/2}e^{2s\Phi}\Theta \left(\int_x^B \fg(\tau) d\tau + \fh_0
\right) (v_x)^2\right]^{x=B}_{x=A}dt\right),
\end{aligned}
\end{equation}
if $(a_1)$ holds and
\begin{equation}\label{2}
\begin{aligned}
&\int_0^T\int_A^B \left(s\Theta e^{r\zeta}(v_x)^2 + s^3 \Theta^3
e^{3r\zeta} v^2\right)e^{2s\Phi}dxdt\\
&\le C\left(\int_0^T\int_A^B h^{2}e^{2s\Phi}dxdt -
sr\int_0^T\left[ae^{2s\Phi}\Theta e^{r\zeta}(v_x)^2
\right]_{x=A}^{x=B}dt\right),
\end{aligned}
\end{equation}
if $(a_2)$ is in force.
\end{Theorem}
Here the non degenerate Sobolev spaces are
defined as
\[
\begin{aligned}
 H^1_a (A,B):=\big\{&u \text{ is absolutely
continuous in } [A,B],
\\ & \sqrt{a} u' \in  L^2(A,B) \text{ and } u(A)=u(B)=0
\big\},
\end{aligned}
\]
and
\[ H^2_a(A,B) :=  \big\{ u \in H^1_a(A,B) |\,au' \in
H^1(A,B)\big\},
\]
with the related norms.

Observe that, since the function $a$ is non degenerate, $H^1_a(A,B)$ and $H^2_a(A,B)$ coincide with
$H^1_0(A,B)$ and $H^2(A,B)\cap H^1_0(A,B)$, respectively.

\begin{Remark}\label{RemCarleman}
Obviously, in \eqref{2} we can delete all factors $e^{r\zeta}$ and
$e^{3r\zeta}$, since $\zeta$ is non negative and bounded. Indeed, in
Chapter \ref{secobserv}, we will use such a version. However, we
think that inequality \eqref{2} is more interesting due to the
presence of the weights.
\end{Remark}

Let us proceed with the proof of Theorem \ref{mono}. For $s> 0$, define the function
\[
w(t,x) := e^{s \Phi (t,x)}v(t,x),
\]
where $v$ is any solution of \eqref{1} in $\mathcal{V}_1$; observe
that, since $v\in\mathcal{V}_1$ and $\psi<0$, then $w\in\mathcal{V}_1$.
Of course, $w$ satisfies
\begin{equation}\label{1'nd}
\begin{cases}
(e^{-s\Phi}w)_t + \left(a(e^{-s\Phi}w)_x \right) _x  =h, & (t,x) \in (0,T) \times (A,B),\\
w(t,A)=w(t,B)=0, &  t \in (0,T),\\ w(T^-,x)= w(0^+, x)= 0, & x \in
(A,B).
\end{cases}
\end{equation}
The previous problem can be recast as follows. Set
\[
Lv:= v_t + (av_x)_x \quad \text{and} \quad
L_sw= e^{s\Phi}L(e^{-s\Phi}w), \quad s  > 0.
\]
Then \eqref{1'nd} becomes
\begin{equation}\label{elles}
\begin{cases}
L_sw= e^{s\Phi}h,\\
w(t,A)=w(t,B)=0, & t \in (0,T),\\
w(T^-,x)= w(0^+, x)= 0, & x \in (A,B).
\end{cases}
\end{equation}
Computing $L_sw$, one has
\[
\begin{aligned}
L_sw
=L^+_sw + L^-_sw,
\end{aligned}
\]
where
\[
L^+_sw := (aw_x)_x
 - s \Phi_t w + s^2a (\Phi_x)^2 w,
\]
and
\[
L^-_sw := w_t -2sa\Phi_x w_x -
 s(a\Phi_x)_xw.
\]
Moreover,
\begin{equation}\label{stimettand}
\begin{aligned}
2\langle L^+_sw, L^-_sw\rangle &\le 2\langle L^+_sw, L^-_sw\rangle+
\|L^+_sw \|_{L^2(\tilde Q_T)}^2 + \|L^-_sw\|_{L^2(\tilde Q_T)}^2\\
& =\| L_sw\|_{L^2(\tilde Q_T)}^2= \|he^{s\Phi}\|_{L^2(\tilde
Q_T)}^2,
\end{aligned}
\end{equation}
where $\langle\cdot, \cdot \rangle$ denotes the usual scalar product
in $L^2(\tilde Q_T)$ and $\tilde Q_T=(0,T)\times (A,B)$. As usual,
we will separate the scalar product $\langle L^+_sw, L^-_sw\rangle$
in distributed terms and boundary terms.

The following lemma is the crucial starting point, which will be
used also in the degenerate cases; for this reason, some comments
refer to the degenerate situation.
\begin{Lemma}\label{lemma1nd}
The following identity holds:
\begin{equation}\label{D&BTnd}
\left.
\begin{aligned}
&\langle L^+_sw,L^-_sw\rangle \\
&= \frac{s}{2} \int_0^T \int_A^B \Phi_{tt} w^2dxdt+ s^3
\int_0^T \int_A^B\big(2a \Phi_{xx} + a'\Phi_x\big)a(\Phi_x)^2w^2dxdt\\
&- 2s^2 \int_0^T \int_A^Ba \Phi_x \Phi_{tx}w^2dxdt  +s
\int_0^T \int_A^B(2a \Phi_{xx} + a'\Phi_x)a(w_x)^2 dxdt\\
&+ s\int_0^T \int_A^B a(a\Phi_x)_{xx} w w_xdxdt
\end{aligned}\right\}\;\text{\{D.T.\}}
\end{equation}
\begin{equation}\nonumber
 \text{\{B.T.\}}\;\left\{
\begin{aligned}
& + \int_0^T[aw_xw_t]_{x=A}^{x=B} dt- \frac{s}{2}
\int_A^B[w^2\Phi_t]_{t=0}^{t=T}dx+ \frac{s^2}{2}\int_A^B
[a(\Phi_x)^2 w^2]_{t=0}^{t=T}dt\\
&-\frac{1}{2} \int_A^B [a(w_x)^2]_{t=0}^{t=T}dx
+\int_0^T[-sa(a\Phi_x)_xw w_x]_{x=A}^{x=B}dt\\
&+ \int_0^T[-s\Phi_x a^2(w_x)^2 +s^2a\Phi_t \Phi_x w^2 - s^3
a^2(\Phi_x)^3w^2 ]_{x=A}^{x=B}dt.
\end{aligned}\right.
\end{equation}
\end{Lemma}
\begin{proof}
It {\em formally} reminds the proof of \cite[Lemma 3.4]{acf} in
$(0,1)$, but therein all the calculations were immediately motivated
due to the choice of the domain of the operator: in particular, $a$
was assumed to be of class $C^1$ with the unique possible exception
of the degeneracy point $x=0$, where Dirichlet boundary conditions
were imposed in the (WD) case and the condition $(au_x)(0)=0$ was
assumed in the (SD) case, thus making all integration by parts
possible.

Now integrations by parts are not immediately justified, since, at
least in the (SD) case - or if $(a_1)$ holds -, the unknown function
is {\sl not} in a Sobolev space of the whole interval $(A,B)$, and
so different motivations are necessary; moreover, the boundary
condition for the (SD) case chosen in \cite{acf} corresponds exactly
to the one which characterizes the domain of the operator, and of
course this fact makes life easier.

Here, we start noticing that all integrals appearing in $\langle
L^+_sw,L^-_sw\rangle$ are well defined both in the non degenerate
case and in the degenerate case by Lemma \ref{rem}, as simple
calculations show, recalling that $w=e^{s\Phi}v$. Then, in the
following, we perform {\em formal} calculations, that we will
justify accurately in Appendix \ref{secA}.

Let us start with
\begin{equation}\label{!nd}
\begin{aligned}
&\int_0^T \int_A^BL^+_sw w_t  dxdt= \int_0^T \int_A^B \{ (aw_x)_x
 - s \Phi_t w + s^2a(\Phi_x)^2 w\}w_t dxdt\\
 &=\int_0^T[aw_xw_t]_{x=A}^{x=B}dt -
 \int_0^T\frac{1}{2}\frac{d}{dt}\left(\int_A^Ba(w_x)^2dx\right)dt\\
& - \frac{s}{2}\int_A^B dx \int_0^T\Phi_t(w^2)_tdt +
\frac{s^2}{2}\int_A^B dx \int_0^T a(\Phi_x)^2(w^2)_tdt\\
&= \int_0^T[aw_xw_t]_{x=A}^{x=B}dt - \frac{s}{2}
\int_A^B[w^2\Phi_t]_{t=0}^{t=T}dx+ \frac{s^2}{2}\int_A^B
[a(\Phi_x)^2 w^2]_{t=0}^{t=T}dt\\
& -\frac{1}{2}\int_A^B[a(w_x)^2]_{t=0}^{t=T}dx+
\frac{s}{2}\int_0^T \int_A^B\Phi_{tt}w^2dxdt\\
& - s^2\int_0^T \int_A^Ba \Phi_x \Phi_{xt}w^2 dxdt.
\end{aligned}
\end{equation}

In addition, we have
\begin{equation}\label{!1nd}
\begin{aligned}
&\int_0^T \int_A^BL^+_sw (-2sa\Phi_xw_x)  dxdt= -s\int_0^T
\int_A^B\Phi_x\left[(aw_x)^2\right]_x  dxdt \\
&+ s^2\int_0^T \int_A^Ba \Phi_t \Phi_x \left(w^2\right)_x
dxdt-s^3 \int_0^T \int_A^Ba^2(\Phi_x)^3 \left(w^2\right)_x  dxdt\\
&=\int_0^T[-s\Phi_x(aw_x)^2 + s^2a\Phi_t\Phi_xw^2 - s^3a^2(\Phi_x)^3w^2]_{x=A}^{x=B}dt\\
&+s\int_0^T\int_A^B \Phi_{xx} (aw_x)^2dxdt  -s^2 \int_0^T\int_A^B (a\Phi_x)_x \Phi_t w^2\\
&-s^2 \int_0^T\int_A^B a \Phi_x\Phi_{tx}w^2  dxdt+
s^3\int_0^T\int_A^B[a^2(\Phi_x)^3]_x w^2 dxdt.
\end{aligned}
\end{equation}

Finally,
\begin{equation}\label{!2nd}
\begin{aligned}
&\int_0^T \int_A^BL^+_sw (-s(a\Phi_x)_xw) dxdt =\int_0^T[-saw_x
w(a \Phi_x)_x ]_{x=A}^{x=B}dt\\
&+s\int_0^T\int_A^B a (a \Phi_x)_{xx} ww_x dxdt+s \int_0^T\int_A^Ba(a\Phi_x)_x (w_x)^2dxdt\\
& + s^2\int_0^T\int_A^B(a\Phi_x)_x \Phi_t w^2  dxdt- s^3
 \int_0^T \int_A^B a (\Phi_x)^2 (a \Phi_x)_x w^2 dxdt.
\end{aligned}
\end{equation}
Adding \eqref{!nd}--\eqref{!2nd}, writing $[a^2(\Phi_x)^3]_x
=[a(\Phi_x)^2]_xa\Phi_x+a(\Phi_x)^2(a\Phi_x)_x$, \eqref{D&BTnd}
follows immediately.
\end{proof}

Now, the crucial step is to prove the following estimates:
\begin{Lemma}\label{lemma2ndWD}
Assume that Hypothesis $\ref{ipoadebole}.(a_1)$ holds. Then there
exist two positive constants $s_0$  and $C$ such that for all $s \ge
s_{0}$ the distributed terms of \eqref{D&BTnd} satisfy the estimate
\[
\begin{aligned}
& \frac{s}{2} \int_0^T \int_A^B \Phi_{tt} w^2dxdt+ s^3
\int_0^T \int_A^B\big(2a \Phi_{xx} + a'\Phi_x\big)a(\Phi_x)^2w^2dxdt\\
&- 2s^2 \int_0^T \int_A^Ba \Phi_x \Phi_{tx}w^2dxdt  +s
\int_0^T \int_A^B(2a \Phi_{xx} + a'\Phi_x)a(w_x)^2 dxdt \\
&+ s\int_0^T \int_A^B a(a\Phi_x)_{xx} w w_xdxdt\\
&\ge Cs\int_0^T\int_A^B \Theta (w_x)^2 dxdt + Cs^3
\int_0^T\int_A^B\Theta^3 w^2 dxdt.
\end{aligned}
\]
\end{Lemma}
\begin{proof}
Using the definition of $\Phi$, the distributed terms of $\int_0^T
\int_0^1L^+_s w L^-_s w dxdt$ take the form
\begin{equation}\label{02ndWD}
\begin{aligned}
&\frac{s}{2} \int_0^T \int_A^B \ddot{\Theta} \psi w^2 dxdt +2
s^3r^3\int_0^T \int_A^B\Theta^3\sqrt{a}\fg \left( \int_x^B
\fg(\tau)d\tau + \fh_0\right)^2w^2
dxdt\\
& -2s^2r^2\int_0^T
\int_A^B\Theta\dot{\Theta} \left( \int_x^B \fg(\tau)d\tau + \fh_0\right)^2w^2 dxdt\\
& -sr \int_0^T \int_A^B\Theta a
\left(\frac{a'}{2\sqrt{a}}\left(\int_x^B \fg (\tau) d\tau +
\fh_0\right) -\sqrt{a} g\right)_x w w_x dxdt\\
&+ 2s r \int_0^T \int_A^B\Theta a\sqrt{a}\fg(w_x)^2 dxdt.
\end{aligned}
\end{equation}

Hence, since, by Hypothesis \ref{ipoadebole}.$(a_1)$, $\fg \ge
\fg_0$ and $a \ge a_0$, we can estimate \eqref{02ndWD} from below in
the following way:
\[
\begin{aligned}
&\eqref{02ndWD}\ge\frac{s}{2}\int_0^T \int_A^B \ddot{\Theta}\psi
w^2dxdt+
2s^3r^3\sqrt{a_0}\fg_0 \fh_0^2\int_0^T \int_A^B\Theta^3 w^2 dxdt\\
&-2s^2r^2\int_0^T
\int_A^B\Theta\dot{\Theta} \left( \int_x^B\fg(\tau)d\tau + \fh_0\right)^2w^2 dxdt\\
& +2 sr\fg_0 a_0\sqrt{a_0}\int_0^T \int_A^B\Theta (w_x)^2 dxdt +sr
\int_0^T \int_A^B\Theta a \fh' w w_x dxdt.
\end{aligned}
\]

Observing that
\begin{equation}\label{magtheta}
|\Theta \dot{\Theta}| \le c \Theta^{9/4}, \Theta^{\mu} \le c \Theta ^\nu \mbox{ if } 0<\mu<\nu \mbox{ and
} |\ddot{\Theta}| \le c\Theta ^{3/2}\leq c\Theta^3
\end{equation}
for some positive constants $c$, we conclude that, for $s$ large
enough,
\[
\begin{aligned}
&\left|-2s^2r^2\int_0^T \int_A^B\Theta\dot{\Theta} \left(
\int_x^B\fg(\tau)d\tau  + \fh_0\right)^2w^2 \right| \\&\le 2r^2
s^2c\left(  \int_A^B\fg(\tau)d\tau  + \fh_0\right)^2\int_0^T
\int_A^B\Theta^3w^2 dxdt
\\
&\le \frac{C}{6}s^3\int_0^T \int_A^B\Theta^3 w^2 dxdt,
\end{aligned}
\]
for some $C>0$ and $s\geq \displaystyle \frac{12r^2c\left(
\displaystyle\int_A^B \fg(\tau)d\tau  + \fh_0\right)^2}{C}$.
Moreover,  we have
\[
\begin{aligned}
\left|\frac{s}{2} \int_0^T \int_A^B \ddot{\Theta}\psi w^2dxdt\right| &\leq s c \max_{[A,B]} |\psi|\int_0^T \int_A^B\Theta^3w^2 dxdt\\
& \leq \frac{C}{6}s^3\int_0^T \int_A^B\Theta^3w^2dxdt
\end{aligned}
\]
for $s\geq \displaystyle \sqrt{\frac{6c\max_{[A,B]} |\psi|}{C}}$ and
\[
\begin{aligned}
& \left|sr\int_0^T \int_A^B\Theta a \fh' w w_x dxdt\right| \le \frac{1}{\ve} sr \int_0^T \int_A^B\Theta a^2 |\fh'|^2 w^2dxdt\\
&
 +\ve sr\int_0^T \int_A^B\Theta  (w_x)^2 dxdt\\
& \le \frac{1}{\ve} sr c \max_{[A,B]}a^2\|\fh'\|^2_{L^\infty(A,B)}\int_0^T \int_A^B\Theta^3 w^2dxdt +\ve sr \int_0^T \int_A^B\Theta (w_x)^2dxdt \\
& \le\frac{C}{6}s^3\int_0^T \int_A^B\Theta^3w^2dxdt + \ve s
r\int_0^T \int_A^B\Theta (w_x)^2dxdt,
\end{aligned}
\]
for $s \geq \displaystyle
\sqrt{\frac{6\ve^{-1}rc\max_{[A,B]}a^2\|\fh'\|_{L^\infty(A,B)}^2}{C}}$.
In conclusion, by the previous inequalities, we find
\[
\begin{aligned}
\eqref{02ndWD}&\ge
s^3\left(2r^3\sqrt{a_0}\fg_0 \fh_0^2 -\frac{C}{2}\right)\int_0^T \int_A^B\Theta^3 w^2 dxdt\\
 &+ sr\left(2\fg_0 a_0\sqrt{a_0}- \ve\right)\int_0^T
\int_A^B\Theta (w_x)^2 dxdt.
\end{aligned}
\]

Finally, choosing $\ve = \displaystyle \fg_0 a_0 \sqrt{a_0}$ and $r$
such that
\[
2r^3\sqrt{a_0}\fg_0 \fh_0^2 -\frac{C}{2}>0,
\]
the claim follows.
\end{proof}
The counterpart of the previous inequality in the $W^{1,\infty}$
case is the following
\begin{Lemma}\label{lemma2nd}
Assume that Hypothesis $\ref{ipoadebole}.(a_2)$ holds. Then there
exist two positive constants $s_0$  and $C$ such that for all $s \ge
s_{0}$ the distributed terms of \eqref{D&BTnd} satisfy the estimate
\[
\begin{aligned}
& \frac{s}{2} \int_0^T \int_A^B \Phi_{tt} w^2dxdt+ s^3
\int_0^T \int_A^B\big(2a \Phi_{xx} + a'\Phi_x\big)a(\Phi_x)^2w^2dxdt\\
&- 2s^2 \int_0^T \int_A^Ba \Phi_x \Phi_{tx}w^2dxdt  +s
\int_0^T \int_A^B(2a \Phi_{xx} + a'\Phi_x)a(w_x)^2 dxdt \\
&+ s\int_0^T \int_A^B a(a\Phi_x)_{xx} w w_xdxdt\\
&\ge Cs\int_0^T\int_A^B \Theta e^{r\zeta}(w_x)^2 dxdt + Cs^3
\int_0^T\int_A^B\Theta^3 e^{3r\zeta} w^2 dxdt.
\end{aligned}
\]
\end{Lemma}
\begin{proof}
We proceed as in the proof of the previous Lemma. In this case the
distributed terms of $\int_0^T \int_0^1L^+_s w L^-_s w dxdt$ take
the form
\begin{equation}\label{02nd}
\begin{aligned}
&\frac{s}{2} \int_0^T \int_A^B \ddot{\Theta} \psi w^2 dxdt +
s^3r^3\fd^3\int_0^T \int_A^B\Theta^3\frac{e^{3r\zeta}}{a^2}[2r\fd-
a']w^2
dxdt\\
& -2s^2r^2\int_0^T
\int_A^B\Theta\dot{\Theta}\frac{e^{2r\zeta}}{a}w^2 dxd+ s r
\fd\int_0^T
\int_A^B\Theta e^{r\zeta}[2r\fd -a'](w_x)^2 dxdt\\
& - sr^2\fd^2 \int_0^T \int_A^B \frac{1}{a}\Theta e^{r\zeta} [r\fd +
a']w w_xdxdt.
\end{aligned}
\end{equation}
By Hypothesis \ref{ipoadebole}.$(a_2)$, choosing $r >1$, one has
\[
\fd^3(2r\fd-a')\ge \|a'\|_{L^\infty(A,B)}^4 \ \mbox{ and } \
\fd(2r\fd-a')\ge \|a'\|_{L^\infty(A,B)}^2,
\]
thus
\begin{equation}\label{Capparend}
\begin{aligned}
&\eqref{02nd}\ge\frac{s}{2}\int_0^T \int_A^B \ddot{\Theta}\psi
w^2dxdt+
\frac{s^3r^3\|a'\|_{L^\infty(A,B)}^4}{\max a^2}\int_0^T \int_A^B\Theta^3 e^{3r\zeta}w^2 dxdt\\
&-\frac{2s^2r^2}{a_0}\int_0^T
\int_A^B|\Theta\dot{\Theta}|e^{2r\zeta}w^2 dxdt + sr
\|a'\|_{L^\infty(A,B)}^2\int_0^T \int_A^B\Theta e^{r\zeta}(w_x)^2
dxdt\\& - sr^2\fd^2 \int_0^T \int_A^B \frac{1}{a}\Theta e^{r\zeta}
[r\fd + a']w w_xdxdt.
\end{aligned}
\end{equation}

Using the estimates in \eqref{magtheta}, we conclude that, for $s$
large enough,
\begin{equation}\label{ziand}
\begin{aligned}
&\left|-\frac{2s^2r^2}{a_0}\int_0^T
\int_A^B|\Theta\dot{\Theta}|e^{2r\zeta}w^2 dxdt \right| \le
\frac{2r^2c}{a_0\min e^{r\zeta}} s^2\int_0^T \int_A^B\Theta^3
e^{3r\zeta}w^2 dxdt
\\
&\le \frac{C}{6}s^3\int_0^T \int_A^B\Theta^3 e^{3r\zeta}w^2 dxdt,
\end{aligned}
\end{equation}
for some $C>0$ and $s\geq \displaystyle \frac{12r^2c}{Ca_0\min
e^{r\zeta}}$. Moreover, we have
\begin{equation}\label{ziond}
\begin{aligned}
\left|\frac{s}{2} \int_0^T \int_A^B \ddot{\Theta}\psi w^2dxdt\right| &\leq\frac{s}{2} c \max_{[A,B]} |\psi|\int_0^T \int_A^B\Theta^3w^2 dxdt\\
& \leq\frac{cs\max_{[A,B]} |\psi|}{\min e^{3r\zeta}}\int_0^T \int_A^B\Theta^3e^{3r\zeta}w^2dxdt\\
& \leq \frac{C}{6}s^3\int_0^T \int_A^B\Theta^3e^{3r\zeta}w^2dxdt
\end{aligned}
\end{equation}
for $s\geq \displaystyle \sqrt{\frac{6c\max_{[A,B]} |\psi|}{C\min
e^{3r\zeta}}}$ and
\[
\begin{aligned}
&\left|- sr^2\fd^2 \int_0^T \int_A^B \frac{1}{a}\Theta e^{r\zeta} [r\fd + a']w w_xdxdt\right|\\
&  \le  2sr^3\|a'\|_{L^{\infty}(A,B)}^3\frac{1}{a_0} \int_0^T \int_A^B \Theta e^{r\zeta} |w w_x|dxdt\\
&  \le sr^3\frac{ \|a'\|_{L^{\infty}(A,B)}^3}{\ve a_0} \int_0^T \int_A^B \Theta e^{r\zeta} w^2dxdt \\
&+  sr^3\frac{ \ve \|a'\|_{L^{\infty}(A,B)}^3}{a_0} \int_0^T \int_A^B \Theta e^{r\zeta} (w_x)^2dxdt\\
& \le s r^3 c\frac{ \|a'\|_{L^{\infty}(A,B)}^3}{\min e^{2r\zeta}a_0\ve} \int_0^T \int_A^B \Theta^3 e^{3r\zeta} w^2dxdt + \\
& sr^3\frac{ \ve \|a'\|_{L^{\infty}(A,B)}^3}{a_0} \int_0^T \int_A^B \Theta e^{r\zeta} (w_x)^2dxdt\\
& \le \frac{C}{6}s^3\int_0^T
\int_A^B\Theta^3e^{3r\zeta}w^2dxdt+sr^3\frac{ \ve
\|a'\|_{L^{\infty}(A,B)}^3}{a_0} \int_0^T \int_A^B \Theta e^{r\zeta}
(w_x)^2dxdt
\end{aligned}
\]
for $s\geq \displaystyle \sqrt{\frac{6r^3 c\frac{
\|a'\|_{L^{\infty}(A,B)}^3}{\min e^{2r\zeta}a_0\ve}}{C}}$.

In conclusion, by the previous inequalities, we obtain
\[
\begin{aligned}
\eqref{02nd}&\ge
\left(\frac{r^3\|a'\|_{L^\infty(A,B)}^4}{\max a^2}-  \frac{C}{2}\right)s^3\int_0^T \int_A^B\Theta^3 e^{3r\zeta}w^2 dxdt\\
& + sr\|a'\|_{L^\infty(A,B)}^2 \left(1- r^2\frac{ \ve
\|a'\|_{L^{\infty}(A,B)}}{a_0} \right)\int_0^T \int_A^B\Theta
e^{r\zeta}(w_x)^2 dxdt.
\end{aligned}
\]

Finally, choosing $\ve=
\displaystyle\frac{a_0}{2r^2\|a'\|_{L^\infty(A,B)}} $ and $r>1$ such
that
\[
\frac{r^3\|a'\|_{L^\infty(A,B)}^4}{\max a^2}-  \frac{C}{2}>0,
\]
the claim follows.

\end{proof}

Concerning the boundary terms in \eqref{D&BTnd}, we have
\begin{Lemma}\label{lemma4nd}
The boundary terms in \eqref{D&BTnd} reduce to
\[
sr\int_0^T\left[a^{3/2}\Theta \left(\int_x^B \fg(\tau) d\tau + \fh_0
\right) (w_x)^2\right]^{x=B}_{x=A}dt
\]
if $(a_1)$ holds and
\[
sr\|a'\|_{L^\infty(A,B)}\int_0^T[a\Theta e^{r\zeta}
(w_x)^2]^{x=B}_{x=A}dt
\]
if $(a_2)$ holds.
\end{Lemma}
\begin{proof}
First of all, since $w \in \mathcal{V}_1$, then $ w\in C\big([0,
T];H^1(A,B) \big)$. Thus $w(0, x)$, $w(T,x)$, $w_x(0,x)$, $w_x(T,x)$, $(w_x)^2(t,B)$ and $(w_x)^2(t,A)$  are indeed well defined. Moreover, we have that $w_t
(t,A)$ and $w_t(t,B)$ make sense and are actually 0.

But also $w_x(t,A)$ and $w_x(t,B)$ are well defined. In fact
$w(t,\cdot)\in H^2(A,B)$ and $w_x(t,\cdot)\in W^{1,2}(A,B)\subset
C([A,B])$. Thus $ \int_0^T[a w_xw_t]_{x=A}^{x=B}dt$ is well defined
and actually equals $0$, as we get using the  boundary conditions on
$w$. Thus, using the boundary conditions of $w=e^{s\Phi}v$ with
$v\in \mathcal{V}_1$, we get
\[
\begin{aligned}
& [w^2\Phi_t]^{t=T}_{t=0}=[e^{2s\Phi}v^2\Phi_t]^{t=T}_{t=0}=0,\\
& [a(\Phi_x)^2w^2]^{t=T}_{t=0}=
\begin{cases}
[\Theta^2r^2\left( \int_x^B \fg(\tau) d\tau
+\fh_0\right)^2e^{2s\Phi}v^2]^{t=T}_{t=0}=0, & \text{if $(a_1)$ holds},\\
[\Theta^2e^{2r\zeta}r^2e^{2s\Phi}v^2]^{t=T}_{t=0}=0, & \text{if
$(a_2)$ holds},
 \end{cases}
 \\
& [a(w_x)^2]^{t=T}_{t=0}=
\begin{cases}
\left[ae^{2s\Phi}\left(-\frac{sr}{\sqrt{a}}\Theta \left(\int_x^B
\fg(\tau) d\tau + \fh_0 \right)v+v_x\right)^2\right]^{t=T}_{t=0}=0,
&
\text{if $(a_1)$ holds},\\
\displaystyle\left[ae^{2s\Phi}\left(-\frac{sr\|a'\|_{L^\infty(A,B)}\Theta
e^{r\zeta}}{a}v+v_x\right)^2\right]^{t=T}_{t=0}=0,& \text{if $(a_2)$
holds}.
\end{cases}
\end{aligned}
\]
Finally, all integrals involving $[w]^{x=B}_{x=A}$ are obviously 0,
so that the boundary terms in \eqref{D&BTnd} reduce to
\[
\begin{aligned}
& -s\int_0^T[\Phi_x a^2(w_x)^2]^{x=B}_{x=A}dt\\
&=
\begin{cases}
sr\int_0^T\left[a^{3/2}\Theta \left(\int_x^B \fg(\tau) d\tau + \fh_0
\right) (w_x)^2\right]^{x=B}_{x=A}dt , & \text{if $(a_1)$ holds},\\
sr\|a'\|_{L^\infty(A,B)}\int_0^T[a\Theta e^{r\zeta}
(w_x)^2]^{x=B}_{x=A}dt, & \text{if $(a_2)$ holds}.
\end{cases}
\end{aligned}
\]
\end{proof}

>From Lemmas \ref{lemma1nd} -
 \ref{lemma4nd}, we deduce immediately that
there exist two positive constants $C$ and $s_0$, such that all
solutions $w$ of \eqref{1'nd} satisfy, for all $s \ge s_0$,
\begin{equation}\label{D&BT1ndWD}
\begin{aligned}
\int_0^T \int_A^BL^+_s w L^-_s w dxdt &\ge Cs\int_0^T\int_A^B \Theta
(w_x)^2 dxdt\\
&+ Cs^3 \int_0^T\int_A^B\Theta^3
w^2 dxdt\\
&+sr\int_0^T\left[a^{3/2}\Theta \left(\int_x^B \fg(\tau) d\tau +
\fh_0 \right) (w_x)^2\right]^{x=B}_{x=A}dt,
\end{aligned}
\end{equation}
if $(a_1)$ holds, and
\begin{equation}\label{D&BT1nd}
\begin{aligned}
\int_0^T \int_A^BL^+_s w L^-_s w dxdt &\ge Cs\int_0^T\int_A^B \Theta
e^{r\zeta}(w_x)^2 dxdt\\
&+ Cs^3 \int_0^T\int_A^B\Theta^3 e^{3r\zeta}
w^2 dxdt\\
&+ sr\|a'\|_{L^\infty((A,B)}\int_0^T[a\Theta e^{r\zeta}
(w_x)^2]^{x=B}_{x=A}dt,
\end{aligned}
\end{equation}
if $(a_2)$ holds.

Thus, a straightforward consequence of \eqref{stimettand},
\eqref{D&BT1ndWD} and \eqref{D&BT1nd} is the next result.
\begin{Proposition}\label{Carlemannd}
Assume Hypothesis $\ref{ipoadebole}$. Then, there
exist three positive constants $C$, $s_0$ and $r$ such that all
solutions $w$ of \eqref{1'nd} in $\mathcal{V}_1$ satisfy, for all $s
\ge s_0$,
    \[
\begin{aligned}
   & s\int_0^T\int_A^B \Theta(w_x)^2dxdt  + s^3
\int_0^T\int_A^B\Theta^3  w^2 dxdt\\
&\le
    C\left(\int_0^T\int_A^B h^2 e^{2s\Phi}dxdt- sr \int_0^{T}
   \left[a^{3/2}\Theta \left(\int_x^B \fg(\tau) d\tau +
\fh_0 \right) (w_x)^2\right]^{x=B}_{x=A}dt
    \right),
    \end{aligned}
    \]
if $(a_1)$ holds and
\[
\begin{aligned}
   & s\int_0^T\int_A^B \Theta e^{r\zeta}(w_x)^2dxdt  + s^3
\int_0^T\int_A^B\Theta^3 e^{3r\zeta} w^2 dxdt\\
&\le
    C\left(\int_0^T\int_A^B h^2 e^{2s\Phi}dxdt- sr \int_0^{T}
    \left[a\Theta e^{r\zeta}
(w_{x})^{2}\right]_{x=A}^{x=B}dt
    \right),
    \end{aligned}
    \]
if $(a_2)$ does.
\end{Proposition}

Now, we are ready to conclude the
\begin{proof}[Proof of Theorem \ref{mono}]
Recalling the definition of $w$, we have $v= e^{-s\Phi}w$ and
$v_{x}= -s\Theta \psi'e^{-s\Phi}w + e^{-s\Phi}w_{x}$. Thus, recalling that $\psi'$ is bounded, since $a$ is non degenerate, we have that, if $(a_1)$ holds, there exist some $c>0$ such that
\[
\begin{aligned}
\left(s\Theta (v_x)^2 + s^3 \Theta^3  v^2\right)e^{2s\Phi}&\leq c\left[ s\Theta(s^2\Theta^2e^{-2s\Phi}w^2+e^{-2s\Phi}(w_x)^2)+s^3 \Theta^3e^{-2s\Phi}w^2\right]e^{2s\Phi}\\
& \leq c \left(s^3\Theta^3 w^2+s\Theta(w_x)^2\right).
\end{aligned}
\]
Analogously, if $(a_2)$ holds, we have
\[
\begin{aligned}
\left(s\Theta e^{r\zeta}(v_x)^2 + s^3 \Theta^3
e^{3r\zeta} v^2\right)e^{2s\Phi}&\leq c\left[ s\Theta e^{r\zeta}(s^2\Theta^2w^2+(w_x)^2)+s^3 \Theta^3 e^{3r\zeta}w^2\right] \\
& \leq c\left[s\Theta e^{r\zeta}(w_x)^2 + s^3 \Theta^3
e^{3r\zeta} w^2\right],
\end{aligned}
\]
since $\zeta>0$. Therefore, there exists a constant $C>0$ such that
\[
\begin{aligned}
&\int_0^T\int_A^B \left(s\Theta (v_x)^2 + s^3 \Theta^3
 v^2\right)e^{2s\Phi}dxdt\\
&\leq C \int_0^T\int_A^B \left(s\Theta  (w_x)^2dxdt+s^3\Theta^3
w^2\right)dxdt,
\end{aligned}
\]
if $(a_1)$ holds, and
\[
\begin{aligned}
&\int_0^T\int_A^B \left(s\Theta e^{r\zeta}(v_x)^2 + s^3 \Theta^3
e^{3r\zeta} v^2\right)e^{2s\Phi}dxdt\\
&\leq C \int_0^T\int_A^B \left(s\Theta
e^{r\zeta}(w_x)^2dxdt+s^3\Theta^3 e^{3r\zeta}w^2\right)dxdt,
\end{aligned}
\]
if $(a_2)$ holds. By Proposition \ref{Carlemannd}, Theorem
\ref{mono} follows at once.
\end{proof}

\section{The non divergence case}
For the problem in non divergence form Theorem \ref{mono} becomes
\begin{Theorem}\label{mono1}
Assume Hypothesis $\ref{ipoadebole}$. Then, there
exist three positive constants $C$, $s_0$ and $r$ such that every
solution $v \in \mathcal{V}_2:=L^2\big(0, T; H^2_{\frac{1}{a}}(A,B)\big) \cap
H^1\big(0, T;H^1_{\frac{1}{a}}(A,B)\big)$ of
\begin{equation}\label{1nd}
\begin{cases}
v_t + av_{xx} =h, & (t,x) \in (0,T) \times (A,B),\\
v(t,A)=v(t,B)=0, &  t \in (0,T),\\
\end{cases}
\end{equation}
satisfies, for all $s \ge s_0$,
\[
\begin{aligned}
&\int_0^T\int_A^B \left(s\Theta (v_x)^2 + s^3 \Theta^3
 v^2\right)e^{2s\Phi}dxdt\\
&\le C\left(\int_0^T\int_A^B h^{2}e^{2s\Phi}dxdt -
sr\int_0^T\Theta(t)\left[\sqrt{a}\left(\int_x^B \fg(\tau) d\tau +
\fh_0 \right)(v_x)^2 e^{2s\Phi}\right]_{x=A}^{x=B}dt\right),
\end{aligned}
\]
Hence, by the definitions of $Z$,
if $(a_1)$ holds and \eqref{2}
if $(a_2)$ is in force.
\end{Theorem}
Here $ H^1_{\frac{1}{a}}(A,B)$ and  $H^2_{\frac{1}{a}}(A,B)$ are
formally defined as in the degenerate case.

\begin{Remark}\label{RemCarleman1}
Remark \ref{RemCarleman} still holds also for the previous theorem.
\end{Remark}

\subsection{Proof of Theorem $\ref{mono1}$ when $(a_1)$ holds:}
We proceed as in Chapter \ref{subsubdiv}. For $s> 0$, define the function
\[
w(t,x) := e^{s \Phi (t,x)}v(t,x),
\]
where $v$ is any solution of \eqref{1} in $\mathcal{V}_2$; observe
that, since $v\in\mathcal{V}_2$ and $\psi<0$, then $w\in\mathcal{V}_2$.
Of course, $w$ satisfies
\begin{equation}\label{1'nd1}
\begin{cases}
(e^{-s\Phi}w)_t + a(e^{-s\Phi}w)_{xx} =h, & (t,x) \in (0,T) \times (A,B),\\
w(t,A)=w(t,B)=0, &  t \in (0,T),\\ w(T^-,x)= w(0^+, x)= 0, & x \in
(A,B).
\end{cases}
\end{equation}
Setting
\[
Lv:= v_t + av_{xx} \quad \text{and} \quad
L_sw= e^{s\Phi}L(e^{-s\Phi}w), \quad s  > 0,
\]
then \eqref{1'nd1} becomes
\begin{equation}\label{elles1}
\begin{cases}
L_sw= e^{s\Phi}h,\\
w(t,A)=w(t,B)=0, & t \in (0,T),\\
w(T^-,x)= w(0^+, x)= 0, & x \in (A,B).
\end{cases}
\end{equation}
Computing $L_sw$, one has
\[
L_sw =L^+_sw + L^-_sw,
\]
where
\[
L^+_sw := aw_{xx}
 - s \Phi_t w + s^2a (\Phi_x)^2 w,
\]
and
\[
L^-_sw := w_t -2sa\Phi_x w_x -
 sa\Phi_{xx}w.
\]
Moreover, similarly to \eqref{stimettand}, we have
\begin{equation}\label{stimettanda}
2\langle L^+_sw, L^-_sw\rangle_{L^2_{\frac{1}{a}}(\tilde Q_T)} \le
\|he^{s\Phi}\|_{L^2_{\frac{1}{a}}(\tilde Q_T)}^2,
\end{equation}
where $\langle\cdot, \cdot \rangle_{L^2_{\frac{1}{a}}(\tilde Q_T)}$
denotes the usual scalar product in $L^2_{\frac{1}{a}}(\tilde Q_T)$
and $\tilde Q_T=(0,T)\times (A,B)$. As usual, separating the scalar
product $\langle L^+_sw, L^-_sw\rangle_{L^2_{\frac{1}{a}}(\tilde
Q_T)}$ in distributed terms and boundary terms, we obtain
\begin{Lemma}\label{LEMMA1ND1}
The following identity holds:
\begin{equation}\label{D&BTnd1}
\begin{aligned}
\left.
\begin{aligned}
<L^+_sw,L^-_sw>_{L^2_{\frac{1}{a}}(\tilde Q_T)} \;&=\; s\int_{\tilde
Q_T}(a\Phi_{xx}+(a\Phi_x)_x)(w_x)^2dxdt\\
&+ s^3 \int_{\tilde Q_T}(\Phi_x)^2(a\Phi_{xx}+(a\Phi_x)_x)w^2dxdt\\
& -2 s^2\int_{\tilde Q_T} \Phi_x\Phi_{xt}w^2dxdt
+\frac{s}{2}\int_{\tilde Q_T} \frac{\Phi_{tt}}{a}w^2dxdt\\
&+s\int_{\tilde Q_T}(a\Phi_{xx})_xw w_xdxdt\end{aligned}
\right\}\;\text{\{D.T.\}}\\
\text{\{B.T.\}}
\left\{
\begin{aligned}
 -\frac{1}{2}\int_A^B\Big[(w_x)^2\Big]_0^Tdx
+\int_0^T\Big[w_xw_t\Big]_A^Bdt \\
- s \int_0^T\Big[ a \Phi_x(w_x)^2\Big]_A^Bdt- s\int_0^T\!\!\Big[ a\Phi_{xx}ww_x\Big]_A^Bdt\\
 +\frac{1}{2} s\int_A^B\!\!\Big[\left(s (\Phi_x)^2 -
\frac{\Phi_t}{a}\right)w^2\Big]_0^Tdx \\
 - s^2 \int_0^T\Big[\big(s a (\Phi_x)^3-
\Phi_x\Phi_t\big)w^2\Big]_A^Bdt.
\end{aligned}\right.
\end{aligned}
\end{equation}
\end{Lemma}

\begin{proof} It is an adaptation of the proof of \cite[Lemma
3.8]{cfr} to which we refer. Let us simply remark that in our case
all integrals and integrations by parts are justified by the
definition of $H^2_{\frac{1}{a}}(A,B)$ and by the regularity of the
functions $\fg$ and $\fh$.
\end{proof}

Now, the crucial step is to
prove the following estimates:
\begin{Lemma}\label{lemma2ndWD1}
Assume that Hypothesis $\ref{ipoadebole}$.$(a_1)$ holds. Then there
exist two positive constants $s_0$  and $C$ such that for all $s \ge
s_{0}$ the distributed terms of \eqref{D&BTnd1} satisfy the estimate
\[
\begin{aligned}
& s\int_0^T\int_A^B(a\Phi_{xx}+(a\Phi_x)_x)(w_x)^2dxdt + s^3
\int_0^T\int_A^B(\Phi_x)^2(a\Phi_{xx}+(a\Phi_x)_x)w^2dxdt
\\
& -2 s^2\int_0^T\int_A^B \Phi_x\Phi_{xt}w^2dxdt
 +\frac{s}{2}\int_0^T\int_A^B\frac{\Phi_{tt}}{a}w^2dxdt
+ s\int_{\tilde Q_T}(a\Phi_{xx})_xw w_xdxdt\\
&\ge Cs\int_0^T\int_A^B \Theta (w_x)^2 dxdt + Cs^3
\int_0^T\int_A^B\Theta^3 w^2 dxdt.
\end{aligned}
\]
\end{Lemma}
\begin{proof}
Using the definition of $\Phi$, the distributed terms of
$\displaystyle\int_{Q_T} \frac{1}{a}L^+_s w L^-_s w dxdt$
take the form
\begin{equation}\label{02ndWD1}
\begin{aligned}
&\frac{s}{2} \int_0^T \int_A^B\frac{1}{a} \ddot{\Theta} \psi w^2
dxdt + 2s^3r^3\int_0^T \int_A^B\frac{1}{\sqrt{a}}\Theta^3\fg \left(
\int_x^B \fg(\tau)d\tau + \fh_0\right)^2w^2
dxdt\\
& -2s^2r^2\int_0^T \int_A^B\frac{1}{a}\Theta\dot{\Theta} \left(
\int_x^B \fg(\tau)d\tau + \fh_0\right)^2w^2 dxdt+ 2s r \int_0^T
\int_A^B\Theta \sqrt{a}\fg(w_x)^2 dxdt\\
& +sr \int_0^T \int_A^B\Theta \fh' ww_x dxdt.
\end{aligned}
\end{equation}

Hence, since, by Hypothesis \ref{ipoadebole}.$(a_1)$, $\fg \ge
\fg_0$ and $a \ge a_0$, we can estimate \eqref{02ndWD1} from below
in the following way:
\[
\begin{aligned}
&\eqref{02ndWD1}\ge\frac{s}{2}\int_0^T \int_A^B
\frac{1}{a}\ddot{\Theta}\psi w^2dxdt+
2s^3r^3\frac{1}{\max_{[A,B]}\sqrt{a}}\fg_0 \fh_0^2\int_0^T \int_A^B\Theta^3 w^2 dxdt\\
&-2s^2r^2 \frac{1}{a_0}\int_0^T \int_A^B\Theta|\dot{\Theta}| \left(
\int_x^B\fg(\tau)d\tau + \fh_0\right)^2w^2 dxdt +2
sr\fg_0\sqrt{a_0}\int_0^T
\int_A^B\Theta (w_x)^2 dxdt\\
& +sr \int_0^T \int_A^B\Theta \fh' ww_x dxdt.
\end{aligned}
\]
By the estimates in \eqref{magtheta}, we conclude that, for $s$ large
enough,
\[
\begin{aligned}
&2s^2r^2\frac{1}{a_0}\int_0^T \int_A^B\Theta |\dot{\Theta}|
\left(  \int_x^B\fg(\tau)d\tau  + \fh_0\right)^2w^2 \\
&\le 2r^2 s^2\frac{1}{a_0}c\left(  \int_A^B\fg(\tau)d\tau  +
\fh_0\right)^2\int_0^T \int_A^B\Theta^3w^2 dxdt
\\
&\le \frac{C}{6}s^3\int_0^T \int_A^B\Theta^3 w^2 dxdt,
\end{aligned}
\]
for some $C>0$ and $s\geq \displaystyle \frac{12r^2c\left(
\displaystyle\int_A^B \fg(\tau)d\tau  + \fh_0\right)^2}{Ca_0}$.
Moreover,  we have
\[
\begin{aligned}
\left|\frac{s}{2} \int_0^T \int_A^B \frac{1}{a}\ddot{\Theta}\psi w^2dxdt\right| &\leq s c \frac{1}{a_0}\max_{[A,B]} |\psi|\int_0^T \int_A^B\Theta^3w^2 dxdt\\
& \leq \frac{C}{6}s^3\int_0^T \int_A^B\Theta^3w^2dxdt
\end{aligned}
\]
for $s\geq \displaystyle \sqrt{\frac{6c\max_{[A,B]} |\psi|}{Ca_0}}$,
and
\[
\begin{aligned}
& \left|sr \int_0^T \int_A^B\Theta \fh' ww_x dxdt\right| \le \frac{1}{\ve} sr \int_0^T \int_A^B\Theta |\fh'|^2 w^2dxdt +\ve sr\int_0^T \int_A^B\Theta  (w_x)^2 dxdt\\
& \le \frac{1}{\ve} sr c\|\fh'\|^2_{L^\infty(A,B)}\int_0^T \int_A^B\Theta^3 w^2dxdt +\ve sr \int_0^T \int_A^B\Theta (w_x)^2dxdt \\
& \le\frac{C}{6}s^3\int_0^T \int_A^B\Theta^3w^2dxdt + \ve s
r\int_0^T \int_A^B\Theta (w_x)^2dxdt
\end{aligned}
\]
for $s \geq \displaystyle
\sqrt{\frac{6\ve^{-1}rc\|\fh'\|_{L^\infty(A,B)}^2}{C}}$. In
conclusion, by the previous inequalities, we find
\[
\begin{aligned}
\eqref{02ndWD1}&\ge
s^3\left(2r^3\frac{1}{\max_{[A,B]}\sqrt{a}}\fg_0 \fh_0^2 -\frac{C}{2}\right)\int_0^T \int_A^B\Theta^3 w^2 dxdt\\
 &+ sr\left(2\fg_0 \sqrt{a_0}- \ve\right)\int_0^T
\int_A^B\Theta (w_x)^2 dxdt
\end{aligned}
\]
for some $C>0$ and $s$ large enough.

Finally, choosing $\ve = \displaystyle\fg_0 \sqrt{a_0}$ and $r$ such
that
\[
2r^3\frac{1}{\max_{[A,B]}\sqrt{a}}\fg_0 \fh_0^2 -\frac{C}{2}>0,
\]
the claim follows.
\end{proof}
Concerning the boundary terms in \eqref{D&BTnd1}, we have
\begin{Lemma}\label{lemma4nd1}
The boundary terms in \eqref{D&BTnd1} reduce to
\[
sr\int_0^T\Theta(t)\left[\sqrt{a}\left(\int_x^B \fg(\tau) d\tau +
\fh_0 \right)(w_x)^2\right]_{x=A}^{x=B}dt.
\]

\begin{proof}
Using the definition of $\Phi$ we have that the boundary terms
become
\begin{equation}\label{btnd}
\begin{aligned}
 \big\{B.T.\big\}\; &=\;
-\frac{1}{2}\int_A^B\Big[(w_x)^2\Big]_{t=0}^{t=T}dx
+\int_0^T\Big[w_xw_t\Big]^{x=B}_{x=A}dt\\
&+\frac{1}{2}\int_A^B\Big[\Big(s^2\Theta^2(\psi')^2
-\frac{s}{a}\dot{\Theta}\psi\Big)w^2\Big]_{t=0}^{t=T}dx
-s\int_0^T\Theta(t)\Big[a\psi'(w_x)^2\Big]^{x=B}_{x=A}dt
\\&-s\int_0^T\Theta(t)\Big[a \psi '' w w_x\Big]^{x=B}_{x=A}dt
-s^3\int_0^T\Theta^3(t)\Big[a (\psi')^3 w^2\Big]^{x=B}_{x=A}dt\\&
+s^2\int_0^T\Theta(t)\dot \Theta (t) \Big[\psi
 \psi'w^2\Big]^{x=B}_{x=A}dt.
\end{aligned}
\end{equation}
Since $w \in \mathcal{V}_2$,  $w(0, x)$, $w(T,x)$, $w_x(0,x)$,
$w_x(T,x)$ and $\int_A^B \big[w_x^2\big]_{t=0}^{t=T}dx$ are well defined,
using the boundary conditions and the definition of $w$ itself, we
get
\[
\int_A^B
\Big[-\frac{1}{2}(w_x)^2+\frac{1}{2}\Big(s^2\Theta^2(\psi')^2
-\frac{s}{a}\dot{\Theta}\psi\Big)w^2\Big]_{t=0}^{t=T}dx =0.
\]

Moreover, since $w \in \mathcal{V}_2$, we have that
 $w_t (t,A)$ and $w_t(t,B)$ make sense. Therefore, also $w_x(t,A)$ and $w_x(t,B)$
are well defined, since  $w(t,\cdot)\in H^2_{\frac{1}{a}}(A,B)$.
Thus $ \int_0^T[ w_xw_t]_{x=A}^{x=B}dt$ is well defined and actually
equals $0$. Indeed, by the boundary conditions, we find
\[
|w_t(t,x)|\leq \left(\int_A^Bw_{tx}(t,y)^2dy\right)^{1/2}
\max\{\sqrt{x-A},\sqrt{B-x}\}\to 0
\]
as $x\to A$ or $x\to B$, the integral being finite.

Now, $w(t, A)$ and $w(t,B)$ being well defined, by the boundary
conditions on $w$, the other terms of \eqref{D&BTnd1} reduce to
\[
-s\int_0^T\Theta(t)\Big[a\psi'(w_x)^2\Big]_{x=A}^{x=B}dt
=sr\int_0^T\Theta(t)\left[\sqrt{a}\left(\int_x^B \fg(\tau) d\tau +
\fh_0 \right)(w_x)^2\right]_{x=A}^{x=B}dt.
\]
\end{proof}
\end{Lemma}

\subsection{Proof of Theorem $\ref{mono1}$ when $(a_2)$ holds:}

Now, assume that Hypothesis \ref{ipoadebole}.$(a_2)$ holds. Then
inequality \eqref{2} in the non divergence case is a simple
consequence of Theorem \ref{mono}:\\
rewrite the equation of
    \eqref{1} as $ v_t + (av_x)_x = \bar{h}, $ where $\bar{h}
    := h + a'v_x$. Then, applying Theorem \ref{mono}, there exists
    two positive constants $C$ and $s_0 >0$, such that
 for all $s \ge s_0$,
\begin{equation}\label{sd}
\begin{aligned}
&\int_0^T\int_A^B \left(s\Theta e^{r\zeta}(v_x)^2 + s^3 \Theta^3
e^{3r\zeta} v^2\right)e^{2s\Phi}dxdt\\
&\le C\left(\int_0^T\int_A^B {\bar h}^{2}e^{2s\Phi}dxdt -
sr\int_0^T\left[a\Theta e^{r\zeta}(v_x)^2e^{2s\Phi}
\right]_{x=A}^{x=B}dt\right).
\end{aligned}
\end{equation}
Using the definition of $\bar{h}$, the term $\int_{Q_T}\bar{h}^2
e^{2s\Phi(t,x)}dxdt$ can be estimated in the following way
\begin{equation}\label{sd1}
\begin{aligned}
\int_0^T\int_A^B \bar{h}^2e^{2s\Phi}dxdt &\le
2\int_0^T\int_A^B h^2e^{2s\Phi}dxdt
+2\|a'\|_{L^\infty(\tilde Q_T)}^2\int_0^T\int_A^B
e^{2s\Phi}(v_x)^2dxdt\\
& \le  2\int_0^T\int_A^B h^2e^{2s\Phi}dxdt
+2\|a'\|_{L^\infty(\tilde Q_T)}^2c\int_0^T\int_A^B \Theta e^{r\zeta}
e^{2s\Phi}(v_x)^2dxdt,
\end{aligned}
\end{equation}
where $c:= \max_{[0, T]}(t(T-t))^4= \displaystyle\left(\frac{T}{2}\right)^8$. Thus, by \eqref{sd} and \eqref{sd1}, one has
\[
\begin{aligned}
    &\int_0^T\int_A^B \left(s\Theta e^{r\zeta}(v_x)^2 - 2\|a'\|_{L^\infty(\tilde Q_T)}^2c \Theta e^{r\zeta}
   (v_x)^2+ s^3 \Theta^3
e^{3r\zeta} v^2\right)e^{2s\Phi}dxdt\\
&\le C\left(\int_0^T\int_A^B h^{2}e^{2s\Phi}dxdt -
sr\int_0^T\left[a\Theta e^{r\zeta}(v_x)^2e^{2s\Phi}
\right]_{x=A}^{x=B}dt\right).
    \end{aligned}
    \]
    Now, let $s_1>0$ be such that $\displaystyle \frac{s_1}{2} \ge 2\|a'\|_{L^\infty(\tilde Q_T)}^2c$. Then, for all $s \ge s_1$
    \[
     \begin{aligned}
   & \int_0^T\int_A^B \left(s\Theta e^{r\zeta}(v_x)^2 - 2\|a'\|_{L^\infty(\tilde Q_T)}^2c \Theta e^{r\zeta}
   (v_x)^2 \right)e^{2s\Phi}dxdt\\
    &\ge \frac{s}{2} \int_0^T\int_A^B \Theta e^{r\zeta}(v_x)^2e^{2s\Phi}dxdt.
    \end{aligned}  \]
    Hence the claim follows for all $s \ge \max\{s_0, s_1\}$.

\chapter{Carleman estimate for degenerate non smooth parabolic
problems}\label{Carleman estimate}

In this chapter we prove crucial estimates of Carleman type in
presence of a degenerate coefficient. Such inequalities will be
used, for example, to prove observability inequalities for the
adjoint problem of \eqref{linear} in both the weakly and the
strongly degenerate cases.

\section{Carleman estimate for the problem in divergence form}
Let us consider again problem \eqref{1} in divergence form, where now $a$ satisfies one
of the assumptions describing the (WD) or the (SD) case, which we
briefly recollect in the following
\begin{Assumptions}\label{Ass02}
The function $a$ satisfies Hypothesis $\ref{Ass0}$ or Hypothesis
$\ref{Ass01}$. Moreover,
if Hypothesis $\ref{Ass01}$ holds with $K > \displaystyle
\frac{4}{3}$, we suppose that there exists a constant $\vartheta \in
(0, K]$ such that the function
\begin{equation}\label{dainfinito}
\begin{array}{ll}
x \mapsto \dfrac{a(x)}{|x-x_0|^{\vartheta}} &
\begin{cases}
& \mbox{ is nonincreasing on the left of $x=x_0$,}\\
& \mbox{ is nondecreasing on the right of $x=x_0$}.
\end{cases}
\end{array}
\end{equation}
In addition, when $K > \displaystyle\frac{3}{2}$, the previous map is bounded below away from $0$ and there exists a constant $\Sigma>0$ such that
\begin{equation}\label{Sigma}
|a'(x)|\leq \Sigma |x-x_0|^{2\vartheta-3} \mbox{ for a.e. }x\in
[0,1].
\end{equation}
\end{Assumptions}

\begin{Remark}\label{domfurbo}
Condition \eqref{dainfinito} is more general than the corresponding
one for $x_0=0$ required in \cite{acf} for the (SD) case. Indeed, in
this paper we require it only in the sub-case $\displaystyle K>\frac{4}{3}$ of the
(SD) case. On the other hand, let us note that requiring
\eqref{dainfinito}, also with $x_0=0$ as in \cite{acf}, together
with Hypothesis \ref{Ass0}, implies $\vartheta a\leq (x-x_0)a'\leq
Ka$ in $(0,1)$, so that $a'$ is automatically bounded away from $0$
far from $x_0$. Similar situations were considered in \cite{bena},
\cite{doubova}, \cite{LRBV}, or \cite{lrr}, where a certain
regularity was assumed somewhere, even in the non degenerate non
smooth case.

To our best knowledge, this paper is the first one where non
smoothness is assumed globally in the case of an absolutely
continuous coefficient, besides degenerate.
\end{Remark}

\begin{Remark}\label{remark5} The additional requirements for the sub-case $K>3/2$ are technical
ones, which are used just to guarantee the convergence of some
integrals (see the Appendix). Of course,
the prototype $a(x)=|x-x_0|^K$ satisfies such a condition with
$\vartheta=K$.
\end{Remark}

Now, in order to state our Carleman estimate in presence of a
degenerate non smooth coefficient, we start similarly to the
previous chapter; but, being such an inspiration only formal, the
result is completely different. In particular the sign of the
boundary term will have a different and crucial role in the two
cases.

Let us introduce the function $\varphi(t,x): =\Theta(t)\psi(x)$,
where $\Theta$ is defined as in \eqref{theta} and
\begin{equation}\label{c_1}
\psi(x) := c_1\left[\int_{x_0}^x \frac{y-x_0}{a(y)}dy- c_2\right],
\end{equation}
with $c_2> \displaystyle \max\left\{\frac{(1-x_0)^2}{a(1)(2-K)},
\frac{x_0^2}{a(0)(2-K)}\right\}$ and $c_1>0$. A more precise
restriction on $c_1$ will be needed later for the observability inequalities of Chapter \ref{secobserv}. Observe that $\Theta (t)
\rightarrow + \infty \, \text{ as } t \rightarrow 0^+, T^-$, and by
Lemma \ref{rem} we have that, if $x>x_0$,
\begin{equation}\label{defpsi}
\begin{aligned}
\psi(x)&\leq c_1 \left[ \int_{x_0}^x
\frac{(y-x_0)^K}{a(y)}\frac{1}{(y-x_0)^{K-1}}dy- c_2\right]\\
&\leq c_1\left[
\frac{(x-x_0)^K}{a(x)}\frac{(x-x_0)^{2-K}}{2-K}-c_2\right]\\
&\leq c_1 \left[
\frac{(1-x_0)^K}{a(1)}\frac{(1-x_0)^{2-K}}{2-K}-c_2\right]\\
&=c_1\left[ \frac{(1-x_0)^2}{(2-K)a(1)}-c_2\right]<0.
\end{aligned}
\end{equation}
In the same way one can treat the case $x\in[0,x_0)$, so that
\[
\psi(x)<0 \quad \mbox{ for every }x\in[0,1].
\]
Moreover, it is also easy to see that $\psi \ge -c_1c_2$.

\medskip
Our main result is the following.
\begin{Theorem}\label{Cor1}
Assume Hypothesis  $\ref{Ass02}$. Then, there exist
two positive constants $C$ and $s_0$ such that every solution $v$ of
\eqref{1} in divergence form in
\begin{equation}\label{v}
\mathcal{S}_1:=L^2\big(0, T; {\cal H}^2_a(0,1)\big) \cap H^1\big(0,
T;{\cal H}^1_a(0,1)\big)
\end{equation}
satisfies, for all $s \ge s_0$,
\[
\begin{aligned}
&\int_{Q_T} \left(s\Theta a(v_x)^2 + s^3 \Theta^3
\frac{(x-x_0)^2}{a} v^2\right)e^{2s\varphi}dxdt\\
&\le C\left(\int_{Q_T} h^{2}e^{2s\varphi}dxdt +
sc_1\int_0^T\left[a\Theta e^{2s \varphi}(x-x_0)(v_x)^2
dt\right]_{x=0}^{x=1}\right),
\end{aligned}
\]
where $c_{1}$ is the constant introduced in \eqref{c_1}.
\end{Theorem}
\subsection{Proof of Theorem $\ref{Cor1}$}

We start as in the proof of Theorem \ref{mono}: for $s> 0$,
define the function
\[
w(t,x) := e^{s \varphi (t,x)}v(t,x),
\]
where $v$ is any solution of \eqref{1} in $\mathcal{S}_1$, so that
$w\in\mathcal{S}_1$ and $w$ satisfies \eqref{1'nd}, which we re--write
as \eqref{elles}, with $\Phi$ replaced by $\vp$. Moreover, Lemma
\ref{lemma1nd} still holds also in this case, again with $\Phi$
replaced by $\vp$. Thus we start with the analogue of Lemma
\ref{lemma2nd}, which now gives the following estimate:
\begin{Lemma}\label{lemma2}
Assume Hypothesis $\ref{Ass02}$. Then there exist two positive
constants $s_0$  and $C$ such that for all $s \ge s_{0}$ the
distributed terms of \eqref{D&BTnd} satisfy the estimate
\[
\begin{aligned}
& s^3 \int_{Q_T}\big([a^2(\varphi_x)^3]_x-a (\varphi_x)^2 (a \varphi_x)_x\big)w^2dxdt\\
&+\frac{s}{2} \int_{Q_T} \varphi_{tt} w^2dxdt- 2s^2 \int_{Q_T}a \varphi_x \varphi_{tx}w^2dxdt \\
& +s \int_{Q_T}(\varphi_{xx} a^2 + a(a\varphi_x)_x)(w_x)^2
dxdt\\
& \ge \frac{C}{2}s\int_{Q_T} \Theta a(w_x)^2 dxdt +
\frac{C^3}{2}s^3 \int_{Q_T}\Theta^3 \frac{(x-x_0)^2}{a} w^2
dxdt.
\end{aligned}
\]
\end{Lemma}
\begin{proof}
Using the definition of $\varphi$ and recalling that
\begin{equation}\label{afix}
 a\vp_x=c_1\Theta(x-x_0) \mbox{ and }(a \varphi_x)_x=c_1\Theta,
\end{equation}
so that $(a \varphi_x)_{xx}=0$, the distributed terms of $\int_0^T
\int_0^1L^+_s w L^-_s w dxdt$ take the form
\begin{equation}\label{02}
\begin{aligned}
&\frac{s}{2}\int_{Q_T} \ddot{\Theta} \psi w^2 dxdt +
s^3c_1^3\int_0^T
\int_0^1\Theta^3\left[\left(\frac{(x-x_0)^3}{a}\right)_x+\frac{(x-x_0)^2}{a}\right]w^2
dxdt\\
& -2s^2c_1^2\int_0^T
\int_0^1\Theta\dot{\Theta}\frac{(x-x_0)^2}{a}w^2 dxdt\\
&+ s c_1\int_0^T
\int_0^1\Theta\left[a\left(\frac{x-x_0}{a}\right)_x+1\right]a(w_x)^2
dxdt.
\end{aligned}
\end{equation}

Now, by Lemma \ref{rem}, we immediately have that
$\displaystyle\left(\frac{(x-x_0)^3}{a}\right)_x\geq0$. Moreover,
\[
a\left(\frac{x-x_0}{a}\right)_x=\frac{a-(x-x_0)a'}{a}\geq 1-K \
\mbox{ for every }x\in (0,1).
\]
Hence, in the (WD) case it is immediately positive, while in the
(SD) case we have
\[
a\left(\frac{x-x_0}{a}\right)_x+1\geq 2-K>0 \ \mbox{ for every }x\in
(0,1).
\]

Hence, we can estimate \eqref{02} from below in the following way:
\begin{equation}\label{Cappare}
\begin{aligned}
&\ge\frac{s}{2}\int_0^T \int_0^1 \ddot{\Theta}\psi w^2dxdt+
s^3c_1^3\int_0^T \int_0^1\Theta^3 \frac{(x-x_0)^2}{a}w^2 dxdt\\
&-2s^2c_1^2\int_0^T \int_0^1\Theta\dot{\Theta}\frac{(x-x_0)^2}{a}w^2
dxdt + s\bar C\int_{Q_T}\Theta a(w_x)^2 dxdt
\end{aligned}
\end{equation}
for some universal positive constant $\bar C>0$.

Using \eqref{magtheta}, we conclude that, for
$s$ large enough,
\begin{equation}\label{c18}
\begin{aligned}
&\left|-2s^2c_1^2\int_0^T
\int_0^1\Theta\dot{\Theta}\frac{(x-x_0)^2}{a}w^2 dxdt \right| \le
2cc_1^2 s^2\int_0^T \int_0^1\Theta^3 \frac{(x-x_0)^2}{a}w^2 dxdt
\\
&\le \frac{c_1^3}{4}s^3\int_0^T \int_0^1\Theta^3 \frac{(x-
x_0)^2}{a}w^2 dxdt,
\end{aligned}
\end{equation}
as soon as $s\geq \displaystyle\frac{8c}{c_1}$.

Moreover, $\Theta$ being convex and since \eqref{magtheta} holds, by the very definition of $\psi$, we
have
\begin{equation}\label{zio}
\frac{s}{2}\int_{Q_T} \ddot{\Theta}\psi w^2dxdt \geq -
s\frac{c_1c_2}{2}c\int_0^T \int_0^1\Theta ^{3/2}w^2 dxdt,
\end{equation}
since $\displaystyle\int_{x_0}^x\frac{y-x_0}{a(y)}dy\geq0$ for every
$x\in [0,1]$. Hence, it remains to bound the term $\int_0^T
\int_0^1\Theta ^{3/2}w^2 dxdt$. Using the Young inequality, we find
\begin{equation}\label{sopra}
\begin{aligned}
&s\frac{c_1c_2}{2}c\int_0^1
\Theta^{3/2}w^2 dx \\
&=s\frac{c_1c_2}{2}c\int_0^1\left(\Theta\frac{a^{1/3}}{|x-x_0|^{2/3}}w^2\right)^{3/4}\left(\Theta^3
\frac{|x-x_0|^2}{a} w^2\right)^{1/4} \\
&\le s\ve\frac{3 c_1c_2}{8}c \int_0^1
\Theta\frac{a^{1/3}}{|x-x_0|^{2/3}}w^2dx + s\frac{c_1c_2}{8\ve^3}c
\int_0^1\Theta^3 \frac{|x-x_0|^2}{a} w^2 dx.
\end{aligned}
\end{equation}

Now, we concentrate on the integral $\displaystyle \int_0^1
\frac{a^{1/3}}{|x-x_0|^{2/3}}w^2dx$.

If $K \le \displaystyle \frac{4}{3}$ (where $K$ is the constant
appearing in Hypothesis $\ref{Ass0}$ or \ref{Ass01}), we introduce
the function $p(x) = |x-x_0|^{4/3}$. Obviously, there exists $ q \in
\left(1, \displaystyle\frac{4}{3}\right)$ such that the function
$\displaystyle x\mapsto\frac{p(x)}{|x-x_0|^q}$ is nonincreasing on
the left of $x=x_0$ and nondecreasing on the right of $x=x_0$. Thus,
using the Hardy-Poincar\'{e} inequality (see Proposition \ref{HP})
and Lemma \ref{rem}, one has
\begin{equation}\label{hpapplbis}
\begin{aligned}
\int_0^1 \frac{a^{1/3}}{|x-x_0|^{2/3}}w^2dx & \le \max_{[0,1]}
a^{1/3}\int_0^1 \frac{1}{|x-x_0|^{2/3}}w^2dx \\
&= \max_{ [0,1]} a^{1/3}\int_0^1  \frac{p}{(x-x_0)^2} w^2 dx\\
&\le \max_{[0,1]} a^{1/3}C_{HP}\int_0^1 p (w_x)^2 dx \\
& = \max_{[0,1]} a^{1/3} C_{HP}\int_0^1  a \frac
{|x-x_0|^{4/3}}{a} (w_x)^2 dx\\
&\le \max_{ [0,1]} a^{1/3} C_{HP}C_2 \int_0^1 a (w_x)^2 dx,
\end{aligned}
\end{equation}
where $C_{HP}$ is the Hardy-Poincar\'{e} constant and $C_2:=
\max\left\{\displaystyle\frac{x_0^{4/3}}{a(0)},\displaystyle\frac{(1-x_0)^{4/3}}{a(1)}\right\}$.
In this way, we find
\begin{equation}\label{kminore}
\begin{aligned}
\frac{s}{2}\int_{Q_T} \ddot{\Theta}\psi w^2dxdt \geq & -
s\ve\frac{3 c_1c_2}{8}c \max_{ [0,1]} a^{1/3} C_{HP}C_2 \int_0^1 a
(w_x)^2 dx\\
&-s\frac{c_1c_2}{8\ve^{1/3}}c \int_0^1\Theta^3 \frac{|x-x_0|^2}{a}
w^2 dx.
\end{aligned}
\end{equation}

If $K>4/3$ consider the function $p(x) = (a(x)|x-x_0|^4)^{1/3}$. It
is clear that, setting
\[
C_1:=
\max\left\{\left(\frac{x_0^2}{a(0)}\right)^{2/3},\left(\frac{(1-x_0)^2}{a(1)}\right)^{2/3}\right\},
\]
by Lemma \ref{rem} we have
\[
p(x)=  a(x) \left(\frac{(x-x_0)^2}{a(x)}\right)^{2/3}\le C_1 a(x)
\]
and $\displaystyle \frac{a^{1/3}}{|x-x_0|^{2/3}}=
\frac{p(x)}{(x-x_0)^2}$. Moreover, using \eqref{dainfinito}, one has
that the function $\displaystyle x\mapsto \frac{p(x)}{|x-x_0|^q}$ is
nonincreasing on the left of $x=x_0$ and nondecreasing on the right
of $x=x_0$ for $\displaystyle q=\frac{4+\vartheta}{3}\in(1,2)$.
Thus, the Hardy-Poincar\'{e} inequality (see Proposition \ref{HP})
implies
\begin{equation}\label{hpappl}
\begin{aligned}
\int_0^1 \frac{a^{1/3}}{|x-x_0|^{2/3}}w^2dx &= \int_0^1
\frac{p}{(x-x_0)^2} w^2 dx \le C_{HP}\int_0^1  p (w_x)^2 dx
\\&\le C_{HP}C_1 \int_0^1 a (w_x)^2 dx,
\end{aligned}
\end{equation}
where $C_{HP}$ and $C_1$ are the Hardy-Poincar\'{e} constant and the
constant introduced before, respectively.

Using the estimates above, from \eqref{zio} we finally obtain
\begin{equation}\label{kmaggiore}
\begin{aligned}
\frac{s}{2}\int_{Q_T} \ddot{\Theta}\psi w^2dxdt &\geq
-s\ve\frac{3 c_1c_2}{8}cC_{HP}C_1\int_{Q_T} \Theta a
(w_x)^2 dxdt \\
&- s\frac{c_1c_2}{8\ve^3}c\int_0^T \int_0^1\Theta^3
\frac{(x-x_0)^2}{a}w^2 dxdt.
\end{aligned}
\end{equation}

Thus, in every case, we can choose $\ve$ so small and $s$ so large
that, by \eqref{Cappare}, \eqref{c18}, \eqref{zio}, \eqref{kminore}
and \eqref{kmaggiore}, we can estimate the distributed terms from
below with
\[
\begin{aligned}
s\frac{\bar C}{2} \int_{Q_T} \Theta a (w_x)^2 dxdt +
\frac{c_1^3}{2}s^3 \int_0^1\Theta^3 \frac{(x-x_0)^2}{a}w^2 dxdt.
\end{aligned}
\]
\end{proof}

As for the boundary terms, similarly to Lemma \ref{lemma4nd}, we
have the following calculation, whose proof parallels the one of
Lemma \ref{lemma4nd} and is thus omitted.
\begin{Lemma}\label{lemma4}
The boundary terms reduce to
\[
- sc_1 \int_0^{T} \left[ \Theta a(x-x_0)(w_{x})^2
\right]_{x=0}^{x=1}dt.
\]
\end{Lemma}

From Lemma \ref{lemma2}, and Lemma \ref{lemma4}, we deduce
immediately that there exist two positive constants $C$ and $s_0$,
such that all solutions $w$ of \eqref{1} satisfy, for all $s \ge
s_0$,
\begin{equation}\label{D&BT1}
\begin{aligned}
\int_0^T \int_0^1L^+_s w L^-_s w dxdt &\ge Cs\int_{Q_T} \Theta
a(w_x)^2 dxdt\\
&+ Cs^3 \int_{Q_T}\Theta^3 \frac{(x-x_0)^2}{a}
w^2 dxdt\\
&- sc_1 \int_0^{T} \left[ \Theta a(x-x_0)(w_{x})^2
\right]_{x=0}^{x=1}dt.
\end{aligned}
\end{equation}

Again, we immediately find
\begin{Proposition}\label{Carleman}
Assume Hypothesis $\ref{Ass02}$. Then, there exist two
positive constants $C$ and $s_0$, such that all solutions $w$ of
\eqref{1} in $\mathcal{S}_1$ satisfy, for all $s \ge s_0$,
\[
\begin{aligned}
   & s\int_{Q_T} \Theta a(w_x)^2dxdt  + s^3
\int_{Q_T}\Theta^3 \frac{(x-x_0)^2}{a} w^2 dxdt\\&\le
    C\left(\int_{Q_T} h^2 e^{2s\varphi}dxdt+ sc_1\int_0^{T}
    \left[a\Theta (x-x_0)
(w_{x})^{2}\right]_{x=0}^{x=1}dt
    \right).
    \end{aligned}
    \]
\end{Proposition}
Recalling the definition of $w$ and starting as in the end of the
proof of Theorem \ref{mono}, from Proposition \ref{Carleman} we
immediately obtain Theorem \ref{Cor1}.

\section{Carleman estimate for the problem in non divergence form}

Now, we consider the parabolic problem in non divergence form
\begin{equation}\label{1a}
\begin{cases}
v_t + av_{xx} =h & (t,x) \in Q_T,\\
v(t,0)=v(t,1)=0 &  t \in (0,T),\\
\end{cases}
\end{equation}
where $a$ satisfies one of the assumptions describing the (WD) or
the (SD) case, plus an additional condition, which we briefly recollect in the following
\begin{Assumptions}\label{Ass021}
The function $a$ satisfies Hypothesis $\ref{Ass0}$ or Hypothesis
$\ref{Ass01}$. Moreover,
\[
\frac{(x-x_0)a'(x)}{a(x)} \in
W^{1,\infty}(0,1),
\]
and if $K\geq \displaystyle \frac{1}{2}$ \eqref{dainfinito} holds.
\end{Assumptions}

\begin{Remark}
We underline the fact that in this subsection \eqref{Sigma}
is not necessary since all integrals and integrations by parts are
justified by the definition of $D(\mathcal A_2)$.
\end{Remark}

Now, we introduce the function $\varphi(t,x): =\Theta(t)\psi(x)$,
where $\Theta$ is defined as in \eqref{theta} and
\begin{equation}\label{c_11}
\psi(x) := d_1\left(\int_{x_0}^x \frac{y-x_0}{a(y)}e^{R(y-x_0)^2}dy-
d_2\right),
\end{equation}
with $R>0$, $d_2> \displaystyle
\max\left\{\frac{(1-x_0)^2e^{R(1-x_0)^2}}{(2-K)a(1)},
\frac{x_0^2e^{Rx_0^2}}{(2-K)a(0)}\right\}$ and $d_1>0$. A more
precise restriction on $d_1$ will be given in Chapter \ref{secobserv}, while the reason
for the choice of $d_2$ will be immediately clear: observe that, by
Lemma \ref{rem} and operating as in \eqref{defpsi}, we have that
\[
-d_1d_2\le\psi(x)<0 \quad \mbox{ for every }x\in[0,1].
\]

The basic result concerning Carleman estimates is the following
inequality, which is the counterpart of \cite[Theorem 3.1]{fm} or of Theorem \ref{Cor1} for
the non divergence case:
\begin{Theorem}\label{Cor11}
Assume Hypothesis  $\ref{Ass021}$. Then,
there exist two positive constants $C$ and $s_0$ such that every
solution $v$ of \eqref{1a} in
\begin{equation}\label{v1}
\mathcal{S}_2:=H^1\big(0, T;\cH^1_{\frac{1}{a}}(0,1)\big) \cap
L^2\big(0, T; \cH^2_{\frac{1}{a}}(0,1)\big)
\end{equation}
satisfies
\begin{equation}\label{car}
\begin{aligned}
&\int_{Q_T} \left(s\Theta (v_x)^2 + s^3 \Theta^3
\left(\frac{x-x_0}{a} \right)^2v^2\right)e^{2s\varphi}dxdt\\
&\le C\left(\int_{Q_T} h^{2}\frac{e^{2s\varphi}}{a}dxdt +
sd_1\int_0^T\left[\Theta e^{2s \varphi}(x-x_0)(v_x)^2
dt\right]_{x=0}^{x=1}\right)
\end{aligned}
\end{equation}
for all $s \ge s_0$, where $d_{1}$ is the constant introduced in
\eqref{c_11}.
\end{Theorem}

\subsection{Proof of Theorem $\ref{Cor11}$}

The proof of Theorem \ref{Cor11} follows the ideas of the proof of
\cite[Theorem 3.1]{fm} or Theorem \ref{Cor1}, but the non divergence
structure introduces several technicalities which were absent
before. We start as in the proof of Theorem \ref{mono1}: for
every $s> 0$ consider the function
\[
w(t,x) := e^{s \varphi (t,x)}v(t,x),
\]
where $v$ is any solution of \eqref{1a} in $\mathcal{S}_2$, so that
also $w\in\mathcal{S}_2$, since $v\in\mathcal{S}_2$ and $\varphi<0$.
Moreover, $w$ satisfies \eqref{1'nd1}, which we re--write as
\eqref{elles1}. Moreover, Lemma \ref{LEMMA1ND1} still holds also in
this case. We underline the fact that also in the degenerate case
all integrals and integrations by parts are justified by the
definition of $D(\mathcal A_2)$ and the choice of $\varphi$, while
in \cite{cfr} they were guaranteed by the choice of Dirichlet
boundary conditions at $x=0$, i.e. where their operator degenerates.
Thus we start with the analogue of Lemma \ref{lemma2ndWD1} in the
weakly and in the strongly degenerate case, which now gives the
following estimate:

\begin{Lemma}\label{lemma21}
Assume Hypothesis $\ref{Ass021}$. Then there exists a positive
constant $s_0$ such that for all $s \ge s_{0}$ the distributed terms
of \eqref{D&BTnd1} satisfy the estimate
\[
\begin{aligned}
&s\int_{Q_T}(a\varphi_{xx}+(a\varphi_x)_x)(w_x)^2dxdt + s^3
\int_{Q_T}(\varphi_x)^2(a\varphi_{xx}+(a\varphi_x)_x)w^2dxdt
\\
& -2 s^2\int_{Q_T} \varphi_x\varphi_{xt}w^2dxdt
 +\frac{s}{2}\int_{Q_T} \frac{\varphi_{tt}}{a}w^2dxdt
+s\int_{Q_T}(a\varphi_{xx})_{x}ww_x dxdt
\\&\ge
\frac{C}{2}s\int_{Q_T} \Theta (w_x)^2 dxdt +\frac{C^3}{2}s^3
\int_{Q_T}\Theta^3 \left(\frac{x-x_0}{a} \right)^2 w^2
dxdt,\end{aligned}
\]
for a universal positive constant $C$.
\end{Lemma}
\begin{proof}
Using the definition of $\varphi$, the distributed terms of
$\displaystyle\int_{Q_T} \frac{1}{a}L^+_s w L^-_s w dxdt$
take the form
\[
\begin{aligned}
&\frac{s}{2}\int_{Q_T} \ddot{\Theta} \frac{\psi}{a} w^2 dxdt
-2s^2\int_{Q_T}\Theta\dot{\Theta}(\psi')^2w^2 dxd+ s \int_{Q_T}\Theta (2a\psi''+ a'
\psi')(w_x)^2 dxdt\\
&+ s^3\int_{Q_T}\Theta^3(2a\psi''+ a' \psi')(\psi')^2w^2
dxdt+s\int_{Q_T} \Theta(a\psi '')' ww_x dxdt.
\end{aligned}
\]
Because of the choice of $\psi(x)$, one has \[\displaystyle 2a(x)
\psi''(x) + a'(x) \psi' (x)= d_1e^{R(x-x_0)^2}\frac{2a(x)-a'(x)
(x-x_0)+4R(x-x_0)^2a(x)}{a(x)}.
\]
By Hypothesis \eqref{Ass0} or \eqref{Ass01}, we immediately find
\[
 2-\frac{(x-x_0)a'}{a}\ge 2-K>0\quad \text{a.e. } \; x\in[0,1];
\]
hence, for every $R>0$ we get
\[
  2-\frac{(x-x_0)a'}{a}+ 4R(x-x_0)^2\ge 2-K \quad \text{a.e. } \; x\in[0,1].
\]

Thus, since $e^{R(x-x_0)^2}$ is bounded and bounded away from 0 in
$[0,1]$, the distributed terms satisfy the estimate
\begin{equation}\label{aaaaa}
\begin{aligned}
\{D.T.\} &\ge \frac{s}{2}\int_{Q_T} \ddot{\Theta}
\frac{\psi}{a} w^2 dxdt -s^2C\int_{Q_T}|\Theta\dot{\Theta}|\left(\frac{x-x_0}{a}\right)^2w^2
dxdt\\
&+ s C\int_{Q_T}\Theta (w_x)^2 dxdt+ s^3C\int_{Q_T}\Theta^3\left(\frac{x-x_0}{a} \right)^2w^2
dxdt\\
&+s\int_{Q_T} \Theta(a\psi '')' ww_x
dxdt,
\end{aligned}
\end{equation}
where $C>0$ denotes some universal positive constant which may vary
from line to line.

By \eqref{magtheta}, we conclude that, for $s$ large enough,
\[
\begin{aligned}
s^2C\int_{Q_T}|\Theta \dot{\Theta}|\left(\frac{x-x_0}{a}
\right)^2 w^2 dxdt&\le cC s^2
\int_{Q_T}\Theta^3\left(\frac{x-x_0}{a} \right)^2w^2 dxdt\\
&\le \frac{C^3}{8}s^3\int_{Q_T}\Theta^3 \left(\frac{x-x_0}{a}
\right)^2w^2 dxdt.
\end{aligned}
\]
Moreover, by \eqref{magtheta} we get
\begin{equation}\label{zio1}
\begin{aligned}
\left|  \frac{s}{2}\int_{Q_T} \ddot{\Theta}\frac{\psi}{a}
w^2dxdt \right| &\le \frac{s}{2}c \int_0^T
\int_0^1\Theta^{3/2} \frac{-\psi}{a}w^2 dxdt \\
&\leq s\frac{d_1d_2}{2}c\int_0^T \int_0^1\Theta ^{3/2} \frac{w^2}{a} dxdt,
\end{aligned}
\end{equation}
by the very definition of $\psi$. In order to estimate the last integral, we distinguish the cases $K<\frac{1}{2}$ and $K\geq \frac{1}{2}$. In the former case, using the Young inequality, we get
\begin{equation}\label{sopra1}
\begin{aligned} &s\frac{d_1d_2}{2}c \int_0^1
\Theta^{3/2}\frac{w^2}{a} dx \\&=
s\frac{d_1d_2}{2}c\left|\int_0^1\left(\Theta\frac{w^2}{a^{2/3}|x-x_0|^{2/3}}\right)^{3/4}\left(\Theta^3
\left(\frac{x-x_0}{a}\right)^2 w^2\right)^{1/4}\right| \\
&\le 3s\frac{d_1d_2}{8}c \int_0^1
\Theta\frac{w^2}{a^{2/3}|x-x_0|^{2/3}}dx +
 s\frac{d_1d_2}{8}c \int_0^1\Theta^3
\left(\frac{x-x_0}{a} \right)^2w^2 dx.
\end{aligned}
\end{equation}

Now, we introduce the function
\[
p(x) =\frac{|x-x_0|^{4/3}}{a^{2/3}}=\left(\frac{|x-x_0|^2}{a(x)}\right)^{2/3}\to 0 \mbox{ as $x\to0$ by Lemma \ref{rem}},
\]
and we take $q=\frac{4}{3}-\frac{2}{3}K$. Then the function
\[
x\mapsto \frac{p(x)}{|x-x_0|^q}=\left(\frac{|x-x_0|^{K}}{a}\right)^{2/3}
\]
is nonincreasing on the left of $x=x_0$ and nondecreasing on the right of $x=x_0$ by Lemma \ref{rem}. Since $K<1/2$, we have that $q\in (1,2)$. Thus, using
the Hardy-Poincar\'{e} inequality (see Proposition \ref{HP}), one has
\begin{equation}\label{hpappl1}
\begin{aligned}
\int_0^1 \frac{w^2}{a^{2/3}|x-x_0|^{2/3}}dx &=\int_0^1 \frac{p(x)}{|x-x_0|^2}w^2dx \\
&\le C_{HP}\int_0^1 p (w_x)^2 dx \\
&\le  C_{HP}\max\left\{\frac{x_0^{4/3}}{a(0)^{2/3}},\frac{|1-x_0|^{4/3}}{a(1)^{2/3}}\right\} \int_0^1 (w_x)^2 dx,
\end{aligned}
\end{equation}
by Lemma \ref{rem}. Thus, by \eqref{sopra1} and \eqref{hpappl1}, we have that for $s$
large enough
\begin{equation}\label{quest}
\begin{aligned}
s\frac{d_1d_2}{2}c \int_0^1 \Theta^{3/2} \frac{w^2}{a}  dx \le
\frac{C}{4}s\int_0^1 \Theta  (w_x)^2 dx + \frac{C^3}{8}s^3
\int_0^1\Theta^3 \left(\frac{x-x_0}{a}\right)^2w^2 dx,
\end{aligned}
\end{equation}
for a positive constant $C$. Using \eqref{quest}, from
\eqref{zio1} we finally obtain
\begin{equation}\label{quasfin}
\begin{aligned}
\left|  \frac{s}{2}\int_{Q_T} \ddot{\Theta}\frac{\psi}{a}
w^2dxdt \right| &\le \frac{C}{4}s\int_{Q_T} \Theta  (w_x)^2
dxdt
\\&+ \frac{C^3}{4}s^3\int_{Q_T}\Theta^3
\left(\frac{x-x_0}{a}\right)^2w^2 dxdt.
\end{aligned}
\end{equation}

If $K\geq \displaystyle\frac{1}{2}$ we proceed as follows. We take $r>2$, $\gamma<2$ and $\alpha,\beta>0$ to be chosen later, and, by \eqref{magtheta} and the Young inequality, we get
\begin{equation}\label{primostep}
\begin{aligned}& \int_0^1
\Theta^{3/2}\frac{w^2}{a} dx \leq c \int_0^1\left(\Theta\frac{w^2}{x^2}\right)^{1/r}\left(\Theta^\alpha
\frac{|x-x_0|^\beta}{a^\gamma} w^2\right)^{1-1/r}dx\\
& \leq \frac{c}{r} \int_0^1\Theta\frac{w^2}{x^2}dx+\frac{r}{r-1}\int_0^1\Theta^\alpha
\frac{|x-x_0|^\beta}{a^\gamma} w^2dx,
\end{aligned}
\end{equation}
which holds true provided that:
\[
\alpha=\frac{3r-2}{2r-2}(< 3\mbox{ since }r>2),
\]
and
\[
\frac{1}{a}\leq \bar{c}\frac{1}{x^{2/r}}\frac{|x-x_0|^{\beta\frac{r-1}{r}}}{a^{\gamma\frac{r-1}{r}}},
\]
which, by \eqref{dainfinito}, is true if
\begin{equation}\label{ga2}
\gamma>\frac{r}{r-1}
\end{equation}
and
\[
\frac{\beta(r-1)}{\gamma(r-1)-r}\leq \vartheta.
\]
Notice that \eqref{ga2} is consistent with the requirement $\gamma<2$ since $r>2$.

Moreover, we shall clearly use the inequality
\[
\frac{|x-x_0|^\beta}{a^\gamma}\leq \tilde{c}\left(\frac{|x-x_0|}{a}\right)^2,
\]
which is true if
\[
\frac{2-\beta}{2-\gamma}\leq \vartheta.
\]
Hence, choosing $\beta$ and $\gamma$ satisfying the conditions above, from \eqref{primostep} and the classical Hardy inequality (recall that $\mathcal{H}_{\frac{1}{a}}(0,1)\subset H^1_0(0,1)$) we get
\begin{equation}\label{secondostep}
\int_0^1
\Theta^{3/2}\frac{w^2}{a} dx \leq c_1\int_0^1\Theta (w_x)^2dx+c_2\int_0^1\Theta^3\left(\frac{x-x_0}{a}\right)^2w^2dx
\end{equation}for some universal constant $c_1,c_2>0$. Hence, as before, \eqref{quasfin} also holds in this case, if $s$ is large enough.

Now, we consider the last term in \eqref{aaaaa}, i.e. $s\int_{Q_T} \Theta(a\psi '')' ww_x dxdt$.
Observe that, using the definition of $\psi$ and Hypothesis
\ref{Ass021}, we have
\[
\begin{aligned}
\|(a\psi '')' \|_{L^\infty(0,1)}&\le d_1e^R\left(4R^2+6R+2R\left\|\frac{(x-x_0)a'}{a}\right\|_{L^\infty(0,1)}\right.\\
&\left.+
\left\|\left(\frac{(x-x_0)a'}{a}\right)'\right\| _{L^\infty(0,1)}
\right):=C_R.
\end{aligned}
\]
Hence, proceeding as for \eqref{quest}, one has
\[
\begin{aligned}
& \left|s\int_{Q_T}\Theta (a\psi '')'  ww_x dxdt\right| \le \frac{1}{2} s\int_{Q_T}\Theta |(a\psi '')' |^2 w^2dxdt\\
& + \frac{1}{2}  s\int_{Q_T}\Theta  (w_x)^2 dxdt\\
& \le  \frac{1}{2}  s c\|(a\psi '')' \|^2_{L^\infty(0,1)}\int_{Q_T}\Theta^{3/2} w^2dxdt + \frac{1}{2}  s\int_{Q_T}\Theta (w_x)^2dxdt \\
& \le   \frac{C}{4}s\int_{Q_T} \Theta  (w_x)^2 dxdt +   s^3
\frac{C^3}{8} \int_{Q_T}\Theta^3
\left(\frac{x-x_0}{a}\right)^2w^2 dxdt.
\end{aligned}
\]
Summing up, we obtain
\[
\begin{aligned}
\{D.T.\}&\ge -\frac{C}{4}s\int_{Q_T} \Theta  (w_x)^2 dxdt -
\frac{C^3}{4}s^3\int_{Q_T}\Theta^3
\left(\frac{x-x_0}{a}\right)^2w^2 dxdt \\
& -\frac{C^3}{8}s^3\int_{Q_T}\Theta^3 \left(\frac{x-x_0}{a}
\right)^2w^2 dxdt
\\&
+ s C\int_{Q_T}\Theta (w_x)^2 dxdt+ s^3C\int_{Q_T}\Theta^3\left(\frac{x-x_0}{a} \right)^2w^2 dxdt\\&
-\frac{C}{4}s\int_{Q_T} \Theta  (w_x)^2 dxdt -
\frac{C^3}{8}s^3\int_{Q_T}\Theta^3(w_x)^2 dxdt
\\&
= \frac{C}{2}s\int_{Q_T} \Theta (w_x)^2 dxdt +\frac{C^3}{2}s^3
\int_{Q_T}\Theta^3 \left(\frac{x-x_0}{a} \right)^2 w^2 dxdt.
\end{aligned}
\]
\end{proof}

As for the boundary terms, similarly to Lemma \ref{lemma4nd1}, we
have the following result, whose proof parallels the one of Lemma
\ref{lemma4nd1} and is thus omitted.
\begin{Lemma}\label{lemma41}
Assume Hypothesis $\ref{Ass021}$. Then the boundary terms in
\eqref{D&BTnd1} reduce to
\[-sd_1\int_0^T\Theta(t)\Big[(x-x_0)e^{R(x-x_0)^2}(w_x)^2\Big]_{x=0}^{x=1}dt.
\]
\end{Lemma}

By Lemmas \ref{lemma21} and \ref{lemma41}, there exist $C>0$ and
$s_0>0$ such that all solutions $w$ of \eqref{1'nd1} satisfy, for
all $s \ge s_0$,
\begin{equation}\label{D&BT11}
\begin{aligned}
\int_{Q_T} \frac{1}{a}L^+_s w L^-_s w dxdt &\ge
Cs\int_{Q_T} \Theta (w_x)^2 dxdt\\&+ Cs^3
\int_{Q_T}\Theta^3 \left(\frac{x-x_0}{a}\right)^2 w^2 dxdt
\\&-sd_1\int_0^T\Theta(t)\Big[(x-x_0)e^{R(x-x_0)^2}(w_x)^2\Big]_{x=0}^{x=1}dt.
\end{aligned}
\end{equation}

Thus, by \eqref{stimettanda} and \eqref{D&BT11}, we obtain the next
Carleman inequality for $w$:
\[
\begin{aligned}
& s\int_{Q_T} \Theta (w_x)^2dxdt  + s^3
\int_{Q_T}\Theta^3 \left(\frac{x-x_0}{a}\right)^2w^2
dxdt\\
&\le C\left(\int_{Q_T} h^2 \frac{e^{2s\varphi}}{a}dxdt+
sd_1\int_0^{T} \left[\Theta
(x-x_0)e^{R(x-x_0)^2}(w_x)^2\right]_{x=0}^{x=1}dt \right)
\end{aligned}
\]
for all $s \ge s_0$.

Theorem \ref{Cor11} follows recalling the definition of $w$ and starting as in the end of the
proof of Theorem \ref{mono}.

\chapter{Observability inequalities and application to null
controllability}\label{secobserv}

In this chapter we assume that the control set $\omega$ satisfies
the following assumption:
\begin{Assumptions} \label{ipotesiomega}
The subset $\omega$ is such that
\begin{itemize}
\item it is an interval which contains the degeneracy point, more precisely:
\begin{equation}\label{omega1}
\omega=(\alpha,\beta) \subset (0,1) \mbox{ is such that $x_0 \in
\omega$}.
\end{equation}
or
\item it is an interval lying on one
side of the degeneracy point, more precisely:
\begin{equation}\label{omega}
\omega=(\alpha,\beta) \subset (0,1) \mbox{ is such that $x_0\not \in
\bar \omega$}.
\end{equation}
\end{itemize}
\end{Assumptions}
\section{The divergence case}
Now, we consider the problem in divergence form and we make the
following assumptions on the function $a$:
\begin{Assumptions}\label{Ass03}
Hypothesis $\ref{Ass02}$ is satisfied. Moreover,
if Hypothesis $\ref{Ass0}$ holds, we assume that there exist
two functions $\fg \in L^\infty_{\rm loc}([0,1]\setminus \{x_0\})$, $\fh \in W^{1,\infty}_{\rm loc}([0,1]\setminus \{x_0\}, L^\infty(0,1))$ and
two strictly positive constants $\fg_0$, $\fh_0$ such that $\fg(x) \ge \fg_0$ for a.e. $x$ in $[0,1]$ and
\begin{equation}\label{rieccola}
-\frac{a'(x)}{2\sqrt{a(x)}}\left(\int_x^B\fg(t) dt + \fh_0 \right)+ \sqrt{a(x)}\fg(x) =\fh(x,B)\quad \text{for a.e.} \; x,B \in [0,1]
\end{equation}
with $x<B<x_0$ or $x_0<x<B$.
\end{Assumptions}

\begin{Remark}\label{quellali}
Contrary to the non degenerate case, the identity in \eqref{rieccola} is assumed to hold with functions which are bounded only {\em far from} $x_0$. Indeed, \eqref{rieccola} will be applied in sets where $a$ is non degenerate and the corresponding identity given in Hypothesis \ref{ipoadebole} will be applied. For this reason, functions $\fg$ and $\fh$ can be easily found, once $a$ is given.
\end{Remark}

To the linear problem \eqref{linear} we associate the
homogeneous adjoint problem
\begin{equation}\label{h=0}
\begin{cases}
v_t +(av_x)_x= 0, &(t,x) \in  Q_T,
\\[5pt]
v(t,0)=v(t,1) =0, & t \in (0,T),
\\[5pt]
v(T,x)= v_T(x)\in L^2(0,1),
\end{cases}
\end{equation}
where $T>0$ is given. By the Carleman estimate in Theorem
\ref{Cor1}, we will deduce the following observability inequality
for both the weakly and the strongly degenerate cases:
\begin{Proposition}\label{obser.}
Assume Hypotheses $\ref{ipotesiomega}$ and $\ref{Ass03}$. Then there exists a positive constant $C_T$
such that every solution $v \in  C([0, T]; L^2(0,1)) \cap L^2 (0,T;
{\cal H}^1_a(0,1))$ of \eqref{h=0} satisfies
 \begin{equation}\label{obser1.}
\int_0^1v^2(0,x) dx \le C_T\int_0^T \int_{\omega}v^2(t,x)dxdt.
\end{equation}
\end{Proposition}

Using the observability inequality \eqref{obser1.} and a standard
technique (e.g., see \cite[Section 7.4]{LRL}), one can prove the
null controllability result for the linear degenerate problem
\eqref{linear}, another fundamental result of this paper.
\begin{Theorem}\label{th3}
Assume Hypotheses $\ref{ipotesiomega}$ and $\ref{Ass03}$. Then, for
every $u_0 \in L^2 (0,1)$ there exists $h \in L^2(Q_T)$ such that
the solution $u$ of \eqref{linear} satisfies
\begin{equation*}
u(T,x)= 0 \ \text{ for every  } \  x \in [0, 1].
\end{equation*}
Moreover
\[
\int_0^T \int_0^1 h^2 dxdt \le C \int_0^1 u_0^2 dx,
\]
for some universal positive constant $C$.
\end{Theorem}

We remark that Proposition \ref{obser.} has an immediate
application also in the case in which the control set $\omega$ is
the union of two intervals $\omega_i$, $i=1,2$ each of them lying on
one side of the degeneracy point. More precisely, we have the
following observability inequality, whose proof is straightforward:
\begin{Corollary}\label{corollario}
Assume Hypothesis $\ref{Ass03}$ and $\omega = \omega_1 \cup
\omega_2,$ where $\omega_i$, $i=1,2$ are intervals each of them
lying on one side of the degeneracy point, more precisely:
\begin{equation}\label{omega2}
\omega_i=(\lambda_i,\beta_i) \subset (0,1), \, i=1,2, \mbox{ and
$\beta_1 < x_0< \lambda_2$}.
\end{equation}
Then there exists a positive constant $C_T$ such that every solution
$v$ of \eqref{h=0} satisfies
\[
\int_0^1v^2(0,x) dx \le C_T\int_0^T \int_{\omega}v^2(t,x)dxdt.
\]
\end{Corollary}

As a consequence, one has the next null controllability result:
\begin{Theorem}\label{th3'}
Assume $\eqref{omega2}$ and Hypothesis $\ref{Ass03}$. Then, for
every $u_0 \in L^2 (0,1)$, there exists $h \in L^2(Q_T)$ such that
the solution $u$ of \eqref{linear} satisfies
\begin{equation*}
u(T,x)= 0 \ \text{ for every  } \  x \in [0, 1].
\end{equation*}
Moreover
\[
\int_0^T \int_0^1 h^2 dxdt \le C \int_0^1 u_0^2(x) dx,
\]
for some positive constant $C$.
\end{Theorem}

\subsection{Proof of Proposition \ref{obser.}} In
this subsection we will prove, as a consequence of the Carleman
estimate established in Theorem \ref{Cor1}, the
observability inequality \eqref{obser1.}. For this purpose, we will
give some preliminary results. As a first step, we consider the
adjoint problem with more regular final--time datum
\begin{equation}\label{h=01}
\begin{cases}
v_t +(av_x)_{x}= 0, &(t,x) \in  Q_T,
\\[5pt]
v(t,0)=v(t,1) =0, & t \in (0,T),
\\[5pt]
v(T,x)= v_T(x) \,\in D({\cal A}_1^2),
\end{cases}
\end{equation}
where
\[
D({\cal A}_1^2) = \Big\{u \,\in \,D({\cal A}_1)\;\big|\; {\cal A}_1u \,\in
\,D({\cal A}_1) \;\Big\}
\]
and, we recall, ${\cal A}_1u:=(au_x)_x$. Observe that $D({\cal A}_1^2)$ is densely
defined in $D({\cal A}_1)$ (see, for example, \cite[Lemma 7.2]{b}) and
hence in $L^2(0,1)$. As in \cite{cfr}, \cite{cfr1} or \cite{f},
letting $v_T$ vary in $D({\cal A}_1^2)$, we define the following class
of functions:
\[
\cal{W}_1:=\Big\{v \in
C^1\big([0,T]; L^2(0,1)\big)\cap C\big([0,T]; D(\mathcal
A_1)\big)\,|\, v\text{ is a solution of \eqref{h=01}}\Big\}.
\]
Obviously (see, for example, \cite[Theorem 7.5]{b})
\[ \cal{W}_1\subset
C^1\big([0,T];\:{\cal H}^2_a(0,1)\big) \subset \mathcal{S}_1 \subset
\cal{U}_1,
\]
where, $\mathcal{S}_1$ is defined in \eqref{v} and
\[
\cal{U}_1:= C([0,T]; L^2(0,1)) \cap L^2(0, T; {\cal H}^1_a(0,1)).
\]

We start with the following Proposition, for whose proof we refer to
\cite[Proposition 4.2]{fm}, since also in this weaker setting that proof is still valid. We underline the fact that the
degeneracy point is allowed to belong even to the control set.

\begin{Proposition}[Caccioppoli's inequality]\label{caccio}
Assume Hypothesis $\ref{Ass0}$ or Hypothesis $\ref{Ass01}$. Let
$\omega'$ and $\omega$ two open subintervals of $(0,1)$ such that
$\omega'\subset \subset \omega \subset  (0,1)$ and $x_0 \not \in
\bar\omega'$. Let $\varphi(t,x)=\Theta(t)\Upsilon(x)$, where
$\Theta$ is defined in \eqref{theta} and
\[
\Upsilon \in C([0,1],(-\infty,0))\cap
C^1([0,1]\setminus\{x_0\},(-\infty,0))
\]
is such that
\begin{equation}\label{stimayx}
|\Upsilon_x|\leq \frac{c}{\sqrt{a}} \mbox{ in }[0,1]\setminus\{x_0\}
\end{equation}
for some $c>0$. Then, there exist two positive constants $C$ and
$s_0$ such that every solution $v \in \cal W_1$ of the adjoint problem
\eqref{h=01} satisfies
\[
   \int_{0}^T \int _{\omega'}   (v_x)^2e^{2s\varphi } dxdt
    \ \leq \ C \int_{0}^T \int _{\omega}   v^2  dxdt,
\]
for all $s\geq s_0$.
\end{Proposition}

\begin{Remark}\label{prototipo}
Of course, our prototype for $\Upsilon$ are the functions $\psi$
defined in \eqref{c_1nd} or in \eqref{c_1}. Indeed, if $\psi$ is as
in \eqref{c_1}, then
\[
|\psi'(x)|=c_1\frac{|x-x_0|}{a(x)}=c_1\sqrt{\frac{|x-x_0|^2}{a(x)}}\frac{1}{\sqrt{a(x)}}\leq
c\frac{1}{\sqrt{a(x)}}
\]
by Lemma \ref{rem}. In the case of \eqref{c_1nd}, inequality
\eqref{stimayx} is obvious.
\end{Remark}

\begin{Remark}
Actually, in the proof of Proposition \ref{caccio}, only the
regularity on $a$ required in Hypothesis $\ref{Ass0}$ or Hypothesis
$\ref{Ass01}$ is used, and not the inequality $(x-x_0)a'\leq Ka$.
\end{Remark}

We shall also need the two following lemmas, that deal with the
different situations in which $x_0$ is inside or outside the control
region $\omega$. The statements of the conclusions are in fact the
same, however, we state the results in two separate lemmas, since
their applications are related to different situations and the
proofs, though inspired by the same ideas, are different. First, we
state both the results and then we will prove them.

\begin{Lemma}\label{lemma3}
Assume  \eqref{omega1} and Hypothesis $\ref{Ass03}$. Then there exist
two positive constants $C$ and $s_0$ such that every solution $v \in
\cal W_1$ of \eqref{h=01} satisfies, for all $s \ge s_0$,
\[
\int_{Q_T}\left( s \Theta a (v_x)^{2} + s^3 \Theta ^3
\frac{(x-x_0)^2}{a} v^{2}\right) e^{{2s\varphi}}  dxdt\le C
\int_0^T\int_{\omega}v^{2} dxdt.
\]
Here $\Theta$ and $\varphi$ are as in \eqref{theta} and \eqref{c_1},
respectively, with $c_1$ sufficiently large.
\end{Lemma}

\begin{Lemma}\label{lemma3'}
Assume \eqref{omega} and Hypothesis $\ref{Ass03}$. Then there exist
two positive constants $C$ and $s_0$ such that every solution $v \in
\cal W_1$ of \eqref{h=01} satisfies, for all $s \ge s_0$,
\[
\int_{Q_T}\left( s \Theta a (v_x)^{2} + s^3 \Theta ^3
\frac{(x-x_0)^2}{a} v^{2}\right) e^{{2s\varphi}}  dxdt\le C
\int_0^T\int_{\omega}v^{2} dxdt.
\]
Here $\Theta$ and $\varphi$ are as in \eqref{theta} and \eqref{c_1},
respectively, with $c_1$ sufficiently large.
\end{Lemma}

We underline the fact that for the proof of the previous lemmas a
crucial r\^{o}le will be played also by the Carleman estimate for
nondegenerate equations with nonsmooth coefficient proved in Theorem
\ref{mono}.

\begin{proof}[Proof of Lemma $\ref{lemma3}$]
By assumption, we can find two subintervals
$\omega_1=(\lambda_1,\beta_1)\subset (0, x_0),
\omega_2=(\lambda_2,\beta_2) \subset (x_0,1)$ such that $(\omega_1
\cup \omega_2) \subset \subset \omega \setminus \{x_0\}$. Now,
consider a smooth function $\xi:[0,1]\to[0,1]$ such that
\[
\xi(x)=\begin{cases} 0& x\in [0,\alpha],\\
1 & x\in [\lambda_1,\beta_2]\\
 0 &x\in [\beta,1],
\end{cases}
\]
and define $w:= \xi v$, where $v$ solves \eqref{h=01}.
Hence, $w$ satisfies
\begin{equation}\label{eq-w*}
\begin{cases}
w_t + (a  w_x) _x =( a \xi _x v )_x + \xi _x a v_x =:f,&
(t,x) \in(0, T)\times (0,1), \\
w(t,0)= w(t,1)=0, & t \in (0,T).
\end{cases}
\end{equation}
Applying Theorem \ref{Carleman} and using the fact that $w=0$ in a
neighborhood of $x=0$  and $x=1$, we have
\begin{equation}\label{car9}
\int_{Q_T}\Big( s \Theta  a (w_x)^2 + s^3 \Theta^3
   \frac{(x-x_0)^2}{a} w^2 \Big)
    e^{2s \varphi} \, dx dt
       \le C\int_{Q_T}e^{2s \varphi} f^2  dxdt
\end{equation}
for all $s \ge s_0$. Then, using the definition of $\xi$  and in
particular the fact that  $\xi_x$ and  $\xi_{xx}$ are supported
inside $\tilde \omega:= [\alpha, \lambda_1] \cup[ \beta_2, \beta]$, from Hypothesis \ref{Ass03}
we can write
\[
f^2= (( a \xi_x v )_x + a\xi_x  v_x)^2 \le C(
v^2+ (v_x)^2)\chi_{\tilde \omega}.
\]

Hence, applying Proposition \ref{caccio} and inequality
\eqref{car9}, we get
 \begin{equation}\label{stimacar}
\begin{aligned}
&\int_0^T\int_{\lambda_1}^{\beta_2}\left( s \Theta a (v_x)^{2} + s^3
\Theta ^3 \frac{(x-x_0)^2}{a} v^{2}\right) e^{{2s\varphi}} dxdt\\
&=\int_0^T\int_{\lambda_1}^{\beta_2}\Big( s \Theta a
(w_x)^2 + s^3 \Theta^3 \frac{(x-x_0)^2}{a} w^2 \Big)
e^{2s \varphi} \, dx dt\\
&\le \int_0^T \int_0^1\Big( s \Theta  a (w_x)^2 +
s^3 \Theta^3
\frac{(x-x_0)^2}{a} w^2 \Big) e^{2s \varphi} \, dx dt\\
& \le C  \int_0^T \int_{\tilde \omega}e^{2s \varphi}( v^2+
(v_x)^2)dxdt \le C \int_0^T \int_{\omega} v^2dxdt,
\end{aligned}
\end{equation}
for a positive constant $C$.

Now, we consider a smooth function $\tau: [0,1] \to [0,1]$ such that
\[
\tau (x)=
\begin{cases}
0 &     x \in \left[0,\frac{\lambda_2+ \beta_2}{2}\right
],\\
1& x\in[\beta_2 , 1].
\end{cases}\]

Define $z:= \tau v$, where $v$ is the solution of \eqref{h=01}. Then
$z$ satisfies \eqref{1}, with $h:=( a \tau_x v)_x + a\tau_x v_x$, $A= \lambda_2$ and $B=1$. Since $h$ is supported in
$\left[\frac{\lambda_2+ \beta_2}{2}, \beta_2\right]$, by Proposition
\ref{caccio}, Theorem \ref{mono} applied with $A=\lambda_2$, $B=1$
and Remark \ref{RemCarleman},
we get
\begin{equation}\label{lun}
\begin{aligned}
&\int_0^T\int_{\lambda_2}^1 s \Theta (z_x)^2e^{2s\Phi} dxdt
+\int_0^T\int_{\lambda_2}^1 s^3\Theta^3 z^2e^{2s\Phi} dxdt\\& \le c
\int_0^T\int_{\lambda_2}^1e^{2s\Phi} h^2 dxdt\le C \int_0^T
\int_{\tilde \omega_1} v^2dxdt + C \int_0^T \int_{\tilde
\omega_1}e^{2s\Phi}(v_x)^2dxdt\\&\le C \int_0^T \int_{\omega}
v^2dxdt,
\end{aligned}
\end{equation}
where $\tilde \omega_1 =(\lambda_2, \beta_2)$. Let us remark that
the boundary term in $x=1$ is nonpositive, while the one in
$x=\lambda_2$ is 0, so that they can be neglected in the classical
Carleman estimate.

Now, choose the constant $c_1$ in \eqref{c_1} so
that
\begin{equation}\label{c_1magg'}
c_1 \ge \begin{cases} \frac{\displaystyle  r\left[\int_
{\lambda_2}^1 \frac{1}{\sqrt{a(t)}} \int_t^1 \fg(s) dsdt +
\int_{\lambda_2}^1 \frac{\fh_0}{\sqrt{a(t)}}dt\right]+
\mathfrak{c}}{\displaystyle c_2-\frac{(1-x_0)^2}{a(1)(2-K)}}=:\Pi& \mbox{
in the (WD) case},\\
\frac{\displaystyle \mathfrak{c}-1}{\displaystyle
c_2-\frac{(1-x_0)^2}{a(1)(2-K)}} & \mbox{ in the (SD) case},
\end{cases}
\end{equation}
where $\mathfrak{c}$ is the constant appearing in \eqref{c_1nd}.
Then, by definition of $\varphi$, the choice of $c_1$ and by Lemma
\ref{rem}, one can prove that there exists a positive constant $k$,
for example
\[k = \max \left\{\max_{\left[\lambda_2,
1\right]}a,\frac{(1-x_0)^2}{a(1)}\right\},\] such that
\begin{equation}\label{prima'}
a(x) e^{2s\varphi(t,x)} \le k e^{2s\Phi(t,x)}
\end{equation}
and
\begin{equation}\label{seconda'}
\frac{(x-x_0)^2}{a(x)}e^{2s\varphi(t,x)} \le k e^{2s \Phi(t,x)}
\end{equation}
for every $(t,x) \in [0, T] \times \left[\lambda_2, 1\right]$. Note
that the value of $k$ can be immediately found by estimating the
coefficients of $e^{2s\varphi(t,x)}$ in \eqref{prima'} and
\eqref{seconda'}, once known that $e^{2s\varphi(t,x)}\leq
e^{2s\Phi(t,x)}$, using Lemma \ref{rem}. Finally, condition
\eqref{c_1magg'} is a sufficient one to get $e^{2s\varphi(t,x)}\leq
e^{2s\Phi(t,x)}$, and it can be found by using Lemma \ref{rem} and
rough estimates.

Thus, by \eqref{lun}, one has
\[
\begin{aligned}
&\int_0^T\int_{\lambda_2}^1 \Big(s \Theta a (z_x)^2
 + s^3 \Theta^3
      \frac{(x-x_0)^2}{a}z^2\Big) e^{2s\varphi}dxdt \\&
\le k \int_0^T\int_{\lambda_2}^1 s \Theta (z_x)^2e^{2s\Phi} dxdt
+k\int_0^T\int_{\lambda_2}^1 s^3\Theta^3 z^2e^{2s\Phi} dxdt\\&
      \le kC \int_0^T
\int_{\omega} v^2dxdt,
\end{aligned}
\]
for a positive constant $C$.
As a trivial consequence,
\begin{equation}\label{stimacar2}
\begin{aligned}
&\int_0^T\int_{\beta_2}^1 \Big(s \Theta a (v_x)^2 + s^3 \Theta^3
\frac{(x-x_0)^2}{a}v^2\Big) e^{2s\varphi} dxdt\\
&=\int_0^T\int_{\beta_2}^1 \Big(s \Theta a
(z_x)^2 + s^3 \Theta^3
\frac{(x-x_0)^2}{a}z^2\Big) e^{2s\varphi}dxdt\\
&\le \int_0^T\int_{\lambda_2}^1 \Big(s \Theta a
(z_x)^2 + s^3
\Theta^3\frac{(x-x_0)^2}{a}z^2\Big) e^{2s\varphi}dxdt\\
&\le \int_0^T \int_{\omega} v^2dxdt,
\end{aligned}
\end{equation}
for a positive constant $C$.

Thus \eqref{stimacar} and \eqref{stimacar2} imply
    \begin{equation}\label{carin0}
    \begin{aligned}
      \int_{0}^T \int _{\lambda_1}^{1}  \Big( s \Theta  a (v_x)^2 + s^3 \Theta^3 \frac{(x-x_0)^2}{a}  v^2 \Big)
      e^{2s \varphi } \, dx dt
     \le  C \int_{0}^T \int _{\omega}   v^2  dxdt,
    \end{aligned}
\end{equation}
for some positive constant $C$.

To complete the proof it is sufficient to prove a similar inequality
on the interval $[0,\lambda_1]$. To this aim, we perform a
reflection procedure introducing the functions
\begin{equation}\label{W}
W(t,x):= \begin{cases} v(t,x), & x \in [0,1],\\
-v(t,-x), & x \in [-1,0],
\end{cases}
\end{equation}
where $v$ solves \eqref{h=01}, and
\begin{equation}\label{tildea}
\tilde a(x):= \begin{cases} a(x), & x \in [0,1],\\
a(-x), & x \in [-1,0].
\end{cases}
\end{equation}
Then $W$ satisfies the problem
\begin{equation}\label{dispari}
\begin{cases}
W_t +(\tilde a W_x)_{x}= 0, &(t,x) \in  (0,T)\times (-1,1),
\\[5pt]
W(t,-1)=W(t,1) =0, & t \in (0,T).
\end{cases}
\end{equation}

As above, we introduce a smooth function $\rho:[-1,1]\to[0,1]$ such that
\[
\rho (x)=\begin{cases}
0& x\in
[-1,-\frac{\lambda_1+\beta_1}{2}],\\
1& x\in[-\lambda_1,\lambda_1],\\
0& \in [\frac{\lambda_1+\beta_1}{2},1].
\end{cases}
\]

Finally, set $Z:= \rho W$, where $W$ is the solution of
\eqref{dispari}. Then $Z$ satisfies \eqref{1}  with $h:=( \tilde a
\rho_x W)_x + \tilde a\rho_x W_x$.
Observe that $Z_x(t, -\beta_1)=Z_x(t, \beta_1)=0$.
Using Proposition \ref{caccio}, Theorem \ref{mono} with  $A=
-\beta_1$ and $B=\beta_1$, Remark \ref{RemCarleman},
 the definition
of $W$ and the fact that $h$ is supported in
$\left[-\frac{\lambda_1+ \beta_1}{2},-\lambda_1\right]
\cup\left[\lambda_1, \frac{\lambda_1+ \beta_1}{2}\right]$ give
\begin{equation}\label{lunga}
\begin{aligned}
&\int_0^T\int_{-\beta_1}^{\beta_1} s \Theta (Z_x)^2e^{2s\Phi} dxdt
+\int_0^T\int_{-\beta_1}^{\beta_1} s^3\Theta^3 Z^2e^{2s\Phi} dxdt \\
& \le C
\int_0^T\int_{-\beta_1}^{\beta_1}e^{2s\Phi} h^2 dxdt \\
&\le C \int_0^T
\int_{-\frac{\lambda_1+\beta_1}{2}}^{-\lambda_1}e^{2s\Phi}( W^2+
(W_x)^2)dxdt + C\int_0^T \int_{\lambda_1}^{\frac{\lambda_1+
\beta_1}{2}}e^{2s\Phi}(W^2+ (W_x)^2)dxdt\\
&\le C \int_0^T \int_{\lambda_1}^{\frac{\lambda_1+
\beta_1}{2}}e^{2s\Phi}( W^2+( W_x)^2)dxdt\\
& \leq C \int_0^T \int_{\lambda_1}^{\frac{\lambda_1+ \beta_1}{2}}
v^2dxdt+  C \int_0^T \int_{\lambda_1}^{\frac{\lambda_1+
\beta_1}{2}}e^{2s\Phi}(v_x)^2dxdt \le C \int_0^T \int_{\omega}
v^2dxdt,
\end{aligned}
\end{equation}
for some positive constants $C$, which we allow to vary from line to
line.
Now, define
\[
\tilde \varphi(t,x) := \Theta(t) \tilde \psi (x),
\]
where
\begin{equation}\label{tildepsi}
\tilde \psi(x) := \begin{cases}
\psi(x), & x \ge 0,\\
\displaystyle \psi(-x)= c_1\left[\int_{-x_0}^x \frac{t+x_0}{\tilde
a(t)}dt -c_2\right], & x <0.\end{cases}
\end{equation}
and choose the constant $c_1$ so that
\[
c_1 \ge \begin{cases}\max\left\{ \Pi,
 \frac{\displaystyle  r\left[\int_ {-\beta_1}^{\beta_1} \frac{1}{\sqrt{a(t)}}
\int_t^1 \fg(s) dsdt +
\int_{-\beta_1}^{\beta_1}\frac{\fh_0}{\sqrt{a(t)}}dt\right]+
\mathfrak{c}}{\displaystyle c_2-\frac{x_0^2}{a(0)(2-K)}}\right\} &
\mbox{ in the (WD) case},\\
\max\left\{\frac{\displaystyle \mathfrak{c}-1}{\displaystyle
c_2-\frac{(1-x_0)^2}{a(1)(2-K)}}, \frac{\displaystyle
\mathfrak{c}-1}{\displaystyle
c_2-\frac{x_0^2}{a(0)(2-K)}}\right\} & \mbox{ in the (SD) case}.
 \end{cases}
\]
Thus, by definition of $\tilde \varphi$, one can prove as before
that there exists a positive constant $k$, for example
\[
k = \max \left\{\max_{\left[-\beta_1, \beta_1\right]}\tilde
a,\frac{(x_0)^2}{ a(0)}\right\},
\]
such that
\[
\tilde a(x) e^{2s\tilde\varphi(t,x)} \le k e^{2s\Phi(t,x)}
\]and
\[ \frac{(x-x_0)^2}{\tilde a(x)}e^{2s\tilde\varphi(t,x)} \le k
e^{2s \Phi(t,x)}
\]
for every $(t,x) \,\in \,[0, T] \times \left[-\beta_1,
\beta_1\right]$. Thus, by \eqref{lunga}, one has
 \begin{equation}\label{stimacar20}
\begin{aligned}
& \int_0^T\int_{-\beta_1}^{\beta_1} \Big(s \Theta \tilde a (Z_x)^2+
s^3 \Theta^3  \frac{(x-x_0)^2}{\tilde a}Z^2\Big)
e^{2s\tilde\varphi}dxdt\\&
     \le k\int_0^T\int_{-\beta_1}^{\beta_1} s \Theta (Z_x)^2e^{2s\Phi} dxdt + k\int_0^T\int_{-\beta_1}^{\beta_1}
s^3\Theta^3Z^2e^{2s\Phi} dxdt \\
&\le kC \int_0^T \int_{\omega}v^2 dxdt.
\end{aligned}
\end{equation}
Hence, by \eqref{stimacar20} and the definition of $W$ and $Z$, we
get
\begin{equation}\label{car10}
\begin{aligned}
&\int_0^T\int_0^{\lambda_1}  \Big( s^3 \Theta^3
\frac{(x-x_0)^2}{a}v^2+s \Theta a (v_x)^2\Big) e^{2s\varphi}dxdt\\
&= \int_0^T\int_0^{\lambda_1}  \Big( s^3 \Theta^3
\frac{(x-x_0)^2}{a}W^2+s \Theta a (W_x)^2\Big) e^{2s\varphi} dxdt \\
&\le \int_0^T\int_{-\lambda_1}^{\lambda_1} \Big( s^3 \Theta^3
\frac{(x-x_0)^2}{\tilde a}W^2+s \Theta \tilde a (W_x)^2\Big) e^{2s\tilde\varphi} dxdt \\
&=\int_0^T\int_{-\lambda_1}^{\lambda_1} \Big( s^3
\Theta^3
\frac{(x-x_0)^2}{\tilde a}Z^2+s \Theta \tilde a (Z_x)^2\Big) e^{2s\tilde\varphi}dxdt \\
& \le \int_0^T\int_{-\beta_1}^{\beta_1} \Big( s^3
\Theta^3 \frac{(x-x_0)^2}{\tilde a}Z^2+s \Theta \tilde a
(Z_x)^2\Big)
e^{2s\tilde\varphi}dxdt \\
&\le C \int_0^T \int_{\omega} v^2dxdt,
\end{aligned}
\end{equation}
for a positive constant $C$.

Therefore, by \eqref{carin0} and \eqref{car10}, Lemma \ref{lemma3}
follows.
\end{proof}

\bigskip

\begin{proof}[Proof of Lemma $\ref{lemma3'}$]
The idea is quite similar to that of the proof of Lemma
\ref{lemma3}, so we will be faster in the calculations. Suppose that
$x_0 <\alpha$ (the proof is analogous if we assume that $\beta <
x_0$ with obvious adaptations); moreover, set $\lambda:=
\frac{2\alpha +\beta}{3}$ and $\gamma:= \frac{\alpha +2\beta}{3}$,
so that $\alpha<\lambda<\gamma<\beta$. Now,  fix $\tilde \alpha \in (\alpha, \lambda)$, $\tilde \beta \in (\gamma, \beta)$ and  consider a smooth
function $\xi:[0,1]\to[0,1]$ such that
\[
\xi(x)=\begin{cases} 0&x\in [0,\tilde\alpha],\\
1 & x\in[\lambda,\gamma],\\
0&x\in [\tilde\beta,1].
\end{cases}
\]
Then, define $w:= \xi v$, where $v$ is any fixed solution of
\eqref{h=01}, so that $w$ satisfies \eqref{eq-w*} with
\begin{equation}\label{f}
f^2= (( a \xi_x v )_x + a\xi_x  v_x)^2 \le C(
v^2+ (v_x)^2)\chi_{\hat \omega},
\end{equation}
where $\hat\omega=(\tilde \alpha, \lambda) \cup( \gamma,\tilde \beta)$, by Hypothesis \ref{Ass03}.

Applying Theorem \ref{Cor1} and using the fact that $w\equiv0$ in a
neighborhood of $x=0$  and $x=1$, we have
\begin{equation}\label{car9'}
\begin{aligned}
\int_0^T \int_0^1\Big( s \Theta a  (w_x)^2 + s^3 \Theta^3
\frac{(x-x_0)^2}{a}\ w^2 \Big) e^{2s \varphi} \, dx dt \le C
\int_0^T \int_0^1e^{2s \varphi} f^2  dxdt,
\end{aligned}
\end{equation}
for all $s \ge s_0$. Hence, we find
\begin{equation}\label{stimacar'}
\begin{aligned}
&\int_0^T\int_{\lambda}^{\gamma}\left( s \Theta a(v_x)^{2} + s^3
\Theta ^3 \frac{(x-x_0)^2}{a} v^{2}\right)
e^{{2s\varphi}} dxdt\\
&=\int_0^T\int_{\lambda}^{\gamma}\Big( s \Theta  a
(w_x)^2 + s^3 \Theta^3 \frac{(x-x_0)^2}{a} w^2 \Big)
e^{2s \varphi} \, dx dt\\
&\le \int_0^T \int_0^1\Big( s \Theta  a(w_x)^2 + s^3
\Theta^3 \frac{(x-x_0)^2}{a}w^2 \Big)
e^{2s \varphi} \, dx dt\\
& \mbox{ (by \eqref{car9} and \eqref{f})}\\
&\le C  \int_0^T \int_{\hat \omega}e^{2s \varphi}( v^2+
(v_x)^2)dxdt\\
& \mbox{ (by Proposition \ref{caccio} with $\vp=\Theta\psi$, since
$\hat \omega\subset \subset \omega$, and using the fact that}\\
& \mbox{ $e^{2s\varphi}$ is bounded)}\\
&\le C \int_0^T \int_{\omega} v^2 dxdt.
\end{aligned}
\end{equation}

Analogously, we define a smooth function $\tau: [0,1]
\to [0,1]$ such that
\[
\tau(x)=\begin{cases} 0& x\in [0,\lambda],\\
1& x\in [\gamma, 1].
\end{cases}
\]
Defining
$z:= \tau v$, then $z$ satisfies
\begin{equation}\label{eq-z*'}
\begin{cases}
z_t + (a  z_{x})_x= h,  &(t,x) \in(0,T)\times (\alpha,1)\\
z(t,\alpha)= z(t,1)=0, & t \in (0,T),
\end{cases}
\end{equation}
with $h:=( a \tau_x v )_x + a\tau_x  v_x$, which
is supported in $\tilde \omega =(\lambda, \gamma)$.

Observe that \eqref{eq-z*'} is a nondegenerate problem, hence, thanks to Remark \ref{RemCarleman}, we
can apply the classical
Carleman estimate \eqref{3} with $A=\alpha$ and $B=1$, obtaining
\[
\begin{aligned}
&\int_0^T\int_{\alpha}^1 s \Theta (z_x)^2e^{2s\Phi}
dxdt+ \int_0^T\int_{\alpha}^1 s^3\Theta^3 z^2e^{2s\Phi} dxdt  \\
& \le c\int_0^T\int_{\alpha}^1e^{2s\Phi} h^2 dxdt,
\end{aligned}
\]
where $r>0$, $s\geq s_0$ and $c>0$. Let us note that the boundary
term which appears in the original estimate is nonpositive and thus
is neglected.

Now, we use Proposition \ref{caccio}, getting as above
\begin{equation}\label{standa'}
\begin{aligned}
&\int_0^T\int_{\alpha}^1 s \Theta(z_x)^2e^{2s\Phi} dxdt+
\int_0^T\int_{\alpha}^1 s^3\Theta^3z^2e^{2s\Phi} dxdt \\
&\le C \int_0^T\int_{\tilde \omega}e^{2s\Phi}( v^2+ (v_x)^2)dxdt\le
C \int_0^T\int_{\tilde \omega} v^2dxdt+
C\int_0^T\int_{\tilde \omega}e^{2s\Phi}(v_x)^2dxdt \\
& \le C \int_0^T\int_{\omega} v^2dxdt.
\end{aligned}
\end{equation}

Now, choose the constant $c_1$ in \eqref{c_1} so
that
\[
c_1 \ge \begin{cases} \frac{\displaystyle  r\left[\int_ {\alpha}^1
\frac{1}{\sqrt{a(t)}} \int_t^1 \fg(s) dsdt + \int_{\alpha}^1
\frac{\fh_0}{\sqrt{a(t)}}dt\right]+ \mathfrak{c}}{\displaystyle
c_2-\frac{(1-x_0)^2}{a(1)(2-K)}} &\mbox{ in the (WD)
case},\\
\frac{\displaystyle \mathfrak{c}-1}{\displaystyle
c_2-\frac{(1-x_0)^2}{a(1)(2-K)}} &\mbox{ in the (SD) case}.
\end{cases}
\]
Then, by definition of $\varphi$ and the choice of $c_1$, one can
prove that there exists a positive constant $k$, for example
\[
k = \max \left\{\max_{\left[\alpha,
1\right]}a(x),\frac{(1-x_0)^2}{a(1)}\right\},
\]
such that
\[
a(x) e^{2s\varphi(t,x)} \le k e^{2s\Phi(t,x)}
\]
and
\[
\frac{(x-x_0)^2}{a(x)}e^{2s\varphi(t,x)} \le k e^{2s \Phi(t,x)}
\]
$\forall \,(t,x) \in [0, T] \times \left[\alpha, 1\right]$. Thus, by
\eqref{prima'} and \eqref{seconda'}, via \eqref{standa'}, we find
\[
\begin{aligned}
&\int_0^T\int_{\alpha}^1 \Big(s \Theta a (z_x)^2
 + s^3 \Theta^3
     \frac{(x-x_0)^2}{a}z^2\Big) e^{2s\varphi}dxdt \\
& \le k \int_0^T\int_{\alpha}^1 s \Theta (z_x)^2e^{2s\Phi} dxdt +k\int_0^T\int_{\alpha}^1 s^3\Theta^3 z^2e^{2s\Phi} dxdt\\
& \le C \int_0^T \int_{\omega} v^2dxdt,
\end{aligned}
\]
for a positive constant $C$ and $s$ large enough. Hence, by
definition of $z$ and by the inequality above, we get
\begin{equation}\label{stimacar2'}
\begin{aligned}
&\int_0^T\int_{\gamma}^1 \Big(s \Theta a (v_x)^2 + s^3 \Theta^3
\frac{(x-x_0)^2}{a}v^2\Big) e^{2s\varphi} dxdt\\
&=\int_0^T\int_{\gamma}^1 \Big(s \Theta a (z_x)^2 +
s^3 \Theta^3
\frac{(x-x_0)^2}{a}z^2\Big) e^{2s\varphi}dxdt\\
& \le \int_0^T\int_{\alpha}^1 \Big(s \Theta a
(z_x)^2 + s^3 \Theta^3 \frac{(x-x_0)^2}{a}z^2\Big) e^{2s\varphi}dxdt
\\
&\le C \int_0^T \int_{\omega} v^2dxdt,
\end{aligned}
\end{equation}
for a positive constant $C$ and for $s$ large enough.

Thus \eqref{stimacar'} and \eqref{stimacar2'} imply
\begin{equation}\label{carin0'}
\begin{aligned}
\int_0^T \int _{\lambda}^{1}  \Big( s \Theta  a (v_x)^2 + s^3
\Theta^3 \frac{(x-x_0)^2}{a}  v^2 \Big) e^{2s \varphi } \, dx dt \le
C \int_{0}^T \int _{\omega}   v^2  dxdt,
\end{aligned}
\end{equation}
for some positive constant $C$ and $s\geq s_0$.

To complete the proof it is sufficient to prove a similar inequality
for $x\in[0,\lambda]$. To this aim, we follow the reflection
procedure introduced in the proof of Lemma~\ref{lemma3}: consider
the functions $W$ and $\tilde a$ introduced in \eqref{W} and
\eqref{tildea}, so that $W$ satisfies \eqref{dispari}.

Now, consider a smooth function $\rho: [-1,1] \to [0,1]$ such that
\[
\rho(x)=
\begin{cases} 0 & x\in[-1,-\gamma],\\
1& x\in [-\lambda, \lambda],\\
0&x\in [\gamma,1],
\end{cases}
\]
and define $Z:= \rho W$; thus $Z$
satisfies
\begin{equation}\label{eq-Z*'}
\begin{cases}
Z_t + (\tilde aZ_x)_x=\tilde h,  &(t,x) \in (0,T)\times (-\beta,\beta),\\
Z(t,-\beta)= Z(t,\beta)=0, & t \in (0,T),
\end{cases}
\end{equation}
with $\tilde h=( \tilde a \rho_x W)_x + \tilde a\rho_x
W_x$, which is supported in
$\left[-\gamma,-\lambda\right] \cup\left[\lambda, \gamma\right]$.

Now, define $\tilde \varphi(t,x) := \Theta(t) \tilde \psi (x)$,
where $ \tilde \psi(x)$ is defined as in \eqref{tildepsi}. Using the
analogue of Theorem \ref{Cor1} on $(- \beta, \beta)$ in place of
$(0,1)$ and with $\varphi$ replaced by $\tilde \varphi$, the
equalities $Z_x(t, -\beta)=Z_x(t, \beta)=0$, and the definition of
$W$, we get
\[
\begin{aligned}
& \int_0^T\int_{-\beta}^{\beta}  \left(s\Theta \tilde a (Z_x)^2 +
s^3 \Theta^3
\frac{(x-x_0)^2}{\tilde a } Z^2\right)e^{2s\tilde\varphi}dxdt\\
& \le c
\int_0^T\int_{-\beta}^{\beta}e^{2s\tilde\varphi} \tilde h^{2}dxdt\\
&\le C \int_0^T \int_{-\gamma}^{-\lambda}(
W^2+ (W_x)^2)e^{2s \tilde \varphi}dxdt + C\int_0^T \int_{\lambda}^{\gamma}(W^2+ (W_x)^2)e^{2s\varphi}dxdt\\
&\mbox{(since $\tilde \psi (x)= \psi (-x)$ for $x <0$)}\\
& = 2C\int_0^T \int_{\lambda}^{\gamma}(W^2+
(W_x)^2)e^{2s\varphi}dxdt = 2C\int_0^T \int_{\lambda}^{\gamma}(v^2+
(v_x)^2)e^{2s\varphi}dxdt\\
& \mbox{ (by Propositions \ref{caccio}) }\\
& \le C \int_0^T \int_{\omega} v^2dxdt,
\end{aligned}
\]
for some positive constants $c$ and $C$ and $s$ large enough.

 Hence, by the definitions of $Z$, $W$ and $\rho$, and using the previous inequality one has
\begin{equation}\label{car10'}
\begin{aligned}
&\int_0^T\int_{0}^{\lambda}  \left(s\Theta a(v_x)^2 + s^3 \Theta^3
\frac{(x-x_0)^2}{a} v^2\right)e^{2s\varphi}dxdt\\
&= \int_0^T\int_{0}^{\lambda}  \left(s\Theta a (W_x)^2 + s^3
\Theta^3
\frac{(x-x_0)^2}{a} W^2\right)e^{2s\varphi}dxdt\\
&= \int_0^T\int_{0}^{\lambda}  \left(s\Theta
a(Z_x)^2 + s^3 \Theta^3
\frac{(x-x_0)^2}{a} Z^2\right)e^{2s\varphi}dxdt\\
&\le \int_0^T\int_{0}^{\beta}  \left(s\Theta
a(Z_x)^2 + s^3 \Theta^3
\frac{(x-x_0)^2}{a}Z^2\right)e^{2s\varphi}dxdt\\
&\le \int_0^T\int_{-\beta}^{\beta} \left(s\Theta
\tilde a (Z_x)^2 + s^3 \Theta^3
\frac{(x-x_0)^2}{\tilde a} Z^2\right)e^{2s\varphi}dxdt\\
&\le C \int_0^T \int_{\omega} v^2dxdt,
\end{aligned}
\end{equation}
for a positive constant $C$ and $s$ large enough. Therefore, by
\eqref{carin0'} and \eqref{car10'}, the conclusion follows.
\end{proof}

We shall also use the following
\begin{Lemma}\label{obser.regular}
Assume Hypotheses $\ref{ipotesiomega}$ and $\ref{Ass03}$. Then there
exists a positive constant $C_T$ such that every solution $v \in
\cal W_1$ of \eqref{h=01} satisfies
\[
\int_0^1v^2(0,x) dx \le C_T\int_0^T \int_{\omega}v^2(t,x)dxdt.
\]
\end{Lemma}
\begin{proof}
Multiplying the equation of \eqref{h=01} by $v_t$ and integrating by
parts over $(0,1)$, one has
\[
\begin{aligned}
&0 = \int_0^1(v_t+ (av_x)_x)v_t dx= \int_0^1 (v_t^2+ (av_x)_xv_t )dx
= \int_0^1v_t^2dx + \left[av_xv_t \right]_{x=0}^{x=1} \\&-
\int_0^1av_xv_{tx} dx= \int_0^1v_t^2dx -
\frac{1}{2}\frac{d}{dt}\int_0^1a(v_x)^2
 \ge - \frac{1}{2}
\frac{d}{dt}\int_0^1 a(v_x)^2dx.
\end{aligned}
\]
Thus, the function $t \mapsto \int_0^1 a(v_x)^2 dx$ is increasing
for all $t \in [0,T]$. In particular,
$$
\int_0^1 av_x(0,x)^2dx \le \int_0^1av_x(t,x)^2dx \mbox{ for every
}t\in[0,T].
$$
Integrating the last inequality over $\displaystyle\left[\frac{T}{4},
\frac{3T}{4} \right]$, $\Theta$ being bounded therein, we find
\[
\begin{aligned}
\int_0^1a(v_x)^2(0,x) dx &\le
\frac{2}{T}\int_{\frac{T}{4}}^{\frac{3T}{4}}\int_0^1a(v_x)^2(t,x)dxdt\\&\le
C_T \int_{\frac{T}{4}}^{\frac{3T}{4}}\int_0^1s\Theta
a(v_x)^2(t,x)e^{2s\varphi}dxdt.
\end{aligned}
\]
Hence, by Lemma \ref{lemma3} or by Lemma \ref{lemma3'} and the
previous inequality, there exists a positive constant $C$ such that
\begin{equation}\label{stum}
\int_0^1a (v_x)^2(0,x) dx \le C \int_0^T \int_{\omega}v^2dxdt.
\end{equation}

Proceeding again as in the proof of Lemma \ref{lemma2} and applying
the Hardy- Poincar\'{e} inequality, by \eqref{stum}, one has
\[
\begin{aligned}
\int_0^1 \left(\frac{a}{(x-x_0)^2}\right)^{1/3}v^2(0,x)dx &\leq
\int_0^1 \frac{p}{(x-x_0)^2} v^2(0,x)dx  \\
&\le C_{HP} \int_0^1 p(v_x)^2(0,x) dx \\&\le \max\{C_1, C_2\}C_{HP}
\int_0^1a(v_x)^2(0,x) dx\\& \le C \int_0^T\int_{\omega}v^2dxdt,
\end{aligned}
\]
for a positive constant $C$. Here $p(x) = (a(x)|x-x_0|^4)^{1/3}$ if
$K > \displaystyle\frac{4}{3}$ or $\displaystyle p(x) =\max_{[0,1]}
a|x-x_0|^{4/3}$ otherwise,
$$
C_1:=
\max\left\{\displaystyle\left(\frac{x_0^2}{a(0)}\right)^{2/3},\displaystyle\left(\frac{(1-x_0)^2}{a(1)}\right)^{2/3}\right\},
$$
$C_2:=
\max\left\{\displaystyle\frac{x_0^{4/3}}{a(0)},\displaystyle\frac{(1-x_0)^{4/3}}{a(1)}\right\}$
and $C_{HP}$ is the Hardy-Poincar\'{e} constant, as before.

By Lemma \ref{rem}, the function $\displaystyle x\mapsto
\frac{a(x)}{(x-x_0)^2}$ is nondecreasing on $[0, x_0)$ and
nonincreasing on $(x_0,1]$; then
\[
\left(\frac{a(x)}{(x-x_0)^2}\right)^{1/3}\ge
C_3:=\min\left\{\left(\frac{a(1)}{(1-x_0)^2}\right)^{1/3},
\left(\frac{a(0)}{x_0^2}\right)^{1/3}\right\} >0.
\]
Hence
\[
C_3\int_0^1v(0,x)^2dx \le C \int_0^T\int_{\omega}v^2dxdt\] and the
thesis follows.
\end{proof}

\begin{proof}[Proof of Proposition $\ref{obser.}$] The proof is now standard, but we give it with some precise references: let $v_T \in L^2(0,1)$
and let $v$ be the solution of \eqref{h=0} associated to $v_T$.
Since $D({\cal A}_1^2)$ is densely defined in $L^2(0,1)$, there exists
a sequence $(v_T^n)_{n}\subset D({\cal A}_1^2)$ which converges to
$v_T$ in $L^2(0,1)$. Now, consider the solution $v_n$ associated to
$v_T^n$.

As shown in Theorem \ref{th-parabolic}, the semigroup generated by
$\mathcal A_1$ is analytic, hence $\mathcal A_1$ is closed (for example,
see \cite[Theorem I.1.4]{en}; thus, by \cite[Theorem II.6.7]{en},
we get that $(v_n)_n$ converges to a certain $v$ in $C([0,T];
L^2(0,1))$, so that
\[
\lim_{n \rightarrow + \infty} \int_0^1v_n^2(0,x) dx =
\int_0^1v^2(0,x)dx,
\]
and also
\[
\lim_{n \rightarrow + \infty}\int_0^T \int_{\omega}v_n^2dxdt =
\int_0^T \int_{\omega}v^2dxdt.
\]
But, by Lemma \ref{obser.regular} we know that
\[
\int_0^1v_n^2(0,x) dx \le C_T\int_0^T \int_{\omega}v_n^2dxdt.
\]
Thus Proposition \ref{obser.} is now proved.
\end{proof}

\section{The non divergence case}
 In this section we make the following assumptions on the degenerate
function $a$ (see also Remark \ref{quellali} concerning the divergence case):
\begin{Assumptions}\label{Ass031} Hypothesis $\ref{Ass021}$ is satisfied. Moreover,
if Hypothesis $\ref{Ass0}$ holds, then there exist
two functions $\fg \in L^\infty_{\rm loc}([0,1]\setminus \{x_0\})$, $\fh \in W^{1,\infty}_{\rm loc}([0,1]\setminus \{x_0\},L^\infty(0,1))$ and
two strictly positive constants $\fg_0$, $\fh_0$ such that $\fg(x) \ge \fg_0$ for a.e. $x$ in $[0,1]$ and
\begin{equation}\label{5.3'}
\frac{a'(x)}{2\sqrt{a(x)}}\left(\int_x^B\fg(t) dt + \fh_0 \right)+ \sqrt{a(x)}\fg(x) =\fh(x,B)\quad \text{for a.e.} \; x,B \in [0,1]
\end{equation}
with $x<B<x_0$ or $x_0<x<B$.
\end{Assumptions}

As for the case in divergence form, by the Carleman estimates given in Theorems \ref{mono1} and
\ref{Cor11}, we will deduce a fundamental observability inequality
for the homogeneous adjoint problem to \eqref{linear}, i.e.
\begin{equation}\label{h=00}
\begin{cases}
v_t +av_{xx}= 0, &(t,x) \in  Q_T,
\\[5pt]
v(t,0)=v(t,1) =0, & t \in (0,T),\\
v(T,x)= v_T(x)\in  L^2_{\frac{1}{a}}(0,1),
\end{cases}
\end{equation}
where $T>0$ is given. Such an observability inequality will hold
true both in the weakly and in the strongly degenerate cases, as the
next proposition shows.
\begin{Proposition}\label{obser.1}
Assume \eqref{omega} and Hypothesis $\ref{Ass031}$. Then there
exists a positive constant $C_T$ such that the solution $v \in
 C([0, T]; L^2_{\frac{1}{a}}(0,1)) \cap L^2 (0,T;
\cH^1_{\frac{1}{a}}(0,1))$ of \eqref{h=00} satisfies
\begin{equation}\label{obser1.1}
\int_0^1v^2(0,x)\frac{1}{a} dx \le C_T\int_0^T
\int_{\omega}v^2\frac{1}{a}dxdt.
\end{equation}
\end{Proposition}

Using inequality \eqref{obser1.1} the following null controllability result for \eqref{linear} in non divergence form holds:
\begin{Theorem}\label{th31}
Assume  \eqref{omega} and Hypothesis $\ref{Ass031}$. Then, given $u_0
\in L^2_{\frac{1}{a}}(0,1)$, there exists $h \in L^2_{\frac{1}{a}}
(Q_T)$ such that the solution $u$ of \eqref{linear} satisfies
\begin{equation*}
u(T,x)= 0 \ \text{ for every  } \  x \in [0, 1].
\end{equation*}
Moreover
\[
\int_0^T \int_0^1 h^2 \frac{1}{a}dxdt \le C \int_0^1 u_0^2
\frac{1}{a}dx,
\]
for some positive constant $C$ independent of $u_0$.
\end{Theorem}

We refer to Comment $2$ in Chapter \ref{sec7} to explain why in
Proposition \ref{obser.1} and Theorem \ref{th3} we consider only the
case in which the degeneracy point is {\sl outside} the control
region.
\medskip

As for the divergence case, a straightforward consequence of Proposition \ref{obser.1} and of
Theorem \ref{th31} are the following results, which are of interest
when the control region lies on both the two sides of the degeneracy
point.
\begin{Corollary}\label{dueparti}
Assume Hypothesis $\ref{Ass031}$ and \eqref{omega2}. Then there
exists a positive constant $C_T$ such that every solution $v \in
\cal W_2$ of \eqref{h=011} satisfies
\[
\int_0^1v^2(0,x) \frac{1}{a} dx \le C_T\int_0^T \int_{\omega}v^2(t,x)\frac{1}{a} dxdt.
\]
\end{Corollary}

As a standard consequence one has the next null controllability
result:
\begin{Theorem}\label{th2pezzi}
Assume Hypothesis $\ref{Ass031}$ and $\eqref{omega2}$. Then, given
$u_0 \in L^2_{\frac{1}{a} } (0,1)$, there exists $h \in L^2_{\frac{1}{a} }(Q_T)$ such that the
solution $u$ of \eqref{linear} satisfies
\begin{equation*}
u(T,x)= 0 \ \text{ for every  } \  x \in [0, 1].
\end{equation*}
Moreover
\[
\int_0^T \int_0^1 h^2\frac{1}{a}  dxdt \le C \int_0^1
u_0^2(x)\frac{1}{a}  dx
\]
for some positive constant $C$.
\end{Theorem}

\subsection{Proof of Proposition $\ref{obser.1}$}
As for the problem in divergence form, we will start giving some preliminary results for the following
homogeneous adjoint final--time problems having {\em more regular
final--time datum}:
\begin{equation}\label{h=011}
\begin{cases}
v_t +av_{xx}= 0, &(t,x) \in  Q_T,
\\
v(t,0)=v(t,1) =0, & t \in (0,T),
\\
v(T,x)= v_T(x) \,\in D(\mathcal A_2^2),
\end{cases}
\end{equation}
where, we recall, $\mathcal A_2u:=au_{xx}$ with
$D(\mathcal A_2)=\cH^2_{\frac{1}{a}}(0,1)$, and
\[
D(\mathcal A_2^2) = \Big\{u \,\in \,D(\mathcal A_2)\;\big|\; \mathcal A_2u \,\in \,D(\mathcal A_2) \;\Big\}.
\]
Observe that $D(\mathcal A_2^2)$ is densely defined in $D(\mathcal A_2)$ (see, for
example, \cite[Lemma VII.2]{b}) and hence in
$L^2_{\frac{1}{a}}(0,1)$. As in \cite{cfr}, \cite{cfr1} or \cite{f},
letting $v_T$ vary in $D(\mathcal A_2^2)$, we introduce the class of solutions
to \eqref{h=011}, i.e.
\[
\cal{W}_2:=\Big\{ v \in
C^1\big([0,T]; L^2_{\frac{1}{a}}(0,1)\big)\cap C\big([0,T]; D(\mathcal
A_2)\big)\,|\,v \text{ solves \eqref{h=011}}\Big\},
\]
with the obvious meaning that it is a class and not a set of one
function, since $v_T$ vary.

Obviously (see, for example, \cite[Theorem VII.5]{b})
\[
\cal{W}_2\subset C^1\big([0,T];\:\cH^2_{\frac{1}{a}}(0,1)\big)
\subset \mathcal{S}_2 \subset \cal{U}_2,
\]
where $\mathcal{S}_2$ is defined in \eqref{v1} and
\[
\cal{U}_2:= C([0,T]; L^2_{\frac{1}{a}}(0,1)) \cap L^2(0, T;
\cH^1_{\frac{1}{a}}(0,1)).
\]

Also in this case the Caccioppoli's inequality is crucial. In the non divergence case it reads as follows:
\begin{Proposition}[Caccioppoli's inequality]\label{caccio1}
Assume that either Hypothesis $\ref{Ass0}$ and \eqref{5.3'} or Hypothesis $\ref{Ass01}$ hold.
Let $\omega'$ and $\omega$ two open subintervals of $(0,1)$ such
that $\omega'\subset \subset \omega \subset  (0,1)$ and $x_0 \not
\in \overline{\omega}$. Let $\varphi(t,x)=\Theta(t)\Upsilon(x)$, where
$\Theta$ is defined in \eqref{c_1} and
\[
\Upsilon \in C([0,1],(-\infty,0))\cap
C^1([0,1]\setminus\{x_0\},(-\infty,0))
\]
satisfies \eqref{stimayx}. Then, there exist two positive constants
$C$ and $s_0$ such that every solution $v \in \cal W_2$ of the
adjoint problem \eqref{h=011} satisfies
\[
   \int_{0}^T \int _{\omega'}   (v_x)^2e^{2s\varphi } dxdt
    \ \leq \ C \int_{0}^T \int _{\omega}   v^2  dxdt  \ \leq \ C \int_{0}^T \int _{\omega}   v^2  \frac{1}{a}dxdt ,
\]
for all $s\geq s_0$.
\end{Proposition}
Observe that we require $x_0 \not \in \bar
\omega$, since in the applications below the control region $\omega$
is assumed to satisfy \eqref{omega}. Moreover, as in Remark
\ref{prototipo}, one can prove that our prototype for $\Upsilon$ is
the function $\psi$ defined in \eqref{c_11}.

\begin{proof}[Proof of Proposition \ref{caccio1}]
The proof is an adaptation of the one of \cite[Proposition 4.2]{fm},
so we will skip some details. Let us consider a smooth function
$\xi: [0,1] \to \Bbb R$ such that
\[
\begin{cases}
0 \leq \xi (x)  \leq 1 &  \text{for all } x \in [0,1], \\
\xi (x) = 1   &   x \in \omega', \\
\xi (x)=0 &     x \in [\, 0, 1 ]\setminus \omega.
\end{cases}
\]
Hence, by the very definition of $\vp$, we have
\[
    \begin{aligned}
    0 &= \int _0 ^T \frac{d}{dt} \left(\int _0 ^1 \xi ^2 e^{2s\varphi}
    v^2dx\right)dt
    = \int_{Q_T}\big(2s \xi ^2  \varphi _t e^{2s\varphi} v^2 + 2 \xi ^2
    e^{2s\varphi} vv_t \big)dxdt\\
    & \mbox{(since $v$ solves \eqref{h=011} and has homogeneous boundary conditions)}
   \\
    &= 2\int_{Q_T} s\xi^2 \varphi _t e^{2s\varphi} v^2 dxdt+  2\int_{Q_T}( \xi
    ^2e^{2s\varphi}a)_xvv_xdxdt\\
& +  2\int_{Q_T}  \xi
    ^2e^{2s\varphi}a(v_x)^{2}dxdt.
\end{aligned}
    \]
Therefore, by definition of $\xi$, the previous identity gives
\[
    \begin{aligned}
    &2\int_0^T \int_{\omega}\xi^2 e^{2s\varphi} a (v_x)^2dxdt
=- 2 \int_0^T \int_{\omega}s \xi ^2 \varphi _t e^{2s\varphi} v^2dxdt
\\
&  -2\int_0^T \int_{\omega} \left(\xi^2e^{2s\varphi}a\right)_x\frac{\xi
e^{s\varphi}\sqrt{a}}{\xi e^{s\varphi}\sqrt{a}}\:vv_x\:dxdt
   \\
&\mbox{(by the Cauchy--Schwarz inequality)}\\
 &
 \le - 2 \int_0^T \int_{\omega}s \xi^2 \varphi _t e^{2s\varphi} v^2dxdt
    +\int_0^T \int_{\omega}\left( \xi e^{s\varphi}\sqrt{a}  v_x \right) ^2dxdt
  \\&  + \int_0^T \int_{\omega}\left(\frac{( \xi ^2e^{2s\varphi} a)_x}{\xi
    e^{s\varphi}  \sqrt{a} }v \right)^2dxdt
    \\
   &= - 2 \int_0^T \int_{\omega} s \xi^2\varphi _t e^{2s\varphi} v^2dxdt
 + \int_0^T \int_{\omega} \xi ^2 e^{2s\varphi} a (v_x)^2 dxdt
  \\&+ 4\int_0^T \int_{\omega}[\left(\xi
e^{s\vp}\sqrt{a}\right)_x]^2v^2dxdt.
    \end{aligned}
\]
Thus
\[
\begin{aligned}
\int_0^T \int_{\omega}\xi ^2 e^{2s\varphi} a (v_x)^2dxdt &\le - 2
\int_0^T \int_{\omega}s \xi^2 \varphi _t e^{2s\varphi} v^2dxdt
    \\&+ 4\int_0^T \int_{\omega}[\left(\xi
e^{s\psi}\sqrt{a}\right)_x]^2v^2dxdt.
\end{aligned}
\]
Since $x_0 \not \in \overline{\omega'}$, then
\[
\begin{aligned}
&\inf_{\omega '}a(x)\int_0^{T}\int _{\omega '} e^{2s\varphi}
(v_x)^2dxdt \le \int_0^T \int_{\overline{\omega'}} \xi^2
e^{2s\varphi} a (v_x)^2dxdt\\
&\le \int_0^T \int_{\omega}
\xi ^2 e^{2s\varphi} a (v_x)^2dxdt\\
&\le - 2 \int_0^T \int_{\omega}s \xi^2 \varphi_t e^{2s\varphi}
v^2dxdt
 + 4\int_0^T \int_{\omega}[\left(\xi
e^{s\vp}\sqrt{a}\right)_x]^2v^2dxdt.
  \end{aligned}
\]

As in \cite[Proposition 4.2]{fm}, one can show that $s\varphi_te^{2s\varphi}$ is uniformly
bounded if $s\geq s_0>0$, since $\Upsilon$ is strictly negative, a
rough estimate being
\[
|s\vp_te^{2s\varphi}|\leq cs|\Upsilon|\Theta^{5/4}e^{2s\Upsilon \Theta}\leq c\frac{1}{s_0^{5/4}(-\max
\Upsilon)^{5/4}}.
\]
On the other hand, $\left(\xi
e^{s\vp}\sqrt{a}\right)_x$ can be estimated by
\[
C\left(e^{2s\vp}+s^2(\vp_x)^2e^{2s\vp}+ e^{2s\vp}\frac{(a')^2}{a}\right).
\]

Of course, $e^{2s\vp}<1$, $\displaystyle\frac{(a')^2}{a}$ exists and is bounded in $\omega$ since $x_0\not \in \bar \omega$ and \eqref{5.3'} holds with Hypothesis \ref{Ass0} or Hypothesis \ref{Ass01} is in force, while $s^2(\vp_x)^2e^{2s\vp}$ can be
estimated with
\[
\frac{c}{(-\max \Upsilon)^2}(\Upsilon_x)^2 \leq \frac{c}{a}
\]
by \eqref{stimayx}, for some constants $c>0$ (see \cite[Proposition 4.2]{fm}).
Hence, there exists a
positive constant $C$ such that
\[
\begin{aligned}
- 2 &\int_0^T \int_{\omega}s \xi^2 \varphi_t e^{2s\varphi}
v^2dxdt
 + 4\int_0^T \int_{\omega}[\left(\xi
e^{s\vp}\sqrt{a}\right)_x]^2v^2dxdt\\
 &  \le C\int_0 ^T \int _{\omega} v^2 dxdt \le C \int_0 ^T \int _{\omega} v^2 \frac{1}{a}dxdt,
   \end{aligned}
\]
and the claim follows.
\end{proof}

\begin{Lemma}\label{lemma31}
Assume \eqref{omega} and Hypothesis $\ref{Ass031}$. Then there exist
two positive constants $C$ and $s_0$ such that every solution $v \in
\cal W_2$ of \eqref{h=011} satisfies
\[
\int_{Q_T}\left( s \Theta  (v_{x})^{2} + s^3 \Theta ^3
\left(\frac{x-x_0}{a}\right)^2 v^{2}\right) e^{{2s\varphi}}  dxdt\le
C \int_0^T\int_{\omega}v^{2}\frac{1}{a} dxdt
\]
for all $s \ge s_0$ and $d_1$ sufficiently large.
\end{Lemma}
\begin{proof}
Suppose that $x_0 <\alpha$ (the proof is similar if we assume that
$\beta < x_0$ with simple adaptations); moreover, set $\lambda:=
\frac{2\alpha +\beta}{3}$ and $\gamma:= \frac{\alpha +2\beta}{3}$,
so that $\alpha<\lambda<\gamma<\beta$. Now, fix $\tilde \alpha \in (\alpha, \lambda)$ and $\tilde \beta \in (\gamma, \beta)$ and consider a smooth
function $\xi: [0,1] \to \Bbb R$ such that
\[\begin{cases}
    0 \leq \xi (x)  \leq 1, &  \text{ for all } x \in [0,1], \\
    \xi (x) = 1 ,  &   x \in [\lambda, \gamma],\\
    \xi (x)=0, &     x \in [0,1]\setminus (\tilde \alpha, \tilde \beta).
    \end{cases}
\]

Define $w:= \xi v$, where $v$ is any fixed solution of
\eqref{h=011}. Hence, neglecting the final--time datum (of no
interest in this context), $w$ satisfies
\[
\begin{cases}
w_t + a  w_{xx} = a (\xi _{xx} v + 2\xi _x  v_x )=:f,&
(t,x) \in(0, T)\times (0,1), \\
w(t,0)= w(t,1)=0, & t \in (0,T).
\end{cases}
\]

Applying Theorem \ref{Cor11} and using the fact that $w\equiv0$ in a
neighborhood of $x=0$  and $x=1$, we have
\begin{equation}\label{car91}
\begin{aligned}
\int_0^T \int_0^1\Big( s \Theta   (w_x)^2 + s^3 \Theta^3
   \left(\frac{x-x_0}{a}\right)^2\ w^2 \Big)
    e^{2s \varphi} \, dx dt
       \le C\int_{Q_T}\frac{e^{2s \varphi}}{a} f^2  dxdt,
\end{aligned}
\end{equation}
for all $s \ge s_0$. Then, using the definition of $\xi$  and in
particular the fact that  $\xi_x$ and  $\xi_{xx}$ are supported in
$\hat \omega$, where  $\hat \omega:= (\tilde \alpha, \lambda) \cup( \gamma,
\tilde\beta) $, we can write
\begin{equation}\label{fa}
\frac{f^2}{a}= a(\xi _{xx} v +2 \xi _x v_x)^2 \le C( v^2+
(v_x)^2)\chi_{\hat \omega}.
\end{equation}

Hence, we find
\begin{equation}\label{stimacar1}
\begin{aligned}
&\int_0^T\int_{\lambda}^{\gamma}\left( s \Theta (v_x)^{2} + s^3
\Theta ^3 \left(\frac{x-x_0}{a}\right)^2 v^{2}\right)
e^{{2s\varphi}} dxdt\\
&=\int_0^T\int_{\lambda}^{\gamma}\Big( s \Theta   (w_x)^2 + s^3
\Theta^3 \left(\frac{x-x_0}{a}\right)^2 w^2 \Big)
e^{2s \varphi} \, dx dt\\
&\le\int_{Q_T}\Big( s \Theta  (w_x)^2 + s^3 \Theta^3
\left(\frac{x-x_0}{a} \right)^2w^2 \Big)
e^{2s \varphi} \, dx dt\\
& \mbox{ (by \eqref{car91} and \eqref{fa})}\\
&\le C  \int_0^T \int_{\hat \omega}e^{2s \varphi}( v^2+
(v_x)^2)dxdt\\
& \mbox{ (by Proposition \ref{caccio1} with $\vp=\Theta\psi$, since
$\hat\omega\subset \subset \omega$, and using the fact that}\\
&\mbox{ $ae^{2s\varphi}$ is bounded)}\\
&\le C \int_0^T \int_{\omega} \frac{v^2}{a} xdt,
\end{aligned}
\end{equation}

Now, consider a smooth function $\eta: [0,1] \to \Bbb R$ such that
\[
\begin{cases}
    0 \leq \eta (x)  \leq 1, &  \text{ for all } x \in [0,1], \\
   \eta (x) = 1 ,  &   x \in [\gamma , 1],\\
  \eta (x)=0, &     x \in \left[0,\lambda\right],
\end{cases}
\]
and define $z:= \eta v$. Then $z$ satisfies
\begin{equation}\label{eq-z*}
\begin{cases}
z_t + a  z_{xx}= h,  &(t,x) \in(0,T)\times (\alpha,1)\\
z(t,\alpha)= z(t,1)=0, & t \in (0,T),
\end{cases}
\end{equation}
with $h:=a (\eta _{xx} v + 2\eta _x  v_x)\in L^2\big((0,T)\times
(\alpha, 1)\big)$. Observe that \eqref{eq-z*} is non degenerate,
since $x\in (\alpha,1)$.

Moreover, since the problem is {\em non degenerate},
we can apply Theorem \ref{mono1} with $A= \alpha$, $B=1$ and Remark
\ref{RemCarleman1}, obtaining
\[
\int_0^T\int_\alpha^1 \left(s\Theta (z_x)^2 + s^3 \Theta^3
 z^2\right)e^{2s\Phi}dxdt\le C\int_0^T\int_\alpha^1 h^{2}e^{2s\Phi}dxdt,
\]
for $s\geq s_0$. Let us note that the boundary term which appears in
the original estimate is nonpositive and thus is neglected.

Now, we use the analogue of \eqref{fa} for $h$, Proposition
\ref{caccio1} and, recalling what the support of $\eta$ is, we get
\begin{equation}\label{standa}
\begin{aligned}
&\int_0^T\int_\alpha^1 \left(s\Theta (z_x)^2 + s^3 \Theta^3
 z^2\right)e^{2s\Phi}dxdt \\
&\le C \int_0^T\int_{\tilde \omega}e^{2s\Phi}( v^2+ (v_x)^2)dxdt\le
C \int_0^T\int_{\tilde \omega} v^2dxdt+ C\int_0^T\int_{\tilde
\omega}e^{2s\Phi}(v_x)^2dxdt \\& \le C \int_0^T\int_{\omega}\frac{
v^2}{a}dxdt
\end{aligned}
\end{equation}
where $\tilde \omega =(\lambda, \gamma)$.

Now, choose the constant $d_1$ in \eqref{c_11} so that
\begin{equation}\label{c_1magg}
d_1 \ge \begin{cases} \frac{\displaystyle  r\left[\int_ {\alpha}^1
\frac{1}{\sqrt{a(t)}} \int_t^1 \fg(\tau) d\tau dt + \int_{\alpha}^1
\frac{\fh_0}{\sqrt{a(t)}}dt\right]+ \mathfrak{c}}{\displaystyle
d_2-\frac{(1-x_0)^2e^{R(1-x_0)^2}}{a(1)(2-K)}} & \mbox{
in the (WD) case},\\
\frac{\displaystyle \mathfrak{c}-1}{\displaystyle
d_2-\frac{(1-x_0)^2e^{R(1-x_0)^2}}{a(1)(2-K)}} & \mbox{ in the (SD) case}.
\end{cases}
\end{equation}
Then, by definition of $\varphi$ and the choice of $d_1$, one can
prove that there exists a positive constant $k$, for example
\[
k = \max \left\{1,\left(\frac{1-x_0}{a(1)}\right)^2\right\},
\]
such that
\begin{equation}\label{prima}
e^{2s\varphi(t,x)} \le k e^{2s\Phi(t,x)}
\end{equation}
and
\begin{equation}\label{seconda}
\left(\frac{x-x_0}{a(x)}\right)^2e^{2s\varphi(t,x)} \le k e^{2s
\Phi(t,x)}
\end{equation}
$\forall \,(t,x) \in [0, T] \times \left[\alpha, 1\right]$.

Thus, by \eqref{prima} and \eqref{seconda}, via \eqref{standa}, we
find
\[
\begin{aligned}
&\int_0^T\int_{\alpha}^1 \Big(s \Theta (z_x)^2
 + s^3 \Theta^3
     \left( \frac{x-x_0}{a}\right)^2z^2\Big) e^{2s\varphi}dxdt \\
& \le k \int_0^T\int_{\alpha}^1 s \Theta (z_x)^2e^{2s\Phi} dxdt +k\int_0^T\int_{\alpha}^1 s^3\Theta^3 z^2e^{2s\Phi} dxdt\\
& \le C \int_0^T \int_{\omega} \frac{v^2}{a}dxdt,
\end{aligned}
\]
for a positive constant $C$ and $s$ large enough. Hence, by
definition of $z$ and by the inequality above, we get
\begin{equation}\label{stimacar21}
\begin{aligned}
&\int_0^T\int_{\gamma}^1 \Big(s \Theta  (v_x)^2 + s^3 \Theta^3
\left(\frac{x-x_0}{a}\right)^2v^2\Big) e^{2s\varphi} dxdt\\
&=\int_0^T\int_{\gamma}^1 \Big(s \Theta  (z_x)^2 + s^3 \Theta^3
\left(\frac{x-x_0}{a}\right)^2z^2\Big) e^{2s\varphi}dxdt\\
& \le \int_0^T\int_{\alpha}^1 \Big(s \Theta  (z_x)^2 + s^3 \Theta^3
\left( \frac{x-x_0}{a}\right)^2z^2\Big) e^{2s\varphi}dxdt
\\
&\le C \int_0^T \int_{\omega} \frac{v^2}{a}dxdt,
\end{aligned}
\end{equation}
for a positive constant $C$ and for $s$ large enough.

Thus \eqref{stimacar1} and \eqref{stimacar21} imply
\begin{equation}\label{carin01}
\begin{aligned}
\int_0^T \int _{\lambda}^{1}  \Big( s \Theta  (v_x)^2 + s^3 \Theta^3
\left(\frac{x-x_0}{a}\right)^2  v^2 \Big) e^{2s \varphi } \, dx dt
\le C \int_{0}^T \int _{\omega}   \frac{v^2 }{a} dxdt,
\end{aligned}
\end{equation}
for some positive constant $C$ and $s\geq s_0$. To complete the
proof it is sufficient to prove a similar inequality for
$x\in[0,\lambda]$. To this aim, we follow a reflection procedure
already introduced in \cite{fm}, considering the functions
\[
W(t,x):= \begin{cases} v(t,x), & x \in [0,1],\\
-v(t,-x), & x \in [-1,0]
\end{cases}
\]
and
\[
\tilde a(x):= \begin{cases} a(x), & x \in [0,1],\\
a(-x), & x \in [-1,0],
\end{cases}
\]
so that $W$ satisfies the problem
\[
\begin{cases}
W_t +\tilde a W_{xx}= 0, &(t,x) \in  (0,T)\times (-1,1),\\
W(t,-1)=W(t,1) =0, & t \in (0,T).
\end{cases}
\]

Now, consider a cut off function $\rho: [-1,1] \to \Bbb R$ such that
\[\begin{cases}
0 \leq \rho (x)  \leq 1, &  \text{ for all } x \in [-1,1], \\
\rho (x) = 1 ,  &   x \in [-\lambda, \lambda],\\
\rho (x)=0, &     x \in \left[-1,-\gamma\right ]\cup
\left[\gamma,1\right],
\end{cases}
\]
and define $Z:= \rho W$. Then $Z$ satisfies
\begin{equation}\label{eq-Z*}
\begin{cases}
Z_t + \tilde aZ_{xx}=\tilde h,  &(t,x) \in (0,T)\times (-\beta,\beta),\\
Z(t,-\beta)= Z(t,\beta)=0, & t \in (0,T),
\end{cases}
\end{equation}
where $\tilde h=\tilde a\rho_{xx}W+2\tilde a\rho_xW_x$. Now, defining
$\tilde \varphi(t,x) := \Theta(t) \tilde \psi (x)$, with
\[
\tilde \psi(x) := \begin{cases}
\psi(x), & x \ge 0,\\
\displaystyle \psi(-x)= d_1\left[\int_{-x_0}^x \frac{t+x_0}{\tilde
a(t)}e^{R(t+x_0)^2}dt-d_2\right], & x <0,\end{cases}
\]
we use the analogue of Theorem \ref{Cor11} on $(- \beta,
\beta)$ in place of $(0,1)$ and with $\varphi$ replaced by $\tilde
\varphi$. Moreover, using the fact that $Z_x(t, -\beta)=Z_x(t,
\beta)=0$, the definition of $W$ and the fact that $\rho$ is
supported in $\left[-\gamma,-\lambda\right] \cup\left[\lambda,
\gamma\right]$, we get
\[
\begin{aligned}
& \int_0^T\int_{-\beta}^{\beta}  \left(s\Theta (Z_x)^2 + s^3
\Theta^3
\left(\frac{x-x_0}{\tilde a} \right)^2Z^2\right)e^{2s\tilde\varphi}dxdt\\
& \le c
\int_0^T\int_{-\beta}^{\beta} \tilde h^{2}\frac{e^{2s\tilde\varphi}}{\tilde a}dxdt\\
&\le C \int_0^T \int_{-\gamma}^{-\lambda}(
W^2+ (W_x)^2)e^{2s \tilde \varphi}dxdt + C\int_0^T \int_{\lambda}^{\gamma}(W^2+ (W_x)^2)e^{2s\varphi}dxdt\\
&\mbox{(since $\tilde \psi (x)= \psi (-x)$, for $x <0$)}\\
&= 2C\int_0^T \int_{\lambda}^{\gamma}(W^2+ (W_x)^2)e^{2s\varphi}dxdt
= 2C\int_0^T \int_{\lambda}^{\gamma}(v^2+
(v_x)^2)e^{2s\varphi}dxdt\\
& \mbox{ (by Propositions \ref{caccio1}) }\\
& \le C \int_0^T \int_{\omega} \frac{v^2}{a}dxdt,
\end{aligned}
\]
for some positive constants $c$ and $C$ and $s$ large enough.

 Hence, by the definitions of $Z$, $W$ and $\rho$, and using the previous inequality one has
\begin{equation}\label{car101}
\begin{aligned}
&\int_0^T\int_{0}^{\lambda}  \left(s\Theta (v_x)^2 + s^3 \Theta^3
\left(\frac{x-x_0}{a} \right)^2v^2\right)e^{2s\varphi}dxdt\\
&= \int_0^T\int_{0}^{\lambda}  \left(s\Theta (W_x)^2 + s^3 \Theta^3
\left(\frac{x-x_0}{a} \right)^2W^2\right)e^{2s\varphi}dxdt\\
&= \int_0^T\int_{0}^{\lambda}  \left(s\Theta (Z_x)^2 + s^3 \Theta^3
\left(\frac{x-x_0}{a} \right)^2Z^2\right)e^{2s\varphi}dxdt\\
&\le \int_0^T\int_{0}^{\beta}  \left(s\Theta (Z_x)^2 + s^3 \Theta^3
\left(\frac{x-x_0}{a} \right)^2Z^2\right)e^{2s\varphi}dxdt\\
&\le \int_0^T\int_{-\beta}^{\beta}  \left(s\Theta (Z_x)^2 + s^3
\Theta^3
\left(\frac{x-x_0}{a} \right)^2Z^2\right)e^{2s\tilde \varphi}dxdt\\
&\le C \int_0^T \int_{\omega} \frac{v^2}{a}dxdt,
\end{aligned}
\end{equation}
for a positive constant $C$ and $s$ large enough. Therefore, by
\eqref{carin01} and \eqref{car101}, the conclusion follows.
\end{proof}

We are now ready to prove the observability inequality in the case
of a regular final--time datum:
\begin{Lemma}\label{obser.regular1}
Assume \eqref{omega} and Hypothesis $\ref{Ass031}$. Then there
exists a positive constant $C_T$ such that every solution $v \in
\cal W_2$ of \eqref{h=011} satisfies
 \[
\int_0^1v^2(0,x) \frac{1}{a}dx \le C_T\int_0^T
\int_{\omega}v^2\frac{1}{a}dxdt.
\]
\end{Lemma}
\begin{proof}
Multiplying the equation of \eqref{h=011} by $\displaystyle
\frac{v_t}{a}$ and integrating by parts over $(0,1)$, one has
\[
\begin{aligned}
&0 = \int_0^1(v_t+ av_{xx})\frac{v_t}{a} dx=
\int_0^1\frac{(v_t)^2}{a}dx + \left[v_xv_t \right]_{x=0}^{x=1}-
\int_0^1v_xv_{tx} dx \\&= \int_0^1\frac{(v_t)^2}{a}dx -
\frac{1}{2}\frac{d}{dt}\int_0^1(v_x)^2
 \ge - \frac{1}{2}
\frac{d}{dt}\int_0^1 (v_x)^2dx.
\end{aligned}
\]
Thus, the function $t \mapsto \int_0^1 (v_x)^2 dx$ is nondecreasing
for all $t \in [0,T]$. In particular, $\int_0^1 (v_x)^2(0,x)dx \le
\int_0^1(v_x)^2(t,x)dx$. Integrating the last inequality over
$\displaystyle \left[\frac{T}{4}, \frac{3T}{4} \right]$, we find
\[
\begin{aligned}
\int_0^1(v_x)^2(0,x) dx &\le
\frac{2}{T}\int_{\frac{T}{4}}^{\frac{3T}{4}}\int_0^1(v_x)^2dxdt\\&\le
C_T \int_{\frac{T}{4}}^{\frac{3T}{4}}\int_0^1\Theta
(v_x)^2e^{2s_0\varphi}dxdt.
\end{aligned}
\]
Hence, by Lemma \ref{lemma31}, there exists a positive constant $C$
such that
\[
\int_0^1 (v_x)^2(0,x) dx \le C \int_0^T
\int_{\omega}\frac{v^2}{a}dxdt.
\]
First, in the strongly degenerate case, by Lemma \ref{hardy}, there
exists a positive constant $C>0$ such that
\[
\int_0^1 v^2(t,x)\frac{1}{a}dx \le C \int_0^1 (v_x)^2(t,x) dx,
\]
for all $t \in [0, T]$. Thus, from the previous two inequalities, we
get
\[
\int_0^1v^2(0,x) \frac{1}{a}dx \le C\int_0^T
\int_{\omega}v^2\frac{1}{a}dxdt
\]
for a positive constant $C$, and the conclusion follows.

In the weakly degenerate case, proceeding as in the proof of Lemma
\ref{lemma21} and applying the Hardy--Poincar\'{e} inequality of
Proposition \ref{HP}, one has
\[
\begin{aligned}
\int_0^1 v^2(0,x)\frac{1}{a}dx &= \int_0^1 \frac{p(x)}{(x-x_0)^2}
v^2(0,x)dx  \\&\le C_{HP} \int_0^1 p(x)(v_x)^2(0,x) dx \\&\le
C_1C_{HP} \int_0^1(v_x)^2(0,x) dx \le C
\int_0^T\int_{\omega}\frac{v^2}{a}dxdt,
\end{aligned}
\]
for a positive constat $C$.  Here $p(x) = \displaystyle
\frac{(x-x_0)^2}{a(x)}$, $C_{HP}$ is the Hardy-Poincar\'{e} constant
and $C_1:=\displaystyle \max\left\{ \frac{x_0^2}{a(0)},
\frac{(1-x_0)^2}{a(1)}\right \}$. Observe that the function $p$
satisfies the assumptions of Proposition \ref{HP} (with $q= 2-K$)
thanks to Lemma \ref{rem}. Hence, also in this case, the conclusion
follows.
\end{proof}

Using Lemma \ref{obser.regular1} and proceeding as in the proof of
Proposition \ref{obser.}, one can prove Proposition \ref{obser.1}.

\chapter{Linear and Semilinear Extensions}\label{secsemilinear}
In this chapter we will extend the global null controllability
result proved in the previous chapter to the linear problem
\begin{equation}\label{linear_c}
\begin{cases}
u_t - \mathcal A u + c(t,x)u =h(t,x) \chi_{\omega}(x), & (t,x) \in (0,T) \times (0,1),\\
u(t,1)=u(t,0)=0, & t \in (0,T),\\
u(0,x)=u_0(x),& x \in (0,1),
\end{cases}
\end{equation}
where $u_0 \in X$,  $h \in L^2(0,T;X)$, $c \in L^\infty(Q_T)$,
$\omega$ is as in \eqref{omega1} or in \eqref{omega}.
We recall that $X$ is $L^2(0,1)$ in the divergence case and $L^2_{\frac{1}{a}}(0,1)$ in the non divergence one.
Concerning
$a$, we assume that it satisfies Hypothesis \ref{Ass02} or Hypothesis \ref{Ass021} in order to
prove the Carleman estimates in Corollary \ref{cor_c}, and Hypothesis
\ref{Ass03} or Hypothesis \ref{Ass031} to prove the observability inequalities in Propositions
\ref{6.1} and \ref{6.2}. Observe that the well-posedness of \eqref{linear_c}
follows by \cite[Theorems 4.1, 4.3]{fggr}. As for the previous case, the
global null controllability of \eqref{linear_c} follows in a
standard way from an observability inequality for the solution of
the associated adjoint problem
\begin{equation}\label{h=0_c}
\begin{cases}
v_t +\mathcal A v- cv= 0, &(t,x) \in  (0, T)\times (0, 1),\\
v(t,1)=v(t,0)=0, & t \in (0,T),\\
v(T)= v_T \in L^2(0,1).
\end{cases}
\end{equation}
To obtain an observability inequality like the one in Proposition
\ref{obser.} or in Proposition \ref{obser.1}, the following Carleman estimate, corollary of Theorem
\ref{Cor1} and Theorem \ref{Cor11}, is crucial. For this, consider the problem
      \begin{equation}\label{adjoint0_c}
      \begin{cases}
      v_t + \mathcal A v -cv =h, & (t,x) \in (0,T) \times (0,1),\\
      v(t,1)=v(t,0)=0, &  t \in (0,T)
          \end{cases}
\end{equation}
and denote with $\mathcal S$ the space $\mathcal S_1$ if we consider the divergence case and $\mathcal S_2$ if we consider the non divergence one.
\begin{Corollary}\label{cor_c}
Assume Hypothesis $\ref{Ass02}$ or Hypothesis $\ref{Ass021}$. Then, there exist two
positive constants $C$ and $s_0$, such that every solution $v$ in $
\mathcal{S}$ of \eqref{adjoint0_c} satisfies, for all $s \ge s_0$,
\[
\begin{aligned}
&\int_{Q_T} \left(s\Theta a(v_x)^2 + s^3 \Theta^3
\frac{(x-x_0)^2}{a} v^2\right)e^{2s\varphi}dxdt\\
&\le C\left(\int_{Q_T} h^{2}e^{2s\varphi}dxdt+
sc_1\int_0^T\left[a\Theta e^{2s \varphi}(x-x_0)(v_x)^2
dt\right]_{x=0}^{x=1}\right),
\end{aligned}
\]
if Hypothesis $\ref{Ass02}$ holds and
\[
\begin{aligned}
&\int_{Q_T} \left(s\Theta (v_x)^2 + s^3 \Theta^3
\left(\frac{x-x_0}{a} \right)^2v^2\right)e^{2s\varphi}dxdt\\
&\le C\left(\int_{Q_T} h^{2}\frac{e^{2s\varphi}}{a}dxdt +
sd_1\int_0^T\left[\Theta e^{2s \varphi}(x-x_0)(v_x)^2
dt\right]_{x=0}^{x=1}\right),
\end{aligned}
\] if Hypothesis $\ref{Ass021}$ is in force.
Here $c_{1}$ and $d_{1}$ are the constants introduced in \eqref{c_1} and \eqref{c_11}, respectively.
\end{Corollary}
\begin{proof}
     Rewrite the equation of
    \eqref{adjoint0_c} as $ v_t + (av_x)_x = \bar{h}, $ where $\bar{h}
    := h + cv$.
    Hence $\bar h^2 \le 2h^2 + 2 \|c\|_{L^\infty(Q_T)}v^2$. Now, we will distinguish between the divergence and the non divergence case.

    {\it Divergence case.} If Hypothesis \ref{Ass02} holds, then, applying Theorem \ref{Cor1}, there exists
    two positive constants $C$ and $s_0 >0$, such that
    \begin{equation}\label{fati1_c}
    \begin{aligned}
  &\int_{Q_T} \left(s\Theta a(v_x)^2 + s^3 \Theta^3
\frac{(x-x_0)^2}{a} v^2\right)e^{2s\varphi}dxdt\\
&\le C\left(\int_{Q_T} \bar{h}^{2}e^{2s\varphi}dxdt+
sc_1\int_0^T\left[a\Theta e^{2s \varphi}(x-x_0)(v_x)^2
dt\right]_{x=0}^{x=1}\right)
    \\
    &\le
  C\left(\int_{Q_T} h^2e^{2s\varphi}dxdt
   +\int_{Q_T}
    e^{2s\varphi}v^2dxdt+
sc_1\int_0^T\left[a\Theta e^{2s \varphi}(x-x_0)(v_x)^2
dt\right]_{x=0}^{x=1}\right)
    \end{aligned}
    \end{equation}
    for all $s \ge s_0$.
Applying the Hardy-Poincar\'e inequality (see Proposition \ref{HP})
to $w(t,x) := e^{s\varphi(t,x)} v(t,x)$ and proceeding as in
\eqref{sopra}, recalling that $0<\inf \Theta\leq \Theta\leq c
\Theta^2$, one has
    \[
    \begin{aligned}
    \int_0^1 e^{2s\varphi} v^2 dx &= \int_0^1
w^2 dx\le C \int_0^1 a (w_x)^2 dx + \frac{s}{2}\int_0^1
\frac{(x-x_0)^2}{a}w^2 dx\\&\leq C\Theta\int_0^1
ae^{2s\varphi}(v_x)^2 dx+ C\Theta^3s^2 \int_0^1
    e^{2s\varphi}v^2 \frac{(x-x_0)^2}{a}dx.
    \end{aligned}
    \]
    Using this last inequality in \eqref{fati1_c}, we have
    \begin{equation}\label{fati2_c}
    \begin{aligned}
  &\int_{Q_T} \left(s\Theta a(v_x)^2 + s^3 \Theta^3
\frac{(x-x_0)^2}{a} v^2\right)e^{2s\varphi}dxdt\\
   & \le  C\left(\int_{Q_T} h^2e^{2s\varphi}dxdt
   +\int_{Q_T} \Theta
ae^{2s\varphi}(v_x)^2 dxdt
   \right.\\& \left. +s^2\int_{Q_T} e^{2s\varphi}\Theta^3
 \frac{(x-x_0)^2}{a}v^2dxdt+
sc_1\int_0^T\left[a\Theta e^{2s \varphi}(x-x_0)(v_x)^2
dt\right]_{x=0}^{x=1}\right)
   \end{aligned}
    \end{equation}
    for a positive constant $C$.
Hence, for all $s \ge s_0$, where $s_0$ is assumed sufficiently
large, the claim follows.

{\it Non divergence case.} If Hypothesis \ref{Ass021} holds, then, applying Theorem \ref{Cor11}, there exist two positive constants $C$ and $s_0$ such that
\begin{equation}\label{fati11_c}
\begin{aligned}
&\int_{Q_T} \left(s\Theta (v_x)^2 + s^3 \Theta^3
\left(\frac{x-x_0}{a} \right)^2v^2\right)e^{2s\varphi}dxdt\\
&\le C\left(\int_{Q_T}\bar h^{2}\frac{e^{2s\varphi}}{a}dxdt +
sd_1\int_0^T\left[\Theta e^{2s \varphi}(x-x_0)(v_x)^2
dt\right]_{x=0}^{x=1}\right)\\
    &\le
  C\left(\int_{Q_T} h^2\frac{e^{2s\varphi}}{a}dxdt
   +\int_{Q_T}
    v^2\frac{e^{2s\varphi}}{a}dxdt+
sd_1\int_0^T\left[\Theta e^{2s \varphi}(x-x_0)(v_x)^2
dt\right]_{x=0}^{x=1}\right)
\end{aligned}
\end{equation}
for all $s \ge s_0$. Applying again the Hardy-Poincar\'e inequality
to $w:= e^{s\varphi} v$, setting $p(x) = \displaystyle
\frac{(x-x_0)^2}{a(x)}$, using Lemma \ref{rem} and proceeding as in Lemma \ref{obser.regular1}, one has
\[
\begin{aligned}
\int_0^1 w^2\frac{1}{a}dx &= \int_0^1 \frac{p(x)}{(x-x_0)^2}
w^2dx  \le C_{HP} \int_0^1 p(x)(w_x)^2 dx \le
C \int_0^1(w_x)^2 dx \\&\le
  C\left(\int_{Q_T} \Theta
e^{2s\varphi}(v_x)^2 dxdt
+s^2\int_{Q_T} \Theta^3e^{2s\varphi}
 \left(\frac{x-x_0}{a}\right)^2v^2dxdt\right),
\end{aligned}
\]
for a positive constat $C$ (we recall that $ C_{HP}$ is the Hardy-Poincar\'{e} constant). Finally, using the previous inequality in \eqref{fati11_c}, one has
\[
\begin{aligned}
&\int_{Q_T} \left(s\Theta (v_x)^2 + s^3 \Theta^3
\left(\frac{x-x_0}{a} \right)^2v^2\right)e^{2s\varphi}dxdt\\
&\le
  C\left(\int_{Q_T} h^2\frac{e^{2s\varphi}}{a}dxdt
   +\int_{Q_T} \Theta
e^{2s\varphi}(v_x)^2 dxdt \right. \\
&\left.
+s^2\int_{Q_T} \Theta^3e^{2s\varphi}
 \left(\frac{x-x_0}{a}\right)^2v^2dxdt+
sd_1\int_0^T\left[\Theta e^{2s \varphi}(x-x_0)(v_x)^2
dt\right]_{x=0}^{x=1}\right)
\end{aligned}
\]
As for the divergence case, choosing $s_0$ sufficiently
large, the claim follows.

\end{proof}

As a consequence of the previous corollary, one can deduce an
observability inequality for the homogeneous adjoint problem
\eqref{h=0_c}. In fact, without loss of generality we can assume
that $c \ge 0$ (otherwise one can reduce the problem to this case
introducing $\tilde v:= e^{-\lambda t}v$ for a suitable $\lambda$).
Using this assumption we can prove that the analogous of Lemma
\ref{lemma3}, of Lemma \ref{lemma3'} and of Lemma \ref{lemma31} still hold true. Thus,
as before, one can prove the following observability inequalities:
\begin{Proposition}\label{6.1}
Assume Hypotheses $\ref{ipotesiomega}$ and $\ref{Ass03}$. Then there
exists a positive constant $C_T$ such that every solution  $v \in
C([0, T]; L^2(0,1)) \cap L^2 (0,T; {\cal H}^1_a(0,1))$  of
\eqref{h=0_c} satisfies
 \begin{equation}\label{obser_c}
\int_0^1v^2(0,x) dx \le C_T\int_0^T \int_{\omega}v^2(t,x)dxdt .
\end{equation}
\end{Proposition}
\begin{Proposition}\label{6.2}
Assume Hypothesis $\ref{Ass031}$ and \eqref{omega}. Then there
exists a positive constant $C_T$ such that the solution $v \in
 C([0, T]; L^2_{\frac{1}{a}}(0,1)) \cap L^2 (0,T;
\cH^1_{\frac{1}{a}}(0,1))$ of \eqref{h=0_c} satisfies
\begin{equation}\label{obser_c1}
\int_0^1v^2(0,x)\frac{1}{a} dx \le C_T\int_0^T
\int_{\omega}v^2\frac{1}{a}dxdt.
\end{equation}
\end{Proposition}
Using \eqref{obser_c} and \eqref{obser_c1} one can prove that the analogous of Theorems
\ref{th3}, \ref{th31} still hold for \eqref{linear_c}. We underline the fact
that also Corollaries \ref{corollario}, \ref{dueparti} and Theorems \ref{th3'}, \ref{th2pezzi} still
hold for \eqref{h=0_c} and \eqref{linear_c}, respectively.
\bigskip

Finally, the controllability result can be extended to a semilinear
degenerate parabolic equation of the type
\begin{equation}
\label{nl}
\begin{cases}
u_t - \mathcal Au  + f(t,x,u)  =h(t,x) \chi_{\omega}(x), & \ (t,x) \in (0,T) \times (0,1), \\
u(t,1)=u(t,0)=0, & \  t \in (0,T),\\
u(0,x)=u_0(x) , & \  x \in (0,1),
\end{cases}
\end{equation}
where $u_0 \in X$,  $h \in L^2(0,T;X)$ and $\omega$ satisfies
Hypothesis \ref{ipotesiomega}. However, we can treat this semilinear
problem only in the (WD) case, since we need to use the following
compactness result:
\begin{Theorem}[\cite{fggr}, Theorems 5.4, 5.5]\label{compact_generale}
Suppose that $a\in  C[0,1]$ and $a^{-1}\in L^1(0,1)$. Then the spaces
\[
H^1(0, T;L^2(0,1)) \cap L^2(0, T; {\cal H}^2_a(0,1))
\; \text{
 and }\;
 H^1(0, T;L^2_{\frac{1}{a}}(0,1)) \cap L^2(0, T;
H^2_{\frac{1}{a}}(0,1))\]
are compactly
imbedded in
\[L^2(0, T; {\cal H}^1_a(0,1))\cap C(0,T; L^2(0,1))\; \text{
 and }\; L^2(0, T;
H^1_{\frac{1}{a}}(0,1)) \cap C(0,T; L^2_{\frac{1}{a}}(0,1)),\] respectively.
\end{Theorem}

Also in this case, we assume that $a$ satisfies the degeneracy conditions stated in Hypothesis $\ref{Ass0}$. Moreover, in order to prove observability inequalities analogous to those proved in Chapter \ref{secobserv}, we assume that
the conditions assumed therein are satisfied. More precisely, we suppose that:
\[
\begin{aligned}
&\mbox{Hypothesis \ref{Ass03} holds in the divergence case}\\
& \mbox{Hypothesis \ref{Ass031} holds in the non divergence case.}
\end{aligned}
\]

Concerning $f:Q_T\times \mathbb R \to \mathbb R$ we make the
following assumptions:
\begin{itemize}
\item
$f$ is a Carath\'eodory function, i.e.
\[
\begin{aligned}
&\mbox{the map $(t,x) \mapsto f(t,x,q)$ is
measurable for all $q \in \mathbb R$ and}\\
&\mbox{the map $q\mapsto f(t,x,q)$ is continuous for a.e. $(t,x)\in Q_T$};
\end{aligned}
\]
\item
$f(t, x, 0)=0$ for a.e. $(t,x)  \in Q_T$;
\item $f_q(t,x,q)$ exists for a.e. $(t,x)\in Q_T$;
\item $f_q$ is a Carath\'{e}odory function;
\item there exists $C >0$ such that
\[  |f_q(t,x,q)| \le C
\]
for a.e. $(t,x) \in Q_T$ and for every
$q\in\mathbb{R} $.
\end{itemize}

To prove a null controllability result for \eqref{nl} one can use,
as in \cite{acf} or in \cite{cf}, a fixed point method, considering
a suitable sequence of linear problem associated to the semilinear
one, apply a related approximate null controllability property for
the linear case, and pass to the limit. We omit the details, which
are now standard, just stating the following
\begin{Theorem}
Under the assumptions above, problem \eqref{nl} is globally null
controllable.
\end{Theorem}

\chapter{Final Comments}\label{sec7}

{\bf Comment $1$.} If $\omega=\omega_1 \cup \omega_2$, $\omega_i$
intervals, with $\omega_1\subset\subset (0,x_0)$,
$\omega_2\subset\subset (x_0,1)$, and $x_0 \not \in
\overline{\omega}$, the global null controllability for
\eqref{linear} follows by \cite[Theorem 4.1]{acf} when $\mathcal A = \mathcal A_1$ and by \cite[Theorem 4.5]{cfr} when $\mathcal A = \mathcal A_2$, at least in the
strongly degenerate case and if the initial datum is more regular.
Indeed, in this case, given $u_0\in {\cal H}^1_a(0,1)$ or $u_0\in
\cH^1_{\frac{1}{a}}(0,1)$, $u$ is a
solution of \eqref{linear} if and only if the restrictions of $u$ to
$[0, x_0)$ and to $(x_0,1]$, $u_{|_{[0,x_0)}}$ and
$u_{|_{(x_0,1]}}$, are solutions to
\begin{equation}\label{me1}
\begin{cases}
u_t - \mathcal Au   =h(t,x) \chi_{\omega_1}(x), & \ (t,x) \in (0,T) \times (0,x_0), \\
    u(t,0)=0, & \  t \in (0,T),\\
    \begin{cases}(au_x)(t,x_0)=0, & \text{in the divergence case} \\
    u(t,x_0)=0, &\text{in the non divergence case}, \end{cases} & \ t \in (0, T),\\
    u(0,x)=u_0(x)_{|_{[0,x_0)}},
\end{cases}
\end{equation}
and
\begin{equation}\label{me2}
\begin{cases}
u_t -  \mathcal Au   =h(t,x) \chi_{\omega_2}(x), & \ (t,x) \in (0,T) \times (x_0,1), \\
u(t,1)=0, & \  t \in (0,T),\\
  \begin{cases}(au_x)(t,x_0)=0, & \text{in the divergence case} \\
    u(t,x_0)=0, &\text{in the non divergence case}, \end{cases}&\  t \in (0, T),\\
u(0,x)=u_0(x)_{|_{(x_0,1]}},
\end{cases}
\end{equation}
respectively. This fact is implied by the characterization of the
domains of $\mathcal A_1$ and $\mathcal A_2$ given in Propositions \ref{domain}, \ref{domain1} and by
the Regularity Theorems \ref{th-parabolic}, \ref{theorem_nondivergence} when the initial datum is
more regular. On the other hand if $u_0$ is only of class
$L^2(0,1)$ or $L^2_{\frac{1}{a}}(0,1)$, the solution is not sufficiently regular to verify the
additional condition at $(t,x_0)$ and this procedure cannot be
pursued. 

Moreover, in the weakly degenerate case, the lack of
characterization of the domains of $\mathcal A_1$ and $\mathcal A_2$ doesn't let us
consider a decomposition of the system in two disjoint systems like
\eqref{me1} and \eqref{me2}, in order to apply the results of
\cite{acf} and \cite{cfr}, not even in the case of a regular initial datum.

For this reason, using observability inequalities and Carleman
estimates, in Chapter \ref{secobserv} we have proved a null controllability
result both in the (WD) and (SD) cases, also in the case of a
control region of the form $\omega=\omega_1\cup \omega_2$ as above.

\noindent{\bf Comment $2$.} It is well known that observability
inequalities for the adjoint homogeneous problem imply the validity
of null controllability results for the original parabolic problem.
In fact, as a corollary of the observability inequalities, we give
the associated null controllability results for \eqref{linear}
providing an estimate of the control function $h$, see Theorems
\ref{th3}, \ref{th3'}, \ref{th31} and \ref{th2pezzi}.

Of course, null controllability results can be obtained also in
other ways, but the approach with observability inequalities is very
general and permits to cover all possible situations. For example,
if $x_0\in \omega$, one could think to obtain the null
controllability result directly by a localization argument based on
cut-off functions, as in the non degenerate case. But we now show
that, at least in the divergence case, this is {\em not always
possible} in presence of a weakly degenerate $a$, but only in the
(SD) case. Indeed, assume that the degenerate point $x_0$ belongs to
the control region $\omega$, consider $0 < r' < r$ with $(x_0 -
r,x_0 +r) \subset \omega$, the cut-off functions $\phi_i \in
C^\infty ([0,1])$, $i=0,1,2,$ defined as
\[
\phi_1(x) :=
\begin{cases}
0, & x \; \in [x_0 -r', 1], \\
1, & x \; \in [0, x_0  - r],
\end{cases}
\quad \phi_2(x) :=
\begin{cases}
0, & x \; \in [0,x_0+r'], \\
1, & x \; \in [x_0+r,1],
\end{cases}
\]
and $\phi_0= 1-\phi_1-\phi_2$. Then, given an initial condition $u_0
\in L^2(0,1)$, by classical controllability results in the
nondegenerate case, there exist two control functions $h_1 \in
L^2((0,T) \times \big[\omega \cap (0,x_0 - r'))\big]$ and $h_2\in
L^2((0,T) \times \big[\omega \cap(x_0 + r',1)\big])$, such that the
corresponding solutions $v_1$ and $v_2$ of the parabolic problems
analogous to \eqref{linear} in the domains $(0,T) \times (0,x_0 -
r')$ and $(0,T) \times (x_0 + r',1)$, respectively, satisfy
$v_1(T,x) = 0$ for all $x \in (0,x_0 - r')$ and $v_2(T,x) = 0$ for
all $x \in (x_0 + r',1)$ with
\[
\int_0^T\int_0^{x_0-r'}h_1^2dxdt\leq C\int_0^T\int_0^{x_0-r'}
u_0^2dxdt
\]
and
\[
\int_0^T\int_{x_0+r'}^1h_2^2dxdt\leq C\int_0^T\int_{x_0+r'}^1
u_0^2dxdt
\]
for some constant $C$.

Now, let $v_0$ be the solution of the analogous of problem
\eqref{linear} in divergence form in the domain
$(0,T)\times(x_0-r,x_0 +r)$ without control, and with the same
initial condition $u_0$. Finally, define the function
\begin{equation}\label{falsau}
u(t,x) = \phi_1(x)v_1(t,x)+\phi_2(x)v_2(t,x) + \frac{T- t}{ T}
\phi_0(x)v_0(t,x).
\end{equation}
Then, $u(T,x) = 0$ for all $x\in (0,1)$ and $u$ satisfies problem
\eqref{linear} in the domain $Q_T$ with
\[
\begin{aligned}
h&= \phi_1h_1\chi_{\omega \cap (0, x_0-r')} + \phi_2 h_2\chi_{\omega
\cap (x_0+r',1)} -\frac{1}{T}\phi_0 v_0 -\phi_1'av_{1,x}
-\phi_2'av_{2,x} \\& -\phi_0'\frac{T- t}{ T}av_{0,x}
-\left(\phi_1'av_1+\phi_2'av_2+ \phi_0'\frac{T- t}{ T}av_0\right)_x.
\end{aligned}
\]
We strongly remark that in the (WD) case this function {\em is not}
in $L^2\big((0,T)\times \omega\big)$ since the degenerate function
in this case is only $W^{1,1}(0,1)$, so that the problem fails to be
controllable in the Hilbert space $L^2$. For this reason, we think
that our approach via Carleman estimates can be extremely
interesting also to prove null controllability results, which could
not be obtained in other ways.

On the other hand, using the previous technique, one can prove that \eqref{linear} in {\em non divergence} form is global null controllable if $x_0$ belongs to the control region $\omega$. Being the observability inequality equivalent to the null controllability, it is superfluous to obtain the first inequality as a consequence of Carleman estimate. For this reason we have proved Proposition \ref{obser.1} only when $x_0$ does not belong to the control region $\omega$.
\medskip

\noindent{\bf Comment $3$.} Finally, let us conclude with a remark
on the fact that in the definition of degeneracy (both weak and
strong) we admit only that $K \in (0,2)$. This technical assumption,
which is essential to prove Lemmas \ref{lemma2} and \ref{lemma21}, was already
introduced, for example, in \cite{cfr} or in \cite{cfr1}, with the
following motivation: if $K\ge 2$ and the degeneracy occurs at the
boundary of the domain, the problem fails to be null controllable on
the whole interval $[0,1]$  (see, e.g., \cite{cfr}, \cite{cfv}),
and in this case the only property that can be expected is the {\em
regional null} controllability (see also \cite{cf},
\cite{cfv1} and \cite{fm}). Let
us briefly show that the same phenomenon appears, for example, in the non divergence case,
inspired by \cite[Remark 4.6]{cfr}. Indeed, let us introduce the
following variant of a classical change of variables:
\[
X=\begin{cases} \displaystyle x_0-\int_x^{x_0}\frac{1}{y^{K/2}}dy & \mbox{ if }0< x\leq x_0,\\
\displaystyle x_0+\int_{x_0}^x \frac{1}{\sqrt{a(1+x_0-y)}}dy &\mbox{
if }x_0<x<1,
\end{cases}
\]
so that $(0,1)$ is stretched to $(-\infty,\infty)$, and
$U(t,X)=a^{-1/4}(x)u(t,x)$, where $u$ solves \eqref{linear}. Now,
take the reference function $a(x)=|x-x_0|^K$ with $K>2$, so that we
find that $U$ solves a nondegenerate heat equation of the form
$U_t-U_{xx}+b(X)U=\tilde h\chi_{\tilde\omega}$, where $\tilde \omega$
is a bounded domain compactly contained both in $(-\infty,0)$ and in
$(0,\infty)$. Adapting a result of \cite{mizu}, the new equation is
not controllable, see \cite[Remark 4.6]{cfr}.

In particular,
proceeding as in \cite{cfv}, \cite{cfv1}, one can prove  that if
$a\in W^{1, \infty}(0,1)$, $a^{-1}\not \in L^1(0,1)$ and the control
set $\omega$ is an interval $\omega=(\alpha,\beta)$ lying on one
side of $x_0$, for every $\lambda, \gamma \in (0,1)$ such that
\[
\begin{aligned}
 &0\leq \alpha <\lambda< \beta < x_0<1 \quad(\mbox{if } x_0 > \beta)\\
 \mbox{ or }\; &0 <x_0 < \alpha< \gamma <\beta \leq 1 \quad (\mbox{if }
x_0 < \alpha),
\end{aligned}
\]
there exists $h \in L^2(0,T; H)$ so that the solution $u$ of
\eqref{linear} (or \eqref{linear_c}) satisfies
\begin{equation}\label{PBM-1}
u(T,x)= 0 \ \text{ for every } x \in [0,\lambda] \; (\mbox{if } x_0
> \beta) \mbox{ or for every }x\in [\gamma,1] \;(\mbox{if } x_0 < \alpha).
\end{equation}
Moreover, there exists a positive constant $C_T$ such that
\[
\int_{Q_T} h^2dxdt \le C_T \int_0^1u_0^2 dx,
\]
in the divergence case and
\[
\int_{Q_T} h^2\frac{1}{a}dxdt \le C_T \int_0^1u_0^2\frac{1}{a}
dx,
\]
in the non divergence one.

As pointed out in \cite{cfv}, \cite{cfv1}, we note that the global
null controllability is a stronger property than \eqref{PBM-1}, in
the sense that the former is automatically preserved with time. More
precisely, if $u(T,x) = 0$ for all $x \in [0,1]$ and if we stop
controlling the system at time $T$, then for all $t\geq T$, $u(t,x)
= 0$ for all $x \in [0,1]$. On the contrary, regional null
controllability is a weaker property: in general, \eqref{PBM-1} is
no more preserved with time if we stop controlling at time $T$.
Thus, it is important to improve the previous result, as shown in
\cite{cfv} or in \cite{cfv1}, proving that the solution can be
forced to vanish identically on $[0,\lambda] \; (\mbox{if } x_0
> \beta) \mbox{ or in } [\gamma,1] \;(\mbox{if } x_0 < \alpha)$ during a given time interval $(T,T')$, i.e. that the
solution is {\em persistent regional null controllable}.

These results can be extended also in our situation, i.e. with an
interior degeneracy, for the problem
\begin{equation}\label{linear_b}
\begin{cases}
u_t - \mathcal A u + c(t,x)u +b(t,x)u_x=h(t,x) \chi_{\omega}(x), & (t,x) \in Q_T,\\
u(t,1)=u(t,0)=0,& t \in (0,T), \\
u(0,x)=u_0(x),& x\in (0,1),
\end{cases}
\end{equation}
where $u_0$,  $h$, $\omega$, $a$ and $c$ are as before, while $b \in
L^\infty(Q_T)$ and $|b(t,x)| \le C\sqrt{a(x)}$ for a positive
constant $C$. Observe that, in this case, the well-posedness of
\eqref{linear_b} follows by \cite[Theorems 4.1 and 4.2]{fggr}, and the
persistent regional null controllability follows by using cut-off
functions, adapting the technique developed in \cite{fv} or in
\cite{cfv1}.

\appendix
\chapter{Rigorous derivation of Lemma \ref{LEMMA1ND1}}\label{secA}

Here we show that all integrations by parts used in the proof of
Lemma \ref{LEMMA1ND1} are well justified, both in the non degenerate
and in the degenerate case. We will not prove all integration by
parts, which can be treated in similar ways, so we just consider the
(probably more involved) term
\[
2\int_{Q_T}a^2(\varphi_x)^3 w w_x dxdt =  \int_0^T
\int_0^1a^2(\varphi_x)^3 (w^2)_x dxdt,
\]
where, we recall $w \in \mathcal V_1$, and $\vp$ stands for $\Phi$ in
the non degenerate case. First, let us note that such an integral is
always well defined. Indeed, in the non degenerate case,
$w(\varphi_x)^3=w\Theta^3 (\psi_x)^3\in L^\infty(Q_T)$, since $w=
e^{s\Theta \psi}v$ with $v\in \mathcal V_1\subset L^\infty(Q_T)$,
while $\int_A^B a(w_x)^2\in L^\infty(0,T)$ and $a^{3/2}\in
L^\infty(A,B)$.

On the other hand, in the degenerate cases we have that
\[
a^2(\varphi_x)^3 w
w_x=\left(\frac{(x-x_0)^2}{a}\right)^{3/2}\Theta^3 w\sqrt{a} w_x.
\]
First, by Lemma \ref{rem}.1, we have that the map $x\mapsto
\frac{(x-x_0)^2}{a}$ is bounded. Then, as for the (WD) case, by
definition of ${\mathcal S}_1$, we immediately find that
$\sqrt{a}w_x\in L^2(Q_T)$ and $\sqrt{a}(x-x_0)^3,\Theta^3w^3\in
L^\infty(Q_T)$, so that the integral is well defined. Finally, in
the (SD) case we have that $\sqrt{a}w_x, \Theta^3w\in L^2(Q_T)$ and
$\sqrt{a}(x-x_0)^3\in L^\infty(Q_T)$.

Since the considerations in the degenerate case are more general,
from now on we shall confine to this case; hence we shall prove the
version of Lemma \ref{LEMMA1ND1} just for the degenerate case.

Thus, for any sufficiently small $\delta >0$, we get
\begin{equation}\label{integration}
\begin{aligned}
&\int_{Q_T} a^2(\vp_x)^3 (w^2)_x dx dt=
\int_0^T\int_0^{x_0-\delta}a^2(\vp_x)^3 (w^2)_x dx dt\\
& + \int_0^T\int_{x_0-\delta}^{x_0+ \delta}a^2(\vp_x)^3 (w^2)_x dx
dt+\int_0^T \int_{x_0+ \delta}^1a^2(\vp_x)^3 (w^2)_x dx dt\\
& =\int_0^T[(a^2(\vp_x)^3 w^2)(x_0- \delta) - (a^2(\vp_x)^3
w^2)(0)]dt \\
&-\int_0^T\int_0^{x_0-\delta}(a^2(\vp_x)^3)_x w^2dx dt +
\int_0^T\int_{x_0- \delta}^{x_0+\delta}a^2(\vp_x)^3 (w^2)_x dx dt  \\
&+ \int_0^T[(a^2(\vp_x)^3 w^2)(1) -
(a^2(\vp_x)^3 w^2)(x_0+ \delta)]dt  \\
&-\int_0^T\int_{x_0+ \delta}^1 (a^2(\vp_x)^3 )_xw^2dx dt\\
&= \int_0^T(a^2(\vp_x)^3 w^2)(x_0- \delta)dt - \int_0^T(a^2(\vp_x)^3
w^2)(x_0+ \delta)dt\\
& -\int_0^T\int_0^{x_0-\delta}(a^2(\vp_x)^3 )_xw^2dx dt + \int_0^T\int_{x_0- \delta}^{x_0+\delta}a^2(\vp_x)^3 (w^2)_x dx  dt \\
&  - \int_0^T\int_{x_0+ \delta}^1 (a^2(\vp_x)^3 )_xw^2 dx dt,
\end{aligned}
\end{equation}
since the functions $a^2(\cdot)(\vp_x)^3(t, \cdot)$, $w^2(t, \cdot)$
belong to $H^1(0,x_0-\delta)\cap H^1(x_0+\delta,1)$ for a.e. $t \in
(0,T)$ and $w(t,0)=w(t,1)=0$. Now, we prove that
\[
\begin{aligned}
&\lim_{\delta \rightarrow 0} \int_0^T\int_0^{x_0- \delta}
(a^2(\vp_x)^3
)_xw^2  dx = \int_0^T\int_0^{x_0} (a^2(\vp_x)^3 )_xw^2  dx, \\
&\lim_{\delta \rightarrow 0} \int_0^T\int_{x_0+ \delta}^1
(a^2(\vp_x)^3 )_xw^2  dx = \int_0^T\int_{x_0}^1 (a^2(\vp_x)^3 )_xw^2
dx
\end{aligned}
\]
and
\begin{equation}\label{delta2}
\lim_{\delta \rightarrow 0} \int_0^T\int_{x_0- \delta}^{x_0+ \delta}
a^2(\vp_x)^3 (w^2)_x dx =0.
\end{equation}
Toward this end, observe that
\begin{equation}\label{delta1'}
\int_0^T\int_0^{x_0- \delta} (a^2(\vp_x)^3 )_xw^2 dx =
\int_0^T\int_0^{x_0}(a^2(\vp_x)^3 )_xw^2 dx - \int_0^T\int_{x_0-
\delta}^{x_0} (a^2(\vp_x)^3 )_xw^2 dx\end{equation} and
\begin{equation}\label{delta1''}
\int_0^T\int_{x_0+ \delta}^1(a^2(\vp_x)^3 )_xw^2 dx=
\int_0^T\int_{x_0}^1 (a^2(\vp_x)^3 )_xw^2 dx -
\int_0^T\int_{x_0}^{x_0+ \delta} (a^2(\vp_x)^3 )_xw^2 dx.
\end{equation}
We notice that the identities above are justified by the fact that
$(a^2(\vp_x)^3 )_xw^2\in L^1(Q_T)$. Indeed, by Lemma \ref{rem}.1
applied with $\gamma=2$, we immediately have
\[
\begin{aligned}
|(a^2(\vp_x)^3 )_xw^2| &= |c_1^3 \Theta^3 \left(
\frac{(x-x_0)^3}{a}\right)_x w^2 |\\
&= \left|c_1^3 \Theta^3   \frac{3(x-x_0)^2}{a} w^2 -
 c_1^3 \Theta^3  \frac{a'(x) (x-x_0)^3}{a^2} w^2\right|\\
& \leq 3c_1^3 \max\left\{\frac{x_0^2}{a(0)}, \frac{(1-x_0)^2}{a(1)}
\right\} \Theta^3w^2 +c_1^3\Theta^3 \left|
\frac{(x-x_0)^3a'}{a^2}\right|w^2.
 \end{aligned}
\]
Now, in the (WD) case we have that
$\Theta^3w^2=\Theta^3e^{2s\vp}v^2\in L^\infty(Q_T)$, while
\[
\Theta^3  \left|\frac{(x-x_0)^3a'}{a^2} \right|w^2=\Theta^3
|a'|\left(\frac{|x-x_0|^{3/2}}{a}\right)^{2}w^2\leq c \Theta^3
|a'|w^2
\]
by Lemma \ref{rem}.1 applied with $\gamma=3/2>K$, and where $c$ is a
positive constant. At this point, $a'\in L^1(Q_T)$, while
$\Theta^3w^2\in L^\infty(Q_T)$, and the claim follows. In the (SD)
case for $K\leq 3/2$, from the previous inequality we have that
$\Theta^3w^2=\Theta^3e^{2s\vp}v^2\in L^2(Q_T)$, while $a'\in
L^\infty(Q_T)$, and again this is enough. For the (SD) case when
$K\in (3/2,2)$, we observe that
\[
\Theta^3  \left|\frac{(x-x_0)^3a'}{a^2} \right|w^2=\Theta^3
e^{2s\varphi}\left(\frac{|x-x_0|^\vartheta}{a}\right)^2|x-x_0|^{3-2\vartheta}|a'|v^2.
\]
By the last requirement in condition \eqref{dainfinito} and from
\eqref{Sigma}, since $\Theta^3 e^{2s\varphi}\in L^\infty(Q_T)$ and
$v^2\in L^1(Q_T)$, also this case is finished (recall Remark
\ref{remark5}), and \eqref{delta1'} and \eqref{delta1''} are
justified.

Thus, for any $\epsilon> 0$, by the absolute continuity of the
integral, there exists $\delta:= \delta(\epsilon) >0 $ such that
\[
\left|\int_0^T \int_{x_0- \delta}^{x_0} (a^2(\vp_x)^3 )_xw^2  dxdt
\right| < \epsilon,
\]
\[
\left|\int_0^T\int_{x_0- \delta}^{x_0+
\delta}a^2(\vp_x)^3(w^2)_xdxdt\right| < \epsilon,
\]
\[
\left|\int_0^T\int_{x_0}^{x_0+ \delta} (a^2(\vp_x)^3 )_xw^2  dxdt
\right| < \epsilon.
\]
Now, take such a $\delta$ in \eqref{integration}. Thus, $\epsilon$
being arbitrary,
\[
\begin{aligned}
&\lim_{\delta \rightarrow 0} \int_0^T\int_{x_0-
\delta}^{x_0}(a^2(\vp_x)^3 )_xw^2 dxdt = \lim_{\delta \rightarrow 0}
\int_0^T\int_{x_0- \delta}^{x_0+ \delta} a^2(\vp_x)^3(w^2)_x dxdt\\
& = \lim_{\delta \rightarrow 0}\int_0^T\int_{x_0}^{x_0+
\delta}(a^2(\vp_x)^3 )_xw^2 dxdt =0.
\end{aligned}
\]
The previous limits, \eqref{delta1'}, \eqref{delta1''}, together
with the integrability conditions proved above, imply
\[
\lim_{\delta \rightarrow 0} \int_0^T\int_0^{x_0-
\delta}(a^2(\vp_x)^3 )_xw^2 dxdt = \int_0^T\int_0^{x_0}(a^2(\vp_x)^3
)_xw^2  dxdt
\]
and
\[
\lim_{\delta \rightarrow 0} \int_0^T\int_{x_0+ \delta}^1au'v' dx dt=
\int_0^T\int_{x_0}^1(a^2(\vp_x)^3 )_xw^2 dxdt.
\]

In order to conclude the proof of the desired result, it is
sufficient to prove that
\begin{equation}\label{delta3}
\lim_{\delta \rightarrow 0} \int_0^T(a^2(\varphi_x)^3 w^2)(x_0-
\delta) dt=\lim_{\delta \rightarrow 0} \int_0^T(a^2(\varphi_x)^3
w^2)(x_0+ \delta)dt,
\end{equation}
and in particular they are 0, as it follows from the identity
\[
a^2(\varphi_x)^3 w^2=c_1^3 \Theta^3 e^{2s\vp}(x-x_0)^3a v^2.
\]
Indeed, in the (WD) case \eqref{delta3} follows from the fact that
$v$ is absolutely continuous in $Q_T$ and from Lebesgue's Theorem,
since the map $(t,x)\mapsto \Theta^3(t) e^{2s\vp(t,x)}$ is bounded;
in the (SD) case we use the characterization of the domain given by
Proposition \ref{characterization}.

The other integrations by parts in Lemma \ref{LEMMA1ND1} are easier
and can be proved proceeding as above.

\section*{Acknowledgement}
The authors are member of the INDAM  ``Gruppo Nazionale per l'Analisi Matematica, la Probabilit\`a e le loro Applicazioni'' and are supported by the GNAMPA Project ``Systems with irregular operators''. D. M. is also member of the MIUR project ``Variational and perturbative aspects in nonlinear differential problems''.

\end{document}